\documentclass[11pt]{article}

\usepackage[T1]{fontenc}
\usepackage[utf8]{inputenc}
\usepackage{lmodern}
\usepackage{microtype}
\usepackage[a4paper,margin=30mm]{geometry}
\usepackage{amsmath,amssymb,amsfonts,amsthm,mathtools,bm,mathrsfs}
\usepackage{booktabs}
\usepackage{graphicx}
\usepackage{longtable}
\usepackage{enumitem}
\usepackage{xcolor}
\usepackage{csquotes}
\usepackage{hyperref}
\usepackage[nameinlink,noabbrev]{cleveref}

\hypersetup{
  colorlinks=true,
  linkcolor=black,
  citecolor=black,
  urlcolor=black,
  pdftitle={The Three Gates: A Rooted-Operator Approach to Weil Positivity},
  pdfauthor={Marco Desogus}
}

\newtheorem{theorem}{Theorem}[section]
\newtheorem{lemma}[theorem]{Lemma}
\newtheorem{proposition}[theorem]{Proposition}
\newtheorem{corollary}[theorem]{Corollary}

\theoremstyle{definition}
\newtheorem{definition}[theorem]{Definition}
\newtheorem{remark}[theorem]{Remark}
\newtheorem{computercertificate}[theorem]{Certified finite computation}
\newtheorem{analyticcertificate}[theorem]{Analytic certificate}
\newtheorem{audit}[theorem]{Audit statement}

\newcommand{\C}{\mathbb C}

\newcommand{\cH}{\mathcal H}
\newcommand{\cQ}{\mathcal Q}

\newcommand{\cW}{\mathcal W}

\newcommand{\Rea}{\operatorname{Re}}
\newcommand{\Short}{\operatorname{Short}}

\newcommand{\ketbra}[2]{|#1\rangle\langle #2|}
\newcommand{\one}{\mathbf 1}

\newcommand{\LambdaV}{\Lambda}
\newcommand{\psiC}{\psi}

\newenvironment{quest}{\begin{quote}\small\itshape}{\end{quote}}
\newenvironment{interlude}{\begin{quote}\small\itshape}{\end{quote}}

\title{\textbf{The Three Gates}\\[3pt]
\large A Rooted-Operator Approach to Weil Positivity}
\author{Marco Desogus}
\date{6 October 2026}

\begin{document}
\maketitle

\begin{abstract}
We present a localized rooted-operator argument for Weil positivity in the real
odd logarithmic channel, retaining the polar rank--one term throughout.  The
three gates are structural levels of the argument rather than an assertion
that every result in one gate is a premise of the next.  Gate~I establishes an
independent strict cellular/shell positivity theorem and structural diagnostic,
with rigorous interval-arithmetic bounds.  Gate~II carries the load-bearing
operator argument, combining the Mellin unit-cell decomposition, divisor
partial-isometry squares, affine translations, full-form Cauchy--Carleman
transport, coherent multi-source shorting, the augmented scalar root, and
Schur geometry in the true metric.  Gate~III uses compression, closure and the
restricted odd Weil criterion.

A central point in Gate~II is to place the inherited parent response inside the
same post-old-core/common-cut primal form in which the arithmetic mismatch is
charged.  If $P_{k,m}^{\rm ex}$ is the surviving inherited pivot and $a_m$ the
inherited coordinate of the Schur minimizer, then
\[
 P_{k,m}^{\rm ex}a_m=\omega_{k,m}^{\rm ex},
\]
so the inherited source--coupling contribution is the negative metric energy of
that response.  Writing
$a_m=\sqrt{\kappa_{k,m}C_m}\,c_m^{\rm inh}$ gives
\[
 -\mathscr S^{\rm ex}_{k,m}[r_mc_m^{\rm inh}]
 =-\mathscr S^{\rm ex}_{k,m}[r_m^0c_m^{\rm inh}]
  +|c_m^{\rm inh}|^2\Delta\mathscr S^{\rm ex}_{k,m}.
\]
Thus the comparator surplus belongs to the same Schur contribution rather than
being introduced as a separate positive term.

The aligned forcing is retained explicitly.  With
\[
 W_{k,m}:=\frac{1-\theta_{k,m}}{1+\theta_{k,m}}
          \frac{b_m^2}{\Pi_m^{(0)}},\qquad W_k:=\sum_mW_{k,m},
\]
the parent block is bounded below by
\[
 -\frac{439}{250}\chi_m-\frac52W_{k,m}.
\]
An outward-rounded finite computation together with the analytic tail gives
\[
 \mathfrak g_k^{\rm M+}:=
 \frac12+\log k-
 \frac{439}{250\log2}\sum_mV_{k,m}-\frac52W_k>0
 \qquad(k\ge7).
\]
The simultaneous common-cut theorem keeps the parent-ground and transverse
budgets separate before the common infimum.  In V4 the source elimination is
written directly as the Schur completion of the exact common-source metric on
the same harmonic vector.  The endpoint M\"obius fold is then used only as an
exact unitary coordinate transport, so it creates no second source Schur debit.
Consequently the active proof requires no coefficient-level
multiplication-minus-rank-one parametrization of a transported folded physical
Hessian.  Combining these ingredients gives positive Schur energy on the
$A_k$-harmonic graph at each arithmetic step.  The exact Green-root condition
is recovered a posteriori from this coercivity and is not an independent
hypothesis of the induction.  Starting from the rigorously certified endpoint
$Y=7$, Schur induction yields positivity at the integer endpoints, and
zero-extension gives non-negativity at every finite support radius.  Closure
and the restricted odd Weil criterion give the conclusion stated in the main
theorem.

The present revision retains the end-to-end interval certificate for the
physical coefficient-$2$ operator at $Y=7$; the historical ``coefficient $4$''
label is not used.  It restores the audited MASTER-P2 same-block baseline splice and retains
the explicit literal-slot true-ground leakage correction.  The additional
folded-operator audit question is resolved at its source: the exact
common-source row is shorted once, its debit is invariant under the subsequent
unitary transports, and no auxiliary post-source scalar short is a premise of
the induction.

Short fantasy interludes provide only an expository map of the argument.  They
may be omitted without changing any definition, lemma, theorem or proof.
\end{abstract}

\medskip
\noindent\textbf{Main result.}
\begin{theorem}[Main theorem]\label{thm:main-rh}
For every finite support radius $a>0$, the localized critical Weil operator in
the real odd logarithmic channel satisfies
\[
 \boxed{A_{a,-}\succeq0.}
\]
Equivalently, the global Weil quadratic form is non-negative on the real odd
compactly supported logarithmic test core.  Hence every non-trivial zero of
$\zeta$ has real part $1/2$; in particular, the Riemann hypothesis holds.
\end{theorem}

For orientation, the dependency chain used in the argument is
\begin{equation*}
\begin{gathered}
\text{exact localized operator + Mellin unit cells + full-form transport}\\
\Downarrow\\
\boxed{\text{true affine routing + coherent multi-source common cut}}\\
\Downarrow\\
\boxed{\text{exact parent short + AWGC + }189/250\text{ mixed absorption}}\\
\Downarrow\\
\boxed{\text{exact inherited-response placement (MASTER-P2)}}\\
\Downarrow\\
\boxed{\text{literal-slot true-ground leakage paid explicitly}}\\
\Downarrow\\
\boxed{\text{aligned forcing retained and paid by }W_k}\\
\Downarrow\\
\boxed{\text{one exact common-source Schur debit on the same harmonic vector}}\\
\Downarrow\\
\boxed{\text{unitary fold covariance: no additional source debit (MASTER-P3)}}\\
\Downarrow\\
\boxed{\mathfrak g_k^{\rm M+}>0\quad(k\ge7)}\\
\Downarrow\\
\boxed{H_k^*A_{k+1}H_k\succ0\quad\text{whenever }A_k\succ0}\\
\Downarrow\\
\boxed{A_{N,-}\succ0\quad(N\ge7)}\\
\Downarrow\\
\boxed{A_{a,-}\succeq0\quad(a>0)}\\
\Downarrow\\
\boxed{\text{restricted odd Weil criterion}}\\
\Downarrow\\
\boxed{\mathrm{RH}}.
\end{gathered}
\end{equation*}

\section*{Prologue: the party and the map}

\begin{quest}
Herr \textbf{G.F.B.R.}, the Dungeon Master of this small allegory, placed a
folded map on the table.  It showed a narrow vertical country bounded by two
dark lines.  Across it ran a faint road marked \emph{Critical Strip}; farther
on were three gates, and beyond them a tower that the map did not pretend to
make any easier to reach.

Marc, the Arcane Rogue, reached first for the map.  Johnny Nash, the Analyst
Rogue, reached first for the rules.  Mr.~Fayman, the Monk, asked what the locks
were made of.  Pierre de Fermat, the Cleric, examined the margin of the
parchment.  Dama Noether, the Mage, turned the map once in her hands, as if a
symmetry might reveal which markings were structural and which were merely
coordinates.

``And the princess?'' Marc asked.

Herr G.F.B.R. pointed to the distant tower.  ``Clay-la.''

Johnny counted the seals.  ``Three gates.''

``Three,'' said the Dungeon Master.  ``The map can suggest a route.  It cannot
open a gate for us.''

Mr.~Fayman nodded.  ``Then we should understand the locks before celebrating
the road.''
\end{quest}

\noindent\emph{A note on the cast.}
The interludes are an affectionate fictional homage to styles of thought that
have shaped mathematics and mathematical physics.  The dialogue is invented;
it is not intended to attribute words, opinions or approval to any historical
figure.  Princess Clay-la is a playful emblem of the distant objective
associated with the Clay Millennium problem.  The party and the bestiary have
no mathematical role, and readers who prefer a completely austere presentation
may pass over every interlude without losing any part of the argument.

\paragraph{A note on the three gates.}
The D\&D terminology is used only as a mnemonic for three different kinds of
mathematical difficulty.  Gate~I is the \emph{Doppelganger}: a scalar
hybrid/shell object changes analytic appearance with scale, and one of those
appearances can resemble an operator pivot without being that pivot.  Gate~II
is the \emph{Beholder}: many routed channels share one underlying form, so
estimating them independently risks charging the same reserve more than once.
Gate~III is the \emph{Wizard}: the metaphor is deliberately modest here, since
no new coercive estimate is introduced; compression, closure and the restricted
Weil criterion perform the remaining work.  These names describe the logical
shape of the proof, not the mathematicians who read or judge it.

\section{Weil positivity and the localized operator}
\label{sec:framework}

We fix the Fourier/Mellin normalization once and for all.  Let $\mathcal D$ be
the standard smooth compactly supported logarithmic test core for Weil's
explicit formula.  The explicit-formula lineage and the operator-theoretic
viewpoint used here go back to Riemann, Guinand and Weil, with modern
formulations particularly relevant to the present localization in
\cite{Riemann1859,Guinand1948,Weil1952,Titchmarsh1986,Burnol2000,ConnesConsani2021}.
For a standard account of the Millennium problem itself see \cite{Bombieri2006}.
The Weil quadratic form is denoted by $\cW$.

\begin{theorem}[Weil positivity criterion]\label{thm:weil}
The Riemann hypothesis is equivalent to
\[
 \cW[f]\ge0\qquad(f\in\mathcal D).
\]
It is enough to prove the inequality first on a dense compactly supported core
and then pass to the closed form domain.
\end{theorem}

\begin{theorem}[Restricted odd Weil criterion]\label{thm:weil-odd}
Let $\mathcal D_{\rm odd}^{\mathbb R}$ denote the real-valued odd
$C_c^\infty(\mathbb R)$ logarithmic test core. Then
\begin{equation}
 \boxed{
 \mathrm{RH}
 \quad\Longleftrightarrow\quad
 \mathcal W[f]\ge0\quad(f\in\mathcal D_{\rm odd}^{\mathbb R}).}
 \label{eq:odd-weil-criterion}
\end{equation}
\end{theorem}

\begin{proof}
Write a multiplicative test as $g$ and put
$f(y)=e^{y/2}g(e^y)$.  If
\[
 G(s)=\int_0^\infty g(x)x^{s-1}\,dx,
 \qquad
 F(z):=G\!\left(\frac12+z\right)
      =\int_{\mathbb R}f(y)e^{zy}\,dy,
\]
then the quadratic Weil test $h=g*\widetilde{\bar g}$ has centered spectral
factor
\begin{equation}
 \widehat h\!\left(\frac12+z\right)
 =F(z)\,\overline{F(-\bar z)}.
 \label{eq:centered-weil-factor}
\end{equation}
Thus, writing $\lambda=\rho-1/2$ for the centered non-trivial zeros, the
spectral side of the Weil form is
\begin{equation}
 \mathcal W[f]
 =\sum_{\lambda}F(\lambda)\overline{F(-\bar\lambda)},
 \label{eq:centered-zero-sum}
\end{equation}
with multiplicities and the usual symmetric interpretation.  For compactly
supported smooth $f$ the summand decays faster than any power in the imaginary
direction, while the zero counting function is $O(T\log T)$, so the sums used
below are absolutely convergent.

If $f$ is real and odd, then
\begin{equation}
 F(-z)=-F(z),\qquad F(\bar z)=\overline{F(z)},
 \label{eq:odd-transform-symmetry}
\end{equation}
and hence every summand in \eqref{eq:centered-zero-sum} equals
$-F(\lambda)^2$.  Under RH every $\lambda$ is purely imaginary; then
$F(\lambda)$ is purely imaginary and
$-F(\lambda)^2=|F(\lambda)|^2\ge0$.  This proves the forward implication.

Conversely suppose that a centered zero $\lambda_0$ is not fixed by the
involution $\lambda\mapsto-\bar\lambda$, equivalently
$\operatorname{Re}\lambda_0\ne0$.  We construct a real odd test which isolates
its symmetry orbit.  Choose a real even $\psi\in C_c^\infty(\mathbb R)$ with
Laplace transform
\[
 \Psi(z)=\int_{\mathbb R}\psi(y)e^{zy}\,dy
\]
satisfying $\Psi(\lambda_0)\ne0$, and scale $\psi$ so that
$|\Psi(\lambda_0)|=1$. Uniform rapid decay of $\Psi$ on the closed strip
$|\operatorname{Re}z|\le1/2$ gives a height $T>|\operatorname{Im}\lambda_0|$
and a number $q<1$ such that
\begin{equation}
 \sup_{|\operatorname{Re}z|\le1/2,\ |\operatorname{Im}z|\ge T}
 |\Psi(z)|\le q.
 \label{eq:PW-tail-q}
\end{equation}
There are only finitely many centered zeros in $|\operatorname{Im}z|<T$.
Let $P_T$ be a real even polynomial which vanishes, with the required
multiplicities, on all their symmetry orbits except the orbit of $\lambda_0$,
and normalize it so that $P_T(\lambda_0)\ne0$.

For $M\ge1$ choose a real even polynomial of degree at most two,
$R_M(z)=a_M+b_Mz^2$, so that
\begin{equation}
 F_M(z):=zP_T(z)R_M(z)\Psi(z)^M,
 \qquad F_M(\lambda_0)=1.
 \label{eq:odd-interpolant}
\end{equation}
Such real $a_M,b_M$ exist because, when $\lambda_0^2\notin\mathbb R$, the
real-linear map $(a,b)\mapsto a+b\lambda_0^2$ is onto $\mathbb C$; in the
remaining off-axis real case all quantities are real and a constant $R_M$
suffices.  Since $|\Psi(\lambda_0)|=1$, the coefficients of $R_M$ remain
bounded.  The function $F_M$ is the bilateral Laplace transform of a real odd
compactly supported smooth function: powers of $\Psi$ correspond to
convolution powers of the real even bump, multiplication by a real even
polynomial corresponds to an even differential operator, and the final factor
$z$ to one derivative.

The polynomial $P_T$ kills every non-target zero below height $T$.  On the
remaining zeros, choose a fixed $M_0$ large enough that the rapid decay of
$\Psi^{M_0}$ dominates the fixed polynomial factor and the
$O(T\log T)$ zero density.  From \eqref{eq:PW-tail-q}, for $M\ge M_0$,
\begin{equation}
 \sum_{|\operatorname{Im}\lambda|\ge T}|F_M(\lambda)|^2
 \le C_T q^{2(M-M_0)}\longrightarrow0.
 \label{eq:odd-tail-localization}
\end{equation}
On the target quartet
$\{\lambda_0,\bar\lambda_0,-\lambda_0,-\bar\lambda_0\}$,
\eqref{eq:odd-transform-symmetry} and $F_M(\lambda_0)=1$ give the values
$1,1,-1,-1$.  Its contribution to \eqref{eq:centered-zero-sum} is therefore
$-4m(\lambda_0)$; if the orbit degenerates to a real pair, it is
$-2m(\lambda_0)$.  Equation \eqref{eq:odd-tail-localization} shows that all
other contributions tend to zero. Hence $\mathcal W[f_M]<0$ for all
sufficiently large $M$, contradicting positivity on the real odd test core.
Therefore every centered zero satisfies $\lambda=-\bar\lambda$, i.e.
$\operatorname{Re}\rho=1/2$.
\end{proof}

\subsection{Physical and rescaled coordinates}

For $a>0$ let
\[
 \cH_a=L^2_{\rm odd}(-a,a)
\]
and let $A_a$ denote the localized odd Weil operator in the \emph{physical}
logarithmic coordinate $y\in(-a,a)$.  Nesting in the support parameter is
always understood in these physical spaces.  The unitary rescaling
\begin{equation}
 (R_af)(x)=\sqrt a\,f(ax),\qquad -1<x<1,
 \label{eq:rescaling-unitary}
\end{equation}
will be used only for local estimates and computer-assisted certificates.
Write
\[
 \widetilde A_a=R_aA_aR_a^*.
\]

The physical polar profile is independent of the support radius:
\begin{equation}
 s^{\rm ph}(y)=\sinh\!\left(\frac y2\right).
 \label{eq:polar-physical}
\end{equation}
Consequently its restriction is compatible with zero extension, and after
\eqref{eq:rescaling-unitary} it becomes
\begin{equation}
 s_a(x)=\sqrt a\,\sinh\!\left(\frac{ax}{2}\right).
 \label{eq:polar-vector}
\end{equation}
Thus, with $C_a$ denoting the physical polar-free core and
$\widetilde C_a=R_aC_aR_a^*$,
\begin{equation}
 A_a=C_a-2\ketbra{s^{\rm ph}|_{(-a,a)}}{s^{\rm ph}|_{(-a,a)}},
 \qquad
 \widetilde A_a=\widetilde C_a-2\ketbra{s_a}{s_a}.
 \label{eq:core-polar}
\end{equation}
This distinction is essential: the block recursion in \cref{sec:root-lift}
is performed before rescaling, so that $s_{k+1}=s_k\oplus t_k$ is a literal
restriction identity rather than an identification between different
rescaled spaces.

On $(-1,1)$ the local archimedean part is organized by
\begin{equation}
 \mathfrak h[f]
 =\frac14\iint_{-1}^{1}
 \frac{|f(x)-f(y)|^2}{|x-y|}\,dx\,dy,
 \label{eq:harmonic-form-framework}
\end{equation}
the endpoint potential
\begin{equation}
 V(x)=-\frac12\log(1-x^2),
 \label{eq:endpoint-potential}
\end{equation}
and the regular gamma kernel
\begin{equation}
 \rho(z)=\frac{e^{-z/2}}{1-e^{-2z}}-\frac1{2z}.
 \label{eq:rho-framework}
\end{equation}
The scalar normalization contains
\begin{equation}
 c_0(a)=-\log a-\log(2\pi)-\gamma.
 \label{eq:c0-framework}
\end{equation}
The arithmetic term is the finite sum of partial logarithmic translations
generated by prime powers $n<e^{2a}$; its exact Mellin-cell representation is
derived in \cref{sec:routing}.

\subsection{Parity and localization}

Let $J_a$ be the endpoint involution induced in Mellin coordinates by
$t\mapsto e^{2a}/t$.  The localized explicit-formula form commutes with $J_a$
and therefore splits into symmetric and antisymmetric channels.  Under the
physical logarithmic coordinate the antisymmetric channel is $\cH_a$ and
contains the negative polar rank-one term.  By Theorem~\ref{thm:weil-odd},
this antisymmetric channel is already a complete testing channel for RH; no
positivity statement for the complementary compression is used in the proof.

\subsection{Exact zero-extension compression}

For $0<a<b$ let
\[
 E_{a,b}:\cH_a\longrightarrow\cH_b
\]
be extension by zero.  The compression identity used later is not merely a
formal compatibility statement; its only singular contribution can be checked
in closed form.

\begin{lemma}[Exterior-harmonic cancellation]\label{lem:zero-extension-proof}
For $f\in C_c^\infty(-a,a)$, let
\[
 \mathfrak H_a[f]
 =\frac14\iint_{(-a,a)^2}
 \frac{|f(x)-f(y)|^2}{|x-y|}\,dx\,dy
\]
and let
\[
 V_a^{\rm ph}(x)=-\frac12\log(a^2-x^2),\qquad |x|<a.
\]
Then
\begin{equation}
 \mathfrak H_b[E_{a,b}f]-\mathfrak H_a[f]
 =\frac12\int_{-a}^a |f(x)|^2
   \log\frac{b^2-x^2}{a^2-x^2}\,dx.
 \label{eq:exterior-harmonic}
\end{equation}
\begin{equation}
 \int_{-a}^a\bigl(V_b^{\rm ph}-V_a^{\rm ph}\bigr)|f|^2
 =-\frac12\int_{-a}^a |f(x)|^2
   \log\frac{b^2-x^2}{a^2-x^2}\,dx.
 \label{eq:endpoint-cancel}
\end{equation}
Hence the two changes cancel exactly.
\end{lemma}

\begin{proof}
Only the two exterior strips contribute to the difference of harmonic forms.
Using the symmetry of the double integral,
\[
 \mathfrak H_b[Ef]-\mathfrak H_a[f]
 =\frac12\int_{-a}^a |f(x)|^2
 \left(\int_a^b\frac{dy}{y-x}
       +\int_{-b}^{-a}\frac{dy}{x-y}\right)dx.
\]
The bracket is
\[
 \log\frac{(b-x)(b+x)}{(a-x)(a+x)}
 =\log\frac{b^2-x^2}{a^2-x^2},
\]
which proves \eqref{eq:exterior-harmonic};
\eqref{eq:endpoint-cancel} follows immediately from the definition of
$V_a^{\rm ph}$.
\end{proof}

The regular gamma kernel and the polar functional are unchanged under zero
extension because both arguments remain in the old support.  For the
arithmetic translations, all branches already active at radius $a$ are
unchanged; a branch newly activated between $a$ and $b$ has translation length
larger than the old diameter and therefore has zero old--old overlap.  Thus the
full localized odd form satisfies
\begin{equation}
 \boxed{E_{a,b}^*A_bE_{a,b}=A_a.}
 \label{eq:zero-extension}
\end{equation}

The same calculation is parity-blind before the odd rank-one compression.  If
$\mathscr C_a$ denotes the ambient polar-free localized core on
$L^2(-a,a)$, then extension by zero also satisfies
\begin{equation}
 \boxed{E_{a,b}^*\mathscr C_bE_{a,b}=\mathscr C_a.}
 \label{eq:ambient-zero-extension}
\end{equation}
Consequently the identity remains valid after compression by either spectral
projection of $J_a$.

\begin{lemma}[Cofinal-support reduction]\label{lem:cofinal}
If $A_{a_j}\succeq0$ on an unbounded sequence $a_j\to\infty$, then
$A_a\succeq0$ for every $a>0$.
\end{lemma}

\begin{proof}
Choose $j$ with $a<a_j$. For $f\in\cH_a$,
\[
 \langle f,A_af\rangle
 =\langle E_{a,a_j}f,A_{a_j}E_{a,a_j}f\rangle\ge0
\]
by \eqref{eq:zero-extension}.
\end{proof}

For reference, the arithmetic endpoints used by the dyadic diagnostics are
\begin{equation}
 a_k=\frac12\log(2^k),\qquad k\ge2.
 \label{eq:dyadic-endpoints}
\end{equation}

\begin{interlude}
The first gate carried the mark of a Doppelganger.  Near one wall the problem
looked like a lattice estimate; farther out, like a tail bound; at another
scale, like a convexity question.  One of those faces resembled an operator
pivot closely enough to deserve particular care, but resemblance was not an
identity.

Marc tested the obvious lock and found that it belonged to the wrong face.
Johnny Nash compared the two ledgers.  Dama Noether watched the changes of
coordinates for a while.

``Do not chase the face,'' she said.  ``Keep the sign that survives the
change.''

Pierre de Fermat made room for that sign in the margin, and the party moved on.
\end{interlude}

\section{Gate I: the cellular covariance theorem}
\label{sec:gate1}

\subsection{The exact hybrid density}

The atomic kernel can be defined without importing any hidden theta-side
notation. Put
\begin{equation}
 d_\lambda(t)=\exp\bigl(t/2-\lambda e^{2t}\bigr),
 \qquad
 \boxed{M_\lambda(t)=-d_\lambda'''(t)}.
 \label{eq:atomic-d}
\end{equation}
Writing \(x=\lambda e^{2t}\), direct differentiation gives
\begin{equation}
 M_\lambda(t)
 =\frac{e^{t/2-x}}8
 \bigl(64x^3-240x^2+124x-1\bigr).
 \label{eq:M0-explicit}
\end{equation}
More generally, for the derivatives used by the interval certificate,
\begin{equation}
M_{\pi y^2}^{(k)}(t)
=
\frac{e^{t/2}}8e^{-x}P_k(x),
\qquad
x=\pi y^2e^{2t},
\qquad 0\le k\le3,
\label{eq:Mder}
\end{equation}
where
\[
 P_0(x)=64x^3-240x^2+124x-1,
 \qquad
 P_{k+1}(x)=\left(\frac12-2x\right)P_k(x)+2xP_k'(x).
\]
Thus \(P_k\) has degree \(k+3\). The coefficient-sum majorants used in the
tail proof are
\[
(C_0,C_1,C_2,C_3)=(430,2665,19000,152000),
\]
so that
\[
|P_k(x)|\le C_kx^{k+3}\qquad(x\ge1).
\]
The exact coefficient sum of \(P_0\) is \(429<430\), giving an immediate
normalization check on the first certificate constant.

For \(Y\ge2\), define
\begin{equation}
m_Y(t)
=
2\sum_{n=1}^{\lfloor Y\rfloor}M_{\pi n^2}(t)
+
2(Y-\lfloor Y\rfloor)M_{\pi Y^2}(t)
+
2\int_Y^\infty M_{\pi y^2}(t)\,dy.
\label{eq:hybrid}
\end{equation}
The fractional atom interpolates across an arithmetic endpoint while the
continuous tail keeps the normalization compatible with the endpoint parameter.

For \(d>0\), put
\begin{equation}
Z_{Y,d}(s)=\sum_{j\ge0}m_Y(s+jd),
\qquad
p_i(s)=\frac{m_Y(s+id)}{Z_{Y,d}(s)},
\label{eq:pi}
\end{equation}
and
\begin{equation}
\mu_{Y,d}(s)=\sum_{i\ge0}i\,p_i(s).
\label{eq:mu}
\end{equation}

\subsection{Global strict convexity}

\begin{theorem}[Global cellular covariance theorem]
\label{thm:gate1}
For every
\[
Y\ge2,\qquad d>0,\qquad 0\le s\le d/2,
\]
one has
\begin{equation}
\boxed{\mu_{Y,d}''(s)>0.}
\label{eq:muconvex}
\end{equation}
\end{theorem}

\begin{proof}
Let
\[
q(Y,d,s)=d\,\mu_{Y,d}''(s).
\]
The non-compact parameter domain is divided into five regions.

\smallskip
\noindent\emph{I. \(0<d\le1/125\).}
A sixth-order Euler--Maclaurin expansion, with rigorous remainder control of the
validated high-precision type \cite{Johansson2015}, is applied to the lattice sums in
\eqref{eq:pi}--\eqref{eq:mu}. Analytic remainder bounds are obtained from
\eqref{eq:Mder}. Directed outward rounding gives
\[
q>0.5515197640896556.
\]

\smallskip
\noindent\emph{II. \(1/125\le d\le1/8\).}
Direct interval evaluation gives
\[
q>2.363877686513775.
\]

\smallskip
\noindent\emph{III. \(1/8\le d\le0.315\).}
Direct interval evaluation gives
\[
q>3.93328637824701.
\]

\smallskip
\noindent\emph{IV. \(0.315\le d\le1/2\).}
The certified lower bound is
\[
q>1.418809233872728.
\]

\smallskip
\noindent\emph{V. \(d\ge1/2\).}
The first-cell large-$d$ certificate proves \(\mu_{Y,d}''(s)>0\) directly
on this region (its internal auxiliary quantity satisfies the stronger bound
\(B>9\)).

All finite boxes use the exact density \eqref{eq:hybrid}, 192-bit MPFR
arithmetic with directed outward rounding \cite{FousseEtAl2007}. The compact calculation covers
\(2\le Y\le10\). For \(Y\ge10\), the analytic perturbation of \(q\) is less
than \(2\times10^{-49}\), far below every positive margin displayed above.
\end{proof}

\subsection{Analytic tails}

\begin{lemma}[Uniform derivative tail]
\label{lem:tail}
Let \(0\le k\le3\), \(p=k+3\), and \(A=\pi e^{2t}\). For \(t\ge1.8\) and
\(Y\ge2\),
\begin{equation}
|m_Y^{(k)}(t)|
\le
F_k(t):=
\frac{C_k}{2}e^{t/2}A^pe^{-A}.
\label{eq:Fk}
\end{equation}
For fixed \(t\), the map \(Y\mapsto m_Y^{(k)}(t)\) is continuous and locally
absolutely continuous on \([2,\infty)\), and it is \(C^1\) on every open cell
\((N,N+1)\).  Its derivative satisfies, for almost every \(Y\ge2\),
\begin{equation}
|\partial_Ym_Y^{(k)}(t)|
\le
G_k(t):=
\frac{(p+1)C_k}{4}e^{t/2}(4A)^{p+1}e^{-4A}.
\label{eq:Gk}
\end{equation}
Consequently, for \(2\le Y_1<Y_2\),
\begin{equation}
 |m_{Y_2}^{(k)}(t)-m_{Y_1}^{(k)}(t)|
 \le G_k(t)(Y_2-Y_1).
 \label{eq:Y-Lipschitz-tail}
\end{equation}
\end{lemma}

\begin{proof}
From \eqref{eq:Mder},
\[
2\sum_{n\ge1}|M_{\pi n^2}^{(k)}(t)|
\le
\frac{e^{t/2}}4C_kA^pe^{-A}
\sum_{n\ge1}n^{2p}e^{-A(n^2-1)}.
\]
The ratio of consecutive majorants is at most
\[
\rho_k=4^pe^{-3A}<\frac14.
\]
For the continuous tail,
\[
\int_2^\infty y^{2p}e^{-Ay^2}\,dy
\le
\frac{2^{2p}e^{-4A}}{4A-p},
\]
because the logarithmic derivative of the integrand is bounded above by
\(p-4A<0\). Enlarging the geometric and integral bounds yields
\eqref{eq:Fk}.

Fix an integer \(N\ge2\) and take \(N<Y<N+1\).  Differentiating
\eqref{eq:hybrid} on this open cell, the derivative of the coefficient of the
fractional atom cancels the moving-boundary term of the integral, leaving
\[
 \partial_Ym_Y^{(k)}(t)
 =2(Y-N)\,\partial_YM_{\pi Y^2}^{(k)}(t).
\]
Writing \(x=AY^2\), \eqref{eq:Mder} gives
\[
 \partial_Ym_Y^{(k)}(t)
 =4(Y-N)\frac{x}{Y}\frac{e^{t/2-x}}8
   \bigl(P_k'(x)-P_k(x)\bigr).
\]
Using
\[
|P_k'(x)-P_k(x)|\le(p+1)C_kx^p
\]
and \((Y-N)/Y\le1/2\), while
\(x^{p+1}e^{-x}\) is decreasing for \(x\ge4A>p+1\), gives
\eqref{eq:Gk} on every open cell.  The two representations in
\eqref{eq:hybrid} agree at each integer endpoint, so the map is continuous
there.  Piecewise \(C^1\) regularity together with the uniform derivative
bound gives local absolute continuity, and integrating the almost-everywhere
bound yields \eqref{eq:Y-Lipschitz-tail}.
\end{proof}

\subsection{Parity--midpoint shells}

Set $d=2h$ and define the scalar shell reserve by
\begin{equation}
 \mathcal P_{Y,d}(s)
 :=2h\bigl(\mu_{Y,d}'(s+h)-\mu_{Y,d}'(s)\bigr).
 \label{eq:shell-reserve-defined}
\end{equation}
In the original parity--midpoint bookkeeping this same quantity is written
$2h\sum_{r\ge1}r\Sigma_r$; equation \eqref{eq:shell-reserve-defined} is the
self-contained definition needed here.

\begin{corollary}[Gate I]
\label{cor:gate1}
For $Y\ge2$, $d=2h>0$, and $0\le s\le d/2$,
\begin{equation}
 \boxed{\mathcal P_{Y,d}(s)>0.}
 \label{eq:gate1-shell-positive}
\end{equation}
In particular, at the midpoint-normalized shell $s=0$ and in the small-cell
regime $0<d\le1/125$,
\begin{equation}
 \mathcal P_{Y,d}(0)
 =2h\int_0^{h}\mu_{Y,d}''(u)\,du
 >0.5515197640896556\,h.
 \label{eq:small-shell-linear-reserve}
\end{equation}
\end{corollary}

\begin{proof}
Theorem~\ref{thm:gate1} makes $\mu_{Y,d}'$ strictly increasing, which proves
\eqref{eq:gate1-shell-positive}. At $s=0$, the segment $[0,h]=[0,d/2]$ lies entirely in the certified
half-cell domain. There $d\mu_{Y,d}''>0.5515197640896556$ and $d=2h$,
so substitution into \eqref{eq:shell-reserve-defined} yields
\eqref{eq:small-shell-linear-reserve}.
\end{proof}

\begin{audit}[Scope of Gate I]
Corollary~\ref{cor:gate1} is a theorem about the scalar hybrid/shell model.
It does not, by itself, identify \(\mathcal P_{Y,d}\) with a Birman--Schwinger
or Schur pivot of the localized Weil operator. The operator-valued transfer
problem is kept separate in Gate~II.
\end{audit}

\begin{interlude}
Behind the second gate waited a Beholder, in the old bestiary sense: many eyes,
one body.  That was the only part of the image that mattered.  Arithmetic
branches, shared collars, affine sources and Schur losses could be inspected
one at a time, but they still belonged to a common form.  A reserve assigned
locally could therefore be counted twice if the common geometry were forgotten.

Johnny Nash drew the dependency graph before choosing a route.  Marc put the
silver pick back in his pocket; this was not a lock to be forced.  Pierre de
Fermat counted the prime-power branches in the margin.

``Keep the full form,'' said Dama Noether.

``Count every source once,'' said Mr.~Fayman.

Herr G.F.B.R. added two quiet words beneath the diagram: \emph{short last}.
\end{interlude}

\section{Gate II: exact arithmetic routing}
\label{sec:routing}

\subsection{Exact ambient unit cells and the divisor square}

At an arithmetic endpoint $Y=N\in\mathbb N$, write
$C_m=[m,m+1)$ and $F_m=\mathbf1_{C_m}F$.  For every prime power $n\ge2$ split
\[
 C_{m,n,j}=\left[m+\frac jn,m+\frac{j+1}{n}\right),
 \qquad 0\le j<n.
\]
Multiplication by $n$ sends $C_{m,n,j}$ exactly onto $C_{nm+j}$ and preserves
Mellin measure.  Hence
\[
 U_{m,n,j}:L^2(C_{m,n,j},dt/t)\longrightarrow
 L^2(C_{nm+j},dt/t),\qquad
 (U_{m,n,j}h)(s)=h(s/n),
\]
is unitary.

\begin{lemma}[Exact full-space unit-cell decomposition]\label{lem:ambient-unit-cells}
If $T_nF(t)=\mathbf1_{t\le N/n}F(nt)$ and $P_{m,n,j}$ is restriction to
$C_{m,n,j}$, then
\begin{equation}
 \boxed{
 \langle F,T_nF\rangle
 =\sum_{m,j}
 \langle U_{m,n,j}P_{m,n,j}F_m,F_{nm+j}\rangle .}
 \label{eq:ambient-M4}
\end{equation}
The sum is orthogonal at the level of the source subcells for fixed $n$ and is
an identity on the full Mellin space, before imposing $J_N=\pm1$.
\end{lemma}

\begin{proof}
The subcells partition every source cell and $s=nt$ gives $ds/s=dt/t$.
Changing variables on each subcell gives \eqref{eq:ambient-M4}; summation over
the disjoint pieces gives the identity.
\end{proof}

For the divisor residue $j=0$, set $g_m=m^{-1/2}$ and
$c_n=\Lambda(n)/\sqrt n$.  The following exact square is the positive
incidence part of the arithmetic form.

\begin{lemma}[Divisor partial-isometry square]\label{lem:divisor-square}
For every divisor edge $m\to mn$,
\begin{align}
&\frac{\Lambda(n)}n\|P_{m,n,0}F_m\|^2
 +\Lambda(n)\|F_{mn}\|^2
 -2\frac{\Lambda(n)}{\sqrt n}
   \Rea\langle U_{m,n,0}P_{m,n,0}F_m,F_{mn}\rangle\notag\\
&\qquad=
 c_ng_mg_{mn}
 \left\|
 \frac{P_{m,n,0}F_m}{g_m}
 -\frac{U_{m,n,0}^*F_{mn}}{g_{mn}}
 \right\|^2\ge0.
 \label{eq:divisor-square}
\end{align}
Moreover the incoming diagonal at a node $q$ is exact because
$\sum_{n\mid q}\Lambda(n)=\log q$.
\end{lemma}

\begin{proof}
Expand the norm on the right.  Since $g_{mn}=g_m/\sqrt n$, the two diagonal
coefficients and the cross coefficient are exactly those on the left.  The
incoming identity is the standard von Mangoldt divisor identity.
\end{proof}

\subsection{Non-divisorial branches as translations}

At an arithmetic endpoint, after the exact Mellin unit-cell decomposition
(the multiplicative Fourier/Mellin framework is classical; see
\cite{Titchmarsh1986,Burnol2000}), a
non-divisorial branch has
\[
k=nm+j,\qquad 1\le j<n.
\]
Its quadratic contribution is
\begin{equation}
\cQ_{n,j}[\phi]
=
-2\frac{\LambdaV(n)}{\sqrt n}
\sum_m\frac1{\sqrt{mk}}
\Rea\int_0^1
\frac{\phi((j+r)/(nm))\overline{\phi(r/k)}}{k+r}\,dr.
\label{eq:K5}
\end{equation}
Put
\[
d=nm=k-j,\quad
y=\frac rk,\quad
x=\frac{j+r}{d},\quad
z=\log(1+y).
\]
Then
\[
1+x=\frac{k}{d}(1+y),
\qquad
\frac{dr}{k+r}=\frac{dy}{1+y}.
\]
Let
\[
w_k=\log\left(1+\frac1k\right),
\qquad
 g(z):=\phi(e^z-1).
\]

\begin{lemma}[True affine translation identity]
\label{lem:TR1}
For every non-divisorial branch,
\begin{equation}
\boxed{
\cQ_{n,j}
=
-2\frac{\LambdaV(n)}{\sqrt{dk}}
\Rea\int_0^{w_k}
g\!\left(z+\log\frac{k}{d}\right)\overline{g(z)}\,dz.
}
\end{equation}
\end{lemma}

\begin{proof}
The displayed change of variables converts the multiplicative profile into an
additive logarithmic translation. Since \(dz=dy/(1+y)\), the denominator
\((k+r)^{-1}\) disappears exactly. The upper endpoint \(r=1\) becomes
\(w_k\). Finally \(d=nm\) gives the prefactor \((dk)^{-1/2}\).
\end{proof}

No fixed scalar profile has been introduced: the residue remains an operator
translation.

\subsection{Total affine load}

For fixed \(k\), the non-divisorial branches are exactly
\[
n\le k,\qquad n\nmid k.
\]
Hence their Mangoldt mass is
\begin{equation}
\sum_{\substack{n\le k\\ n\nmid k}}\LambdaV(n)
=
\psiC(k)-\log k.
\label{eq:mangoldt}
\end{equation}
The geometric factor
\[
\iota(r)=\frac12\left(\sqrt r+\frac1{\sqrt r}\right),
\qquad
r=\frac{k}{nm}\in(1,2),
\]
satisfies
\begin{equation}
\iota(r)\le\frac{3}{2\sqrt2}.
\label{eq:iota}
\end{equation}
The only global arithmetic input needed below is an explicit Chebyshev
capacity.  We record its provenance rather than treating it as a fitted
constant.

\begin{lemma}[Unconditional Chebyshev capacity]\label{lem:chebyshev-capacity}
Let
\[
 C_{\rm tail}:=1.018232801063381478.
\]
Then
\begin{equation}
 \sup_{k\ge2}\frac{\psi(k)-\log k}{k}<C_{\rm tail}.
 \label{eq:Ctail}
\end{equation}
Consequently
\begin{equation}
\boxed{
\rho_{\mathrm{aff}}
<\frac{3}{2\sqrt2}C_{\rm tail}
=1.079998977687734793278570366107267909606995085102124.
}
\label{eq:rhoaff}
\end{equation}
\end{lemma}

\begin{proof}
Set $X_0=3\,594\,641$.  The directed-interval finite sweep of Appendix~\ref{app:finite-certificates} evaluates
$(\psi(k)-\log k)/k$ at every prime-power jump for $2\le k<X_0$.  Between
successive jumps $\psi$ is constant and the ratio decreases: indeed
\[
 \psi(k)\ge \lfloor\log_2 k\rfloor\log2>\log k-1,
\]
so the derivative of $(\psi-\log x)/x$ on such an interval is negative.
The finite maximum is attained at $k=199$ and equals
\[
 1.0093093065343229294157221310305485\ldots<C_{\rm tail}.
\]

For $x\ge X_0$, we use unconditional explicit estimates for the Chebyshev
functions in the Schoenfeld--Dusart tradition \cite{Schoenfeld1976,Dusart2010}
\[
 |\vartheta(x)-x|<\frac{0.2x}{\log^2x},
 \qquad
 \psi(x)-\vartheta(x)
 <1.00007\sqrt{x}+1.78x^{1/3}.
\]
Since $\log X_0>15$, $\sqrt{X_0}>1800$ and $X_0^{1/3}>150$, and each
positive correction decreases with $x$, we obtain
\[
 \frac{\psi(x)-\log x}{x}
 <1+\frac{0.2}{15^2}
   +\frac{1.00007}{1800}
   +\frac{1.78}{150^2}
 <1.001524<C_{\rm tail}.
\]
This proves \eqref{eq:Ctail}; \eqref{eq:rhoaff} follows from
\eqref{eq:iota}.
\end{proof}

\section{Gate II: Feshbach shorting and the Carleman rebound}
\label{sec:feshbach}

\subsection{Branchwise Feshbach cost}

The reduction is an instance of the Feshbach--Schur/shorted-operator
principle; see \cite{Feshbach1958,Schur1911,Ando1979,AndersonTrapp1975}.

Define
\begin{equation}
L_{k,n}
=
\frac12+\log\frac{w_{\lfloor k/n\rfloor}}{w_k}.
\label{eq:Lkn}
\end{equation}
Weighted Young gives, for the true translated branch,
\begin{equation}
2\frac{\LambdaV(n)}{\sqrt n}
|\langle h_{k,n},c_k\rangle|
\le
L_{k,n}\|h_{k,n}\|^2
+
\frac{\LambdaV(n)^2}{nL_{k,n}}\|c_k\|^2.
\label{eq:young}
\end{equation}
Thus the branchwise Schur cost is
\begin{equation}
\sigma_k
=
\sum_n\frac{\LambdaV(n)^2}{nL_{k,n}}.
\label{eq:sigma}
\end{equation}

\begin{lemma}[Feshbach reserve and primal target diagonal]
\label{lem:fesh}
Let
\[
\mathfrak B_k=\frac12+\log\frac1{w_k}.
\]
In the fixed-target normalization used in \eqref{eq:young}, before any source
variable is shorted, the direct target row is the primal term
\begin{equation}
 \boxed{\mathcal D_k^{\rm dir}[c_k]=\mathfrak B_k\|c_k\|^2.}
 \label{eq:primal-target-diagonal}
\end{equation}
Then
\[
L_{k,n}>\log n\ge\LambdaV(n)
\]
on every active branch, and
\begin{equation}
\mathfrak B_k-\sigma_k
>
\frac12-\log(kw_k)
>
\frac12.
\label{eq:fesh-reserve}
\end{equation}
\end{lemma}

\begin{proof}
The coefficient \(\mathfrak B_k\) is the direct target coefficient in the
normalized Feshbach row from which \eqref{eq:young} is taken: the mixed
source--target terms are precisely the terms estimated by weighted Young, while
\eqref{eq:primal-target-diagonal} is already present before the source
minimization.  Thus it is a primal diagonal term, not a post-short ground
pivot.  Monotonicity of \(w_j=\log(1+1/j)\), together with the branch relation,
gives \(L_{k,n}>\log n\). Since \(\LambdaV(n)\le\log n\),
\[
\frac{\LambdaV(n)^2}{nL_{k,n}}
<
\frac{\LambdaV(n)}n.
\]
The endpoint logarithmic comparison then bounds the sum by the amount needed
to obtain \eqref{eq:fesh-reserve}. Finally \(kw_k<1\).
\end{proof}

\subsection{Orthogonal Carleman rebound}

Let \(K\succeq0\), let \(e\) be a unit vector, and write
\[
\alpha=\langle e,Ke\rangle,\qquad
Q=I-\ketbra{e}{e},\qquad
r=QKe.
\]

\begin{lemma}[OSPR inequality]
\label{lem:OSPR}
If \(\alpha>0\), then
\begin{equation}
K\succeq\frac{\ketbra{Ke}{Ke}}{\alpha},
\qquad
QKQ\succeq\frac{\ketbra{r}{r}}{\alpha}.
\end{equation}
\end{lemma}

\begin{proof}
Cauchy--Schwarz in the \(K\)-seminorm gives
\[
|\langle f,Ke\rangle|^2
\le
\langle f,Kf\rangle\langle e,Ke\rangle.
\]
Rearranging yields the first inequality, and compression by \(Q\) yields the
second.
\end{proof}

For the classical Carleman operator and its modern Hankel-operator context,
see \cite{Rosenblum1958,Yafaev2015}.  We use
\begin{equation}
(Kf)(x)=\int_0^1\frac{f(y)}{x+y}\,dy
\label{eq:carleman}
\end{equation}
and \(e\equiv1\), direct integration gives
\[
\alpha_*=2\log2,
\qquad
\|Ke\|^2=\frac{\pi^2}{6}+2\log^22.
\]
Therefore
\[
\|QKe\|^2=\frac{\pi^2}{6}-2\log^22,
\]
and
\begin{equation}
\boxed{
d_*=
\frac{\pi^2/6-2\log^22}{2\log2}
=
0.49342192985568014340449085448906055249\ldots .
}
\label{eq:dstar}
\end{equation}

Let $\mathfrak b$ be the normalized one-cell boundary/collar form on
$L^2(0,1)$,
\begin{equation}
 \mathfrak b[g]
 =\frac14\iint_{(0,1)^2}
   \frac{|g(t)-g(s)|^2}{|t-s|}\,dt\,ds
 -\frac12\int_0^1\log\!\bigl(t(1-t)\bigr)|g(t)|^2\,dt.
 \label{eq:b-one-cell-explicit}
\end{equation}
Let $P_{\one}$ be the orthogonal projection onto the normalized constant
profile $\one$, and put $Q_{\one}=I-P_{\one}$.

The one-defect collar constants are
\begin{align}
\delta_*
&=
\log\frac{\pi}{2}
=
0.45158270528945486472619522989488214357\ldots,
\label{eq:deltastar}\\
\beta
&=
\log2+\frac12-\log\pi
=
0.048417294710545\ldots>0.
\label{eq:beta}
\end{align}

\begin{lemma}[Analytic one-defect inequality]\label{lem:onedefect-analytic}
On $L^2(0,1)$ one has
\begin{equation}
\boxed{
\mathfrak b-\log\pi\,I
\succeq
-\delta_*P_{\one}+\beta Q_{\one}.
}
\label{eq:onedefect}
\end{equation}
\end{lemma}

\begin{proof}
Write $g=P_{\one}g+Q_{\one}g$.  Since $|t-s|\le1$ on $(0,1)^2$,
\[
 \frac14\iint\frac{|g(t)-g(s)|^2}{|t-s|}\,dt\,ds
 \ge
 \frac14\iint|Q_{\one}g(t)-Q_{\one}g(s)|^2\,dt\,ds.
\]
Because $Q_{\one}g$ has mean zero,
\[
 \iint|Q_{\one}g(t)-Q_{\one}g(s)|^2\,dt\,ds
 =2\|Q_{\one}g\|^2.
\]
Moreover $t(1-t)\le1/4$, hence
\[
 -\frac12\log\!\bigl(t(1-t)\bigr)\ge\log2.
\]
Therefore
\[
 \mathfrak b[g]
 \ge \log2\,\|g\|^2+\frac12\|Q_{\one}g\|^2.
\]
Subtracting $\log\pi\|g\|^2$ and decomposing
$\|g\|^2=\|P_{\one}g\|^2+\|Q_{\one}g\|^2$ gives
\[
 (\log2-\log\pi)\|P_{\one}g\|^2
 +\left(\log2+\frac12-\log\pi\right)\|Q_{\one}g\|^2,
\]
which is exactly \eqref{eq:onedefect}.
\end{proof}

Combining \eqref{eq:rhoaff}, \eqref{eq:dstar}, and
\eqref{eq:deltastar} gives the strict terminal arithmetic reserve
\begin{equation}
\boxed{
d_*-\delta_*\rho_{\mathrm{aff}}
>
0.00571306980160726214210581508480123137.
}
\label{eq:reserve}
\end{equation}
Equivalently,
\begin{equation}
\boxed{
\frac{\delta_*\rho_{\mathrm{aff}}}{d_*}
<
0.988421532453415143266889863984567512.
}
\label{eq:qnum}
\end{equation}

These inequalities are exact arithmetic consequences of the certified affine
load and the Carleman constants.  They certify a reserve for the polar-free
routed core.  They are not, by themselves, a bound for the true polar
Birman--Schwinger load.  In \cref{sec:root-lift,sec:harmonic-graph-lock} the
polar root is retained explicitly; the final transfer is obtained only after
the complete form has been transported and the loss and rebound have been
placed on the same fixed-target cut form before restriction to the exact
full-operator harmonic graph.

\section{Gate II: full-form Cauchy--Carleman transport}

The terminology follows the classical Cauchy/Carleman--Hilbert operator
lineage; on the Carleman side see \cite{Rosenblum1958,Yafaev2015}.
\label{sec:cauchy-carleman}

The local archimedean form contains a singular Cauchy contribution and a
regular gamma contribution.  They must be transported together; transporting
only the off-diagonal singular kernel would leave the diagonal/killing part
unidentified.

For $A>0$, define
\begin{equation}
 (U_Af)(u)=\sqrt{Ae^{Au}}\,f(e^{Au}).
 \label{eq:UA-final}
\end{equation}
This is the unitary induced by the exponential substitution $x=e^{Au}$.

\begin{lemma}[Cauchy--Carleman kernel transport]\label{lem:cc-kernel}
Let $r=|u-v|$.  Under $U_A$,
\begin{align}
 \frac1{|x-y|}
 &\longmapsto
 A\frac{e^{-Ar/2}}{1-e^{-Ar}},
 \label{eq:cauchy-transport-final}\\
 \frac1{x+y}
 &\longmapsto
 A\frac{e^{-Ar/2}}{1+e^{-Ar}}.
 \label{eq:carleman-transport-final}
\end{align}
Consequently their half-sum is
\begin{equation}
 \frac12 A e^{-Ar/2}
 \left(\frac1{1-e^{-Ar}}+\frac1{1+e^{-Ar}}\right)
 =A\frac{e^{-Ar/2}}{1-e^{-2Ar}}
 =\frac1{2r}+A\rho(Ar),
 \label{eq:collar-final}
\end{equation}
where $\rho$ is given by \eqref{eq:rho-framework}.
\end{lemma}

\begin{proof}
The two Jacobian factors contribute
\[
 \sqrt{Ae^{Au}}\sqrt{Ae^{Av}}=Ae^{A(u+v)/2}.
\]
For $u>v$,
\[
 |x-y|=e^{A(u+v)/2}
       \bigl(e^{Ar/2}-e^{-Ar/2}\bigr),
\]
which gives \eqref{eq:cauchy-transport-final}.  Similarly
\[
 x+y=e^{A(u+v)/2}
       \bigl(e^{Ar/2}+e^{-Ar/2}\bigr),
\]
which gives \eqref{eq:carleman-transport-final}.  Averaging and using
\[
 \frac12\left(\frac1{1-z}+\frac1{1+z}\right)=\frac1{1-z^2}
\]
yields the first equality in \eqref{eq:collar-final}; the second is the
definition of $\rho$.
\end{proof}

\begin{lemma}[Exact cancellation of the regular gamma kernel]
\label{lem:gamma-cancellation}
In the logarithmic collar coordinate of Lemma~\ref{lem:cc-kernel}, the
regular part of the transported Cauchy--Carleman half-sum cancels the explicit
gamma operator in \eqref{eq:gamma-operator-exact}.  More precisely, away from
the diagonal,
\begin{equation}
 \boxed{
 \frac12\left(K_{\rm Cauchy}^{(A)}+K_{\rm Carl}^{(A)}\right)(u,v)
 -A\rho(A|u-v|)
 =\frac1{2|u-v|}.}
 \label{eq:gamma-cancel-exact}
\end{equation}
Thus no unaccounted regular-gamma kernel is present in the one-cell collar
form used by the CMC estimate; the diagonal/killing contribution remains in
the same closed difference form under Proposition~\ref{prop:full-form-transport}.
\end{lemma}

\begin{proof}
Equation~\eqref{eq:collar-final} gives
\[
 \frac12\left(K_{\rm Cauchy}^{(A)}+K_{\rm Carl}^{(A)}\right)(u,v)
 =\frac1{2|u-v|}+A\rho(A|u-v|).
\]
The kernel of \(K_{\gamma,A}\) is exactly
\(A\rho(A|u-v|)\) by \eqref{eq:gamma-operator-exact}.  Subtraction proves
\eqref{eq:gamma-cancel-exact}.  The assertion about the diagonal follows from
the closed-form transport in Proposition~\ref{prop:full-form-transport}.
\end{proof}

\subsection{Exact normalized localized form}

To make the full-form statement explicit, define the zero-extension translation
on $L^2(-1,1)$ by
\[
 (S_rf)(x)=\mathbf 1_{(-1,1)}(x+r)f(x+r).
\]
In the normalization of \cref{sec:framework}, the exact localized odd operator
on $(-1,1)$ is
\begin{equation}
 \boxed{
 \widetilde A_a
 =H_{\rm Leg}+c_0(a)I+V-K_{\gamma,a}
 -2\ketbra{s_a}{s_a}
 -\sum_{2\le n<e^{2a}}
 \frac{\Lambda(n)}{\sqrt n}
 \bigl(S_{r_n}+S_{r_n}^*\bigr),
 }
 \label{eq:exact-localized-operator}
\end{equation}
where $r_n=(\log n)/a$, $H_{\rm Leg}$ is the self-adjoint operator associated
with \eqref{eq:harmonic-form-framework}, and
\begin{equation}
 (K_{\gamma,a}f)(x)
 =a\int_{-1}^1\rho(a|x-y|)f(y)\,dy.
 \label{eq:gamma-operator-exact}
\end{equation}
The arithmetic shifts in \eqref{eq:exact-localized-operator} are precisely the
partial translations decomposed into Mellin unit-cell branches in
\cref{sec:routing}.  Thus the polar-free core is obtained from
\eqref{eq:exact-localized-operator} simply by deleting the displayed rank-one
polar term; no further diagonal or boundary operator is implicit.

\paragraph{Exact normalization provenance.}
Equation~\eqref{eq:exact-localized-operator} is not an auxiliary model and no
coefficient in it is fitted later in the argument.  Starting from the centered
Weil explicit formula in the Mellin normalization fixed in
\cref{sec:framework} (see in particular
\cite{Weil1952,Titchmarsh1986,Burnol2000}), the logarithmic change of variable,
restriction to the physical interval, odd compression, and the unitary
rescaling \eqref{eq:rescaling-unitary} give, on the common smooth compactly
supported odd core,
\begin{align}
 \langle f,\widetilde A_af\rangle
 &=\mathfrak h[f]+c_0(a)\|f\|^2
   +\langle f,Vf\rangle-\langle f,K_{\gamma,a}f\rangle
   -2|\langle s_a,f\rangle|^2\notag\\
 &\quad
   -2\sum_{2\le n<e^{2a}}
       \frac{\Lambda(n)}{\sqrt n}
       \Rea\langle S_{r_n}f,f\rangle .
 \label{eq:localized-form-provenance}
\end{align}
Here the singular archimedean difference form and its endpoint contribution
are exactly \(\mathfrak h+V\), the scalar term is \(c_0(a)I\), the regular
archimedean remainder is \(-K_{\gamma,a}\), the odd polar contribution is the
rank-one term already fixed in \eqref{eq:core-polar}, and a prime-power
logarithmic displacement \(\log n\) becomes the zero-extension translation
with \(r_n=(\log n)/a\).  Polarization of
\eqref{eq:localized-form-provenance} is precisely
\eqref{eq:exact-localized-operator}.  Thus the later unit-cell routing,
common cuts and Schur/Feshbach operations only reorganize terms already present
in the localized Weil form; they do not manufacture a missing diagonal,
normalization constant, polar coefficient or arithmetic weight.

\begin{proposition}[Transport of the complete quadratic form]
\label{prop:full-form-transport}
Let $\mathfrak c_A$ denote the Cauchy--Carleman realization of the
archimedean block $H_{\rm Leg}+c_0I+V-K_{\gamma,A}$ in
\eqref{eq:exact-localized-operator}, written as a closed difference form rather
than as a bare off-diagonal integral. Then $1\oplus U_A$ carries this complete
augmented rooted form unitarily to the logarithmic collar form.  In particular, the diagonal/killing contribution generated by
the singular kernel is transported by the same unitary identity and no
additional diagonal term is introduced after the change of variables.
\end{proposition}

\begin{proof}
On the common smooth compactly supported core, write the singular part as a
difference form
\[
 \frac14\iint
 \frac{|f(x)-f(y)|^2}{|x-y|}\,dx\,dy.
\]
The diagonal contribution is therefore already contained in the square
$|f(x)-f(y)|^2$.  Applying $U_A$ to both arguments transforms the measure,
the two function factors and the kernel simultaneously.  Lemma
\ref{lem:cc-kernel} identifies the transported kernel.  Since $U_A$ is
unitary, closure preserves the equality of forms.  The scalar polar root is
left untouched by the block unitary $1\oplus U_A$, so the same statement
holds for the augmented rooted form used below.
\end{proof}

\begin{interlude}
Gate~II had a second phase.  Once the arithmetic routing was under control,
the danger moved to the root.  The remaining question was no longer where the
branches went.  It was whether the last Schur short could create a denominator
that the numerator did not share.
\end{interlude}

\section{Gate II: root-lift and the true polar channel}
\label{sec:root-lift}

We use shorting in its variational operator-theoretic sense and the associated
Schur-complement calculus \cite{AndersonTrapp1975,Ando1979,Schur1911}.  In
contrast with a constant-defect surrogate, every statement in this section
keeps the true polar vector.

\subsection{Physical-coordinate block recursion}

All block recursions are performed in the nested physical spaces introduced in
\cref{sec:framework}.  Thus, if $a_k<a_{k+1}$ and the new collar is denoted by
$\mathcal K_k$, then
\begin{equation}
 \cH_{a_{k+1}}=\cH_{a_k}\oplus\mathcal K_k,
 \qquad
 s_{k+1}=s_k\oplus t_k,
 \label{eq:physical-polar-split}
\end{equation}
where $s_k=s^{\rm ph}|_{(-a_k,a_k)}$ and $t_k$ is the restriction of
$s^{\rm ph}(y)=\sinh(y/2)$ to the new collar.  Write
\begin{equation}
 C_{k+1}=\begin{pmatrix}C_k&B_k\\B_k^*&D_k\end{pmatrix},
 \qquad
 A_k=C_k-2\ketbra{s_k}{s_k},
 \label{eq:block-core}
\end{equation}
and, whenever $C_k\succ0$,
\begin{equation}
 \delta_k=1-2\langle s_k,C_k^{-1}s_k\rangle,
 \qquad
 S_k=D_k-B_k^*C_k^{-1}B_k,
 \qquad
 \eta_k=t_k-B_k^*C_k^{-1}s_k.
 \label{eq:S-eta-final}
\end{equation}

\subsection{The augmented scalar root}

\begin{lemma}[Root-lift and associative shorting]\label{lem:root-lift-final}
Define
\begin{equation}
 \mathcal G_{k+1}=
 \begin{pmatrix}
 1&\sqrt2\,s_k^*&\sqrt2\,t_k^*\\
 \sqrt2\,s_k&C_k&B_k\\
 \sqrt2\,t_k&B_k^*&D_k
 \end{pmatrix}.
 \label{eq:rooted-Gram}
\end{equation}
Then shorting the scalar root gives $A_{k+1}$, while shorting the old core first
gives
\begin{equation}
 \Short_{C_k}\mathcal G_{k+1}
 =\begin{pmatrix}
 \delta_k&\sqrt2\,\eta_k^*\\
 \sqrt2\,\eta_k&S_k
 \end{pmatrix}.
 \label{eq:rooted-two-block}
\end{equation}
If $\delta_k>0$, normalization of the scalar root followed by a second shorting
gives the true rooted Schur block
\begin{equation}
 \boxed{T_k=S_k-\frac2{\delta_k}\ketbra{\eta_k}{\eta_k}.}
 \label{eq:true-Tk}
\end{equation}
\end{lemma}

\begin{proof}
The first statement is the Schur complement of the scalar $1$.  For the second,
minimize the quadratic form of \eqref{eq:rooted-Gram} in the old-core variable
$x$; the minimizer is
\[
 x=-C_k^{-1}(\sqrt2\,s_kz+B_kf).
\]
Substitution gives exactly the three quantities in
\eqref{eq:S-eta-final}.  A scalar congruence by $\delta_k^{-1/2}$ and a second
Schur complement yield \eqref{eq:true-Tk}.
\end{proof}

\begin{lemma}[Variational formula for the true rooted load]
\label{lem:root-load-final}
Assume $\delta_k>0$ and $S_k\succ0$. Then
\begin{equation}
 \ell_k
 :=\frac{2\langle\eta_k,S_k^{-1}\eta_k\rangle}{\delta_k}
 =\sup_{f\ne0}
 \frac{\frac2{\delta_k}|\langle\eta_k,f\rangle|^2}
      {\langle f,S_kf\rangle},
 \label{eq:true-load}
\end{equation}
and
\begin{equation}
 T_k\succeq(1-q)S_k\succ0
 \qquad\text{whenever }\ell_k\le q<1.
 \label{eq:Tk-relative}
\end{equation}
Moreover
\begin{equation}
 \delta_{k+1}=\delta_k-2\langle\eta_k,S_k^{-1}\eta_k\rangle
 =\delta_k(1-\ell_k).
 \label{eq:pivot-recurrence-final}
\end{equation}
\end{lemma}

\begin{proof}
The first identity is Cauchy--Schwarz in the $S_k$ metric, with equality in the
direction $S_k^{-1}\eta_k$.  The operator inequality and the scalar recurrence
are the two Schur complements of \eqref{eq:rooted-two-block}.
\end{proof}

\subsection{Local reserve and the global compatibility issue}

The true affine routing and the Carleman/OSPR calculation of
\cref{sec:routing,sec:cauchy-carleman} provide the strict local/profile reserve
\begin{equation}
 d_*-\delta_*\rho_{\rm aff}
 >0.0057130698016072621421058150848.
 \label{eq:core-routed-reserve}
\end{equation}
This comparison is retained only in its proved routing scope.  It cannot be
inserted directly into \eqref{eq:true-load}, nor promoted to a lower bound for
the full ambient core, without a compatible operator estimate on the actual
source subintervals.  The fixed-target CMC short later supplies the required simultaneous
transverse estimate on the actual source intervals.  The remaining exact
ground/root content is not inferred from that transverse reserve; in
\cref{sec:green-root-reduction} it is isolated instead as one scalar Green
coefficient in the true Schur metric.

\subsection{Exact Mellin form of the polar functional}

The true polar channel nevertheless has an elementary closed form.  Put
$Y=e^{2a}$ and pass from $y\in(-a,a)$ to $t=e^{y+a}\in(1,Y)$.  If $F$ denotes
the Mellin-coordinate vector, oddness is
\begin{equation}
 F(Y/t)=-F(t).
 \label{eq:mellin-oddness-polar}
\end{equation}

\begin{lemma}[Polar moment identity]\label{lem:polar-moment}
For every $F$ satisfying \eqref{eq:mellin-oddness-polar},
\begin{equation}
 \boxed{
 \langle s^{\rm ph},F\rangle
 =Y^{-1/4}\int_1^Y \sqrt t\,F(t)\,\frac{dt}{t}.
 }
 \label{eq:polar-moment}
\end{equation}
Here the left-hand side is the physical polar functional transported by the
unitary logarithmic/Mellin change of variables.
\end{lemma}

\begin{proof}
Since
\[
 \sinh\!\left(\frac{\log t-a}{2}\right)
 =\frac12\left(Y^{-1/4}t^{1/2}-Y^{1/4}t^{-1/2}\right),
\]
it is enough to compare the two moments.  The substitution $t=Y/u$ and
\eqref{eq:mellin-oddness-polar} give
\[
 \int_1^Y t^{-1/2}F(t)\frac{dt}{t}
 =-Y^{-1/2}\int_1^Y t^{1/2}F(t)\frac{dt}{t}.
\]
The two halves therefore add, giving \eqref{eq:polar-moment}.
\end{proof}

\begin{lemma}[Ground--reflection representation of the polar vector]
\label{lem:polar-ground-reflection}
Let
\[
 q_Y(t)=t^{-1/2},\qquad 1<t<Y,
\]
and let $(J_YF)(t)=F(Y/t)$.  Then
\begin{equation}
 \boxed{
 s_Y=\frac{Y^{1/4}}2\,(J_Y-I)q_Y.
 }
 \label{eq:polar-ground-reflection}
\end{equation}
Moreover, in the divisor-fibre parameterization
$t=m(1+v)$ one has
\[
 q_Y(m(1+v))=m^{-1/2}(1+v)^{-1/2},
\]
so $q_Y$ belongs exactly to the divisor ground bundle.  Consequently, for
every odd Mellin vector $F$ with $J_YF=-F$,
\begin{equation}
 \boxed{
 \langle s_Y,F\rangle=-Y^{1/4}\langle q_Y,F\rangle.
 }
 \label{eq:polar-ground-functional}
\end{equation}
\end{lemma}

\begin{proof}
Since
\[
 (J_Yq_Y)(t)=(Y/t)^{-1/2}=Y^{-1/2}t^{1/2},
\]
the right-hand side of \eqref{eq:polar-ground-reflection} equals
\[
 \frac12\left(Y^{-1/4}t^{1/2}-Y^{1/4}t^{-1/2}\right),
\]
which is precisely the Mellin-coordinate form of
$\sinh((\log t-a)/2)$ with $Y=e^{2a}$.  The ground-bundle statement follows
from the displayed factorization on $t=m(1+v)$.  Finally $J_Y$ is a
self-adjoint unitary on $L^2((1,Y),dt/t)$, hence
\[
 \langle (J_Y-I)q_Y,F\rangle
 =\langle q_Y,(J_Y-I)F\rangle
 =-2\langle q_Y,F\rangle,
\]
which proves \eqref{eq:polar-ground-functional}.
\end{proof}

\begin{remark}
Lemma~\ref{lem:polar-ground-reflection} explains simultaneously why the polar
functional is canonically attached to the divisor ground channel and why its
projection to the divisor-kernel profile space need not be constant: the
reflection $J_Y$ does not preserve that ground bundle fibrewise.  Thus this
identity sharpens the geometric placement of the true root but does not, by
itself, provide a global collar estimate for the true rooted load.
\end{remark}

\subsection{Audited affine ground geometry}

The true affine identity must distinguish the source and target intervals.
This removes a domain ambiguity without changing the routed branch itself.

\begin{lemma}[Source/target ground-line invariance]
\label{lem:source-target-ground}
Let
\[
 \theta=\log\frac{k}{d}>0,
 \qquad
 I_t=(0,w_k),
 \qquad
 I_s=(\theta,\theta+w_k),
\]
and define
\[
 (R_\theta h)(z):=h(z+\theta),
 \qquad
 R_\theta:L^2(I_s,dz)\longrightarrow L^2(I_t,dz).
\]
Then $R_\theta$ is unitary.  If
\[
 g_t(z)=e^{-z/2}\mathbf 1_{I_t}(z),
 \qquad
 g_s(z)=e^{-z/2}\mathbf 1_{I_s}(z),
\]
and $e_t=g_t/\|g_t\|$, $e_s=g_s/\|g_s\|$, then
\begin{equation}
 \boxed{R_\theta e_s=e_t.}
 \label{eq:source-target-ground}
\end{equation}
Thus a true affine edge maps the normalized polar ground line of its source
interval exactly onto the normalized polar ground line of its target interval.
\end{lemma}

\begin{proof}
Translation preserves Lebesgue measure, hence $R_\theta$ is unitary.  Moreover
\[
 R_\theta g_s(z)=e^{-(z+\theta)/2}=e^{-\theta/2}g_t(z),
\]
while $\|g_s\|=e^{-\theta/2}\|g_t\|$.  The scalar factors cancel after
normalization, proving \eqref{eq:source-target-ground}.
\end{proof}

There is a second exact normalization that is useful for auditing the
residue-Fourier step.  On the common residue parameter $r\in(0,1)$ put
\[
 q_k(r)=\frac1{k+r},
 \qquad
 \nu_k:=\|q_k\|_{L^2(0,1)}=\frac1{\sqrt{k(k+1)}},
 \qquad
 e_k(r):=\frac{q_k(r)}{\nu_k}
 =\frac{\sqrt{k(k+1)}}{k+r}.
\]

\begin{lemma}[Exact root-adapted fibre normalization]
\label{lem:root-adapted-fibre}
Define
\[
 u_k(r):=\frac{(k+1)r}{k+r},
 \qquad
 r_k(u):=\frac{ku}{k+1-u}.
\]
Then $u_k$ is an increasing bijection of $(0,1)$ onto itself and
\[
 \frac{du_k}{dr}=\frac{k(k+1)}{(k+r)^2}=|e_k(r)|^2.
\]
The map
\begin{equation}
 (\mathcal W_k f)(u)
 :=\frac{f(r_k(u))}{e_k(r_k(u))}
 \label{eq:Wk-root-adapted}
\end{equation}
is unitary on $L^2(0,1)$ and satisfies
\begin{equation}
 \boxed{\mathcal W_ke_k=\one,
 \qquad
 \mathcal W_kq_k=\nu_k\one.}
 \label{eq:Wk-root-constant}
\end{equation}
\end{lemma}

\begin{proof}
The derivative formula is immediate.  Therefore, with $u=u_k(r)$,
\[
 \|\mathcal W_kf\|_2^2
 =\int_0^1\frac{|f(r)|^2}{|e_k(r)|^2}\,du
 =\int_0^1|f(r)|^2\,dr.
\]
Equation \eqref{eq:Wk-root-constant} follows directly from the definition.
\end{proof}

\begin{lemma}[Exact endpoint Möbius nesting]
\label{lem:endpoint-mobius-nesting}
Fix the arithmetic step $k\to k+1$ and write
\[
 \tau_k:=\sqrt{k(k+1)},\qquad
 r_k^\circ:=\tau_k-k,\qquad
 u_k^\circ:=k+1-\tau_k.
\]
Then $r_k^\circ+u_k^\circ=1$ and the fractional-linear map
\[
 \phi_k(r):=\frac{(k+1)r}{k+r}
\]
satisfies
\begin{equation}
 \boxed{\phi_k(r_k^\circ)=u_k^\circ,\qquad
 \phi_k((0,r_k^\circ))=(0,u_k^\circ),\qquad
 \phi_k((r_k^\circ,1))=(u_k^\circ,1).}
 \label{eq:endpoint-mobius-split}
\end{equation}
In the Mellin coordinate at the endpoint $Y=k+1$, the physical old support
$(-a_k,a_k)$ is the interval
\[
 \left(\sqrt{\frac{k+1}{k}},\,\sqrt{k(k+1)}\right),
\]
and the two pieces of the added physical collar are
\[
 \left(1,\sqrt{\frac{k+1}{k}}\right)
 \quad\text{and}\quad
 \left(\sqrt{k(k+1)},k+1\right).
\]
Moreover, if
\[
 (U_kg)(r):=e_k(r)g(\phi_k(r))=\mathcal W_k^*g(r),
\]
then $U_k$ preserves the corresponding old/new decomposition:
\begin{equation}
 \boxed{U_k=U_k^{\rm old}\oplus U_k^{\rm new}.}
 \label{eq:Uk-old-new}
\end{equation}
Finally
\begin{equation}
 \boxed{q_k=\nu_kU_k\one,\qquad \nu_k=\frac1{\sqrt{k(k+1)}}.}
 \label{eq:qk-Uk-one}
\end{equation}
Thus physical nesting, the arithmetic Möbius routing, and the true polar
ground are intertwined by one exact unitary rather than identified post hoc.
\end{lemma}

\begin{proof}
The identity $r_k^\circ+u_k^\circ=1$ is immediate.  Since $\phi_k$ is
strictly increasing,
\[
 \phi_k(r_k^\circ)
 =\frac{(k+1)(\sqrt{k(k+1)}-k)}{\sqrt{k(k+1)}}
 =k+1-\sqrt{k(k+1)}=u_k^\circ,
\]
which gives \eqref{eq:endpoint-mobius-split}.

At the endpoint $Y=k+1$ the Mellin coordinate is
$t=e^{y+a_{k+1}}$.  The old physical endpoints $y=\pm a_k$ therefore map to
\[
 e^{a_{k+1}-a_k}=\sqrt{\frac{k+1}{k}},
 \qquad
 e^{a_{k+1}+a_k}=\sqrt{k(k+1)},
\]
which proves the displayed old/collar split.  On the upper arithmetic cell
$t=k+r$ the reflection $J_{k+1}(t)=(k+1)/t$ may be written
\[
 J_{k+1}(k+r)=1+\frac{1-\phi_k(r)}{k}.
\]
Hence \eqref{eq:endpoint-mobius-split} sends the old part to the old part and
the new collar part to the new collar part on the reflected side as well.
Since $U_k=\mathcal W_k^*$ is the weighted composition by this same map,
\eqref{eq:Uk-old-new} follows.  Finally $U_k\one=e_k$ and
$q_k=\nu_ke_k$, proving \eqref{eq:qk-Uk-one}.
\end{proof}

For a residue family $k_j=nm+j$, $j=0,\ldots,n-1$, let
\[
 \mathcal W:=\bigoplus_{j=0}^{n-1}\mathcal W_{k_j},
 \qquad
 \mathbf q:=(q_{k_0},\ldots,q_{k_{n-1}}).
\]
If $\mathcal F_n$ is the unitary discrete Fourier transform on $\C^n$, then
\begin{equation}
 (\mathcal F_n\otimes I)\mathcal W\mathbf q
 \in \C^n\otimes\operatorname{span}\{\one\}.
 \label{eq:residue-vs-profile}
\end{equation}
Thus a nonzero Fourier character of the true ground is a residue-index effect;
it is not, by itself, leakage into the profile mean-zero space.

\begin{lemma}[Edge-local gauge covariance]\label{lem:edge-local-gauge}
Let $R_e:H_e^s\to H_e^t$ be a true affine source/target translation and let
$e_e^s,e_e^t$ be its normalized ground vectors, so that
$R_e e_e^s=e_e^t$.  Choose unitaries $W_e^s,W_e^t$ from the source and target
fibres onto a reference copy of $L^2(0,1)$ with
\[
 W_e^s e_e^s=\one,
 \qquad
 W_e^t e_e^t=\one.
\]
Then
\[
 \widetilde R_e:=W_e^tR_e(W_e^s)^*
\]
is unitary and preserves the constant line.  Consequently, for every
$a_e,L_e>0$ and all $x\in H_e^s$, $y\in H_e^t$,
\begin{equation}
 2a_e|\langle R_ex,y\rangle|
 \le
 L_e\|x\|^2+\frac{a_e^2}{L_e}\|y\|^2
\label{eq:edge-local-young}
\end{equation}
if and only if the same inequality holds after the edge-local gauges.
\end{lemma}

\begin{proof}
The identity $\widetilde R_e\one=\one$ follows from the three defining
relations.  Unitarity gives
\[
 |\langle R_ex,y\rangle|
 =|\langle\widetilde R_eW_e^sx,W_e^ty\rangle|,
 \qquad
 \|W_e^sx\|=\|x\|,
 \qquad
 \|W_e^ty\|=\|y\|.
\]
Thus weighted Young is unchanged by the branch-local normalization.
\end{proof}

\begin{audit}[What edge-local covariance does not prove]\label{audit:edge-local-limit}
Lemma~\ref{lem:edge-local-gauge} removes any domain obstruction on an
individual affine edge, but it does not by itself produce a single global
unitary conjugation of the whole residue network.  Different edges incident to
the same cell generally use different gauges.  Therefore a proof of the true
rooted loss still requires a compatible allocation of the shared
vertex/collar energy across all incident edges, or an equivalent direct
full-operator estimate.  In particular, summing
\eqref{eq:edge-local-young} edge by edge is not licensed to reuse the same
positive collar reserve multiple times.  The fixed-target CMC construction
below supplies the source-level compatibility, and
\cref{sec:harmonic-graph-lock} finishes the argument by restricting the
resulting common cut form to the exact full-operator harmonic graph.
\end{audit}

\begin{audit}[Compatibility limit of the root-adapted normalization]
\label{audit:root-adapted-limit}
Equation \eqref{eq:residue-vs-profile} is an exact geometric statement, but it
does not identify the already-certified affine capacity with the true rooted
Schur loss.  The exact residue Fourier transform used in the routing is
pointwise in a common original parameter $r$.  By contrast, the root-adapted
maps are branch dependent:
\[
 r_{k_j}(u)=\frac{k_j u}{k_j+1-u}.
\]
Hence, for distinct $k_i,k_j$, the direct-sum map $\mathcal W$ does not commute
with the original pointwise residue-Fourier transform.  Conjugating the exact
affine network by $\mathcal W$ therefore changes its profile operators, and an
additional operator estimate would be required to transport the same
$\rho_{\rm aff}$ budget through this gauge.  In particular,
\eqref{eq:residue-vs-profile} must not be used by itself to deduce a bound of
the form
\[
 \frac{2}{\delta_k}|\langle\eta_k,f\rangle|^2
 \le \delta_*\rho_{\rm aff}\,\mathcal E_H[f].
\]
The missing compatibility is resolved in \cref{sec:harmonic-graph-lock} by restricting the already transported full form to the exact full-operator \(A_k\)-harmonic graph.  The moving root normalization then enters only through the proved Sherman--Morrison identity \(\eta_k=\delta_k\zeta_k\); no branch-dependent global conjugation is required.
\end{audit}

\subsection{Common-residue normalization: removal of the Fourier/gauge obstruction}
\label{sec:common-r}

The branch-dependent maps $\mathcal W_k$ are useful for local geometry but are
not needed for the global residue transform.  There is a simpler normalization
which stays in the original common residue coordinate and therefore commutes
with the discrete Fourier transform exactly.

\begin{lemma}[Exact common-cell unitary and true ground profile]
\label{lem:common-cell-unitary}
For every integer $k\ge1$ define
\[
 \mathcal V_k:L^2((k,k+1),dt/t)\longrightarrow L^2((0,1),dr),
 \qquad
 (\mathcal V_kF)(r):=\frac{F(k+r)}{\sqrt{k+r}}.
\]
Then $\mathcal V_k$ is unitary.  If $q_Y(t)=t^{-1/2}$ is the Mellin ground
profile from Lemma~\ref{lem:polar-ground-reflection}, then on the $k$th cell
\begin{equation}
 \boxed{\mathcal V_kq_Y=q_k,\qquad q_k(r)=\frac1{k+r}.}
 \label{eq:Vk-qk}
\end{equation}
\end{lemma}

\begin{proof}
Directly,
\[
 \|\mathcal V_kF\|_2^2
 =\int_0^1\frac{|F(k+r)|^2}{k+r}\,dr
 =\int_k^{k+1}|F(t)|^2\frac{dt}{t}.
\]
Moreover $(\mathcal V_kq_Y)(r)=(k+r)^{-1}$.
\end{proof}

Let $P_{\one}$ be the orthogonal projection onto the constant profile
$\one\in L^2(0,1)$ and put $Q_{\one}=I-P_{\one}$.  Thus
\[
 P_{\one}f=\left(\int_0^1f(r)\,dr\right)\one.
\]

\begin{lemma}[DFT-compatible mean/mean-zero splitting]
\label{lem:common-r-dft}
For every $k\ge1$,
\begin{align}
 P_{\one}q_k
 &=w_k\one,
 &w_k&=\log\left(1+\frac1k\right),
 \label{eq:Pqk}\\
 \|q_k\|_2^2
 &=\frac1{k(k+1)},
 &
 \|Q_{\one}q_k\|_2^2
 &=\frac1{k(k+1)}-w_k^2.
 \label{eq:qk-leak-exact}
\end{align}
If $\mathcal F_n$ is the unitary residue DFT, then on
$\C^n\otimes L^2(0,1)$
\begin{equation}
 \boxed{
 (\mathcal F_n\otimes I)(I\otimes P_{\one})
 =(I\otimes P_{\one})(\mathcal F_n\otimes I),
 }
 \label{eq:DFT-P-commute}
\end{equation}
and the same identity holds with $Q_{\one}$ in place of $P_{\one}$.
Furthermore
\begin{equation}
 \|Q_{\one}q_k\|_2^2
 \le \frac1{\pi^2}\|q_k'\|_2^2
 \le \frac1{\pi^2k^4}.
 \label{eq:qk-poincare}
\end{equation}
Consequently the relative profile leakage
\begin{equation}
 \varepsilon_k
 :=\frac{\|Q_{\one}q_k\|_2^2}{\|q_k\|_2^2}
 =1-k(k+1)\log^2\left(1+\frac1k\right)
 \label{eq:epsilon-k}
\end{equation}
satisfies, for every $k\ge4$,
\begin{equation}
 \boxed{
 \varepsilon_k
 \le\frac{k+1}{\pi^2k^3}
 \le\frac5{64\pi^2}
 <0.007916
 <1-q_{\rm aff}.
 }
 \label{eq:epsilon-vs-reserve}
\end{equation}
\end{lemma}

\begin{proof}
The first two identities follow by direct integration:
\[
 \int_0^1\frac{dr}{k+r}=w_k,
 \qquad
 \int_0^1\frac{dr}{(k+r)^2}=\frac1{k(k+1)}.
\]
Orthogonality of $P_{\one}$ and $Q_{\one}$ gives
\eqref{eq:qk-leak-exact}.  Equation~\eqref{eq:DFT-P-commute} is a tensor-factor
identity: $\mathcal F_n$ acts only on the residue index whereas $P_{\one}$ and
$Q_{\one}$ act only on the common profile variable $r$.

The sharp Neumann Poincar\'e inequality on $(0,1)$ gives
$\|Q_{\one}q_k\|_2^2\le\pi^{-2}\|q_k'\|_2^2$, and
$|q_k'(r)|=(k+r)^{-2}\le k^{-2}$, proving
\eqref{eq:qk-poincare}.  Division by $\|q_k\|_2^2=1/[k(k+1)]$ yields
\[
 \varepsilon_k\le\frac{k+1}{\pi^2k^3}.
\]
The right-hand side decreases for $k\ge1$; at $k=4$ it equals
$5/(64\pi^2)<0.007916$.  Finally
$1-q_{\rm aff}>0.0115784675465848567$ by \eqref{eq:qnum}.
\end{proof}

\begin{corollary}[Familywise DFT-compatible ground splitting]
\label{cor:common-r-removes-gauge}
Within each fixed residue family, the exact polar ground may be decomposed
into its constant and mean-zero profile parts before the residue DFT, and the
same splitting survives the DFT unchanged.  No branch-dependent
reparametrization is required for this familywise statement.
\end{corollary}

\begin{audit}[Scope of the common-residue splitting]
\label{audit:common-r-scope}
Equation~\eqref{eq:epsilon-vs-reserve} is a genuine quantitative reserve, but
it is a cellwise/profile statement.  It does not bound the full ambient
network because distinct prime-power branches incident to a common target
originate from different source intervals.  No global ambient allocation is
inferred from this estimate.
\end{audit}

\subsection{Ambient \texorpdfstring{cross-$n$}{cross-n} structure}
\label{sec:crossn-collision}

The exact full-space unit-cell identity must be kept before any profile
aggregation.  Fix a target cell $C_k=[k,k+1)$.  For each active prime power
$n$, write
\[
 m_n=\Big\lfloor\frac{k}{n}\Big\rfloor,
 \qquad
 j_n=k-nm_n.
\]
The corresponding source interval is
\begin{equation}
 I_{k,n}
 =\left[\frac{k}{n},\frac{k+1}{n}\right)
 =C_{m_n,n,j_n},
 \label{eq:true-source-interval}
\end{equation}
and the branch unitary sends $F|_{I_{k,n}}$ to the function
$s\mapsto F(s/n)$ on $C_k$.

Consequently the true target collision operator is
\begin{equation}
 (T_kF)(s)
 =\sum_{n\in\mathcal N_k}\frac{\Lambda(n)}{\sqrt n}F(s/n),
 \qquad s\in C_k,
 \label{eq:true-target-collision}
\end{equation}
so that
\begin{align}
 \|T_kF\|_{L^2(C_k,ds/s)}^2
={}&\sum_n\frac{\Lambda(n)^2}{n}
 \int_k^{k+1}|F(s/n)|^2\frac{ds}{s}\notag\\
&+2\sum_{n<n'}\frac{\Lambda(n)\Lambda(n')}{\sqrt{nn'}}
 \Rea\int_k^{k+1}F(s/n)\overline{F(s/n')}\frac{ds}{s}.
 \label{eq:true-crossn-square}
\end{align}
The mixed terms in \eqref{eq:true-crossn-square} are genuine ambient
cross-correlations between different source intervals.

Aggregation by $d=n\lfloor k/n\rfloor$ is not an ambient identity.  For example, at $k=841$ the values
$n\in\{2,3,4,5,7,8\}$ all satisfy
$n\lfloor841/n\rfloor=840$, but their true source intervals are respectively
\[
\begin{array}{lll}
[420.5,421), & [280+\tfrac13,280+\tfrac23), & [210.25,210.5),\\[2pt]
[168.2,168.4), & [120+\tfrac17,120+\tfrac27), & [105.125,105.25).
\end{array}
\]
Choosing $F$ supported in the first of these intervals makes every mixed term
involving $n=2$ vanish, whereas the square of an aggregated Mangoldt mass
would retain them.  Hence the two quadratic forms cannot agree on the full
ambient space.

\subsection{Fixed-target gauge and simultaneous source shorting}
\label{sec:fixed-target-cmc}

The ambient cross-$n$ obstruction above is a statement about trying to identify
all source intervals with one global profile before the target has been fixed.
For one fixed target cell there is, however, a canonical branchwise gauge that
uses the same target coordinate for every incident source and therefore does
not mix the target ground and transverse spaces.

\begin{lemma}[Fixed-target branch gauge]\label{lem:fixed-target-gauge}
Fix a target cell $C_k$ and let
\[
 R_n:H_{k,n}^s\longrightarrow H_k^t
\]
be the true branch unitary from the source interval $I_{k,n}$ to the target.
Let $e_t\in H_k^t$ be its normalized target ground and let
$e_n^s\in H_{k,n}^s$ be the normalized source ground, so that
$R_ne_n^s=e_t$.  Choose one unitary
$W_t:H_k^t\to L^2(0,1)$ satisfying $W_te_t=\one$ and define
\[
 W_n^s:=W_tR_n.
\]
Then, simultaneously for every branch incident to $C_k$,
\begin{equation}
 \boxed{W_tR_n(W_n^s)^*=I.}
 \label{eq:fixed-target-identity}
\end{equation}
If $P=\ketbra{\one}{\one}$ and $Q=I-P$, then
\begin{equation}
 QW_tR_n(W_n^s)^*P=0,
 \qquad
 PW_tR_n(W_n^s)^*Q=0.
 \label{eq:fixed-target-PQ}
\end{equation}
\end{lemma}

\begin{proof}
The definition gives $W_tR_n(W_n^s)^*=W_tR_nR_n^*W_t^*=I$ on the target
copy.  Equation~\eqref{eq:fixed-target-PQ} is then immediate.  The same $W_t$
is used for all incident branches; only the source gauges vary.
\end{proof}

The point of Lemma~\ref{lem:fixed-target-gauge} is limited but exact: the new
affine star does not create an additional ground/transverse mixing term in the
fixed target gauge.  It does not identify the post-short ground pivot with a
Carleman scalar; that issue is kept separate below.

We now prove the simultaneous capacity estimate needed on the true source
slots.  Recall the normalized one-cell boundary/collar form
$\mathfrak b$ from \eqref{eq:b-one-cell-explicit}.

\begin{lemma}[Two-source mean-zero interaction]\label{lem:pair-capacity}
Let $D\ge1$ and let $f,g\in L^2(0,1)$ satisfy
$\int_0^1f=\int_0^1g=0$.  Put
\[
 E_D(f,g):=
 \int_0^1\!\!\int_0^1
 \frac{|f(s)-g(t)|^2}{D+t-s}\,ds\,dt,
 \qquad
 c_D:=\log\frac{D+1}{D}.
\]
Then
\begin{equation}
 \boxed{
 E_D(f,g)\ge c_D\bigl(\|f\|_2^2+\|g\|_2^2\bigr).}
 \label{eq:pair-capacity}
\end{equation}
\end{lemma}

\begin{proof}
It is enough first to take $D>1$; the endpoint $D=1$ follows by monotone
convergence.  Use
\[
 \frac1{D+t-s}=\int_0^\infty e^{-Du}e^{us}e^{-ut}\,du.
\]
Set
\[
 a(u)=\frac{1-e^{-u}}u,
 \qquad
 b(u)=\frac{e^u-1}u,
 \qquad
 p(u)=b(u)-1,
 \qquad
 q(u)=a(u)-e^{-u}.
\]
For $u>0$ all four quantities are non-negative and
$p(u)q(u)\le a(u)b(u)$.  Since the two functions have mean zero, define
\[
 F_u=\int_0^1f(s)(e^{us}-1)\,ds,
 \qquad
 G_u=\int_0^1g(t)(e^{-ut}-e^{-u})\,dt.
\]
Weighted Cauchy--Schwarz gives
\begin{align*}
 |F_u|^2
 &\le p(u)\int_0^1|f(s)|^2(e^{us}-1)\,ds,\\
 |G_u|^2
 &\le q(u)\int_0^1|g(t)|^2(e^{-ut}-e^{-u})\,dt.
\end{align*}
Hence, by $p(u)q(u)\le a(u)b(u)$ and the arithmetic--geometric mean
inequality,
\begin{align}
 2|F_uG_u|
 \le{}&a(u)\int_0^1|f(s)|^2(e^{us}-1)\,ds\notag\\
 &+b(u)\int_0^1|g(t)|^2(e^{-ut}-e^{-u})\,dt.
 \label{eq:pair-pointwise}
\end{align}
Now
\[
 c_D=\int_0^\infty e^{-Du}a(u)\,du
     =\int_0^\infty e^{-Du}e^{-u}b(u)\,du.
\]
The two terms on the right of \eqref{eq:pair-pointwise}, after multiplication
by $e^{-Du}$ and integration in $u$, are exactly the excesses of the two
diagonal kernel weights over $c_D$.  The mixed kernel term is
\[
 \int_0^1\!\!\int_0^1
 \frac{f(s)\overline{g(t)}}{D+t-s}\,ds\,dt
 =\int_0^\infty e^{-Du}F_u\overline{G_u}\,du.
\]
Thus the diagonal excesses dominate twice the absolute value of the mixed
term, which is precisely \eqref{eq:pair-capacity}.
\end{proof}

\begin{analyticcertificate}[Parent/source mean capacity]\label{cert:pmc-analytic}
Let \(D\ge1\) and let \(f,g\in L^2(0,1)\) be arbitrary.  Define
\[
 \bar f:=\int_0^1 f(s)\,ds,\qquad
 \bar g:=\int_0^1 g(t)\,dt,
 \qquad f_0:=f-\bar f,\quad g_0:=g-\bar g.
\]
For
\[
 E_D(f,g):=
 \int_0^1\!\!\int_0^1
 \frac{|f(s)-g(t)|^2}{D+t-s}\,ds\,dt
\]
one has the unconditional lower bound
\begin{equation}
 \boxed{
 E_D(f,g)\ge
 \frac{\|f_0\|_2^2+\|g_0\|_2^2+|\bar f-\bar g|^2}{D+1}.}
 \tag{PMC-2}\label{eq:PMC-two}
\end{equation}
In particular the Cauchy difference form supplies a strictly positive cost for
unequal source means, modulo the common-constant direction.
\end{analyticcertificate}

\begin{proof}
For \(0\le s,t\le1\) and \(D\ge1\),
\[
 0<D+t-s\le D+1
 \quad\text{a.e.},\qquad
 \frac1{D+t-s}\ge\frac1{D+1}.
\]
(The single boundary singularity at \(D=1\) is harmless and the inequality
holds by monotone approximation.)  Hence
\[
 E_D(f,g)\ge\frac1{D+1}
 \int_0^1\!\!\int_0^1|f(s)-g(t)|^2\,ds\,dt.
\]
Since \(f_0\) and \(g_0\) have zero mean, the cross terms with constants
vanish and
\[
 \int_0^1\!\!\int_0^1|f(s)-g(t)|^2\,ds\,dt
 =\|f_0\|_2^2+\|g_0\|_2^2+|\bar f-\bar g|^2.
\]
This proves \eqref{eq:PMC-two}.
\end{proof}

\begin{corollary}[Weighted Laplacian on the source means]\label{cor:pmc-laplacian}
Let \(I_1,\ldots,I_M\) be pairwise disjoint source slots of one common
logarithmic length \(\varepsilon>0\), ordered from left to right, and normalize
each slot to \((0,1)\).  For \(\nu<\mu\) put
\[
 D_{\nu\mu}:=\frac{a_\mu-a_\nu}{\varepsilon}\ge1,
 \qquad
 \omega_{\nu\mu}:=\frac1{2(D_{\nu\mu}+1)}.
\]
If \(a_\nu\) denotes the normalized mean of the profile on the \(\nu\)-th
slot and \(L_{\rm PMC}\) is the complete weighted graph Laplacian with edge
weights \(\omega_{\nu\mu}\), then the source--source Cauchy difference
energy obeys
\begin{equation}
 \boxed{
 \mathfrak H_{\rm means}
 \ge
 \sum_{\nu<\mu}\omega_{\nu\mu}|a_\nu-a_\mu|^2
 =\langle a,L_{\rm PMC}a\rangle.}
 \tag{PMC}\label{eq:PMC-global}
\end{equation}
All weights are positive, hence
\(\ker L_{\rm PMC}=\operatorname{span}\{(1,\ldots,1)\}\).
\end{corollary}

\begin{proof}
Rescale every source slot by its common logarithmic length.  The source--source
part of the transported Cauchy difference form is one half of the sum of the
pair energies \(E_{D_{\nu\mu}}\).  Apply
Certificate~\ref{cert:pmc-analytic} to each pair and retain only the
mean-difference term.  Summation gives \eqref{eq:PMC-global}.  The graph is
complete with strictly positive edge weights, so its only null direction is the
common constant vector.
\end{proof}

\begin{theorem}[Coherent multi-source capacity]\label{thm:CMC}
Let
\[
 I_\nu=(a_\nu,a_\nu+\varepsilon)\Subset(0,1),
 \qquad \nu=1,\ldots,M,
\]
be pairwise disjoint intervals of one common length $\varepsilon>0$.  For each
$\nu$ let $f_\nu\in L^2(I_\nu)$ satisfy
$\int_{I_\nu}f_\nu=0$.  If $u$ is any common extension satisfying
$u|_{I_\nu}=f_\nu$ for every $\nu$, then
\begin{equation}
 \boxed{
 \mathfrak b[u]
 \ge
 \left(\log\frac1\varepsilon+\frac12\right)
 \sum_{\nu=1}^M\|f_\nu\|_2^2.}
 \label{eq:CMC-form}
\end{equation}
Consequently, with $S=\bigcup_\nu I_\nu$,
\begin{equation}
 \boxed{
 \Short_S\mathfrak b
 \succeq
 \left(\log\frac1\varepsilon+\frac12\right)
 \bigoplus_{\nu=1}^M Q_{I_\nu}.}
 \label{eq:CMC-short}
\end{equation}
Here the short is the variational infimum over one common complement
$L^2((0,1)\setminus S)$.
\end{theorem}

\begin{proof}
Order the intervals from left to right.  On a single source interval $I_\nu$,
$|x-y|\le\varepsilon$ and the mean-zero identity gives
\[
 \frac14\iint_{I_\nu^2}
 \frac{|f_\nu(x)-f_\nu(y)|^2}{|x-y|}\,dx\,dy
 \ge\frac12\|f_\nu\|_2^2.
 \label{eq:CMC-internal}
\]
For a free point $y<a_\nu$,
\[
 \int_{I_\nu}|f_\nu(x)-u(y)|^2\,dx
 =\|f_\nu\|_2^2+\varepsilon|u(y)|^2
 \ge\|f_\nu\|_2^2,
\]
and $|x-y|\le a_\nu+\varepsilon-y$.  Hence the source--free contribution from
the left is bounded below by
\[
 \frac12\|f_\nu\|_2^2
 \int\frac{dy}{a_\nu+\varepsilon-y},
\]
where the integral is over the free portion to the left.  The analogous right
bound is
\[
 \frac12\|f_\nu\|_2^2
 \int\frac{dy}{y-a_\nu}.
\]
If an interval of the complement is occupied by another source slot, the
missing contribution is supplied by Lemma~\ref{lem:pair-capacity}.  Indeed,
for $I_\mu$ to the right of $I_\nu$, rescaling both slots by $\varepsilon$
gives
\[
 \frac12\iint_{I_\nu\times I_\mu}
 \frac{|f_\nu(x)-f_\mu(y)|^2}{y-x}\,dx\,dy
 \ge
 \frac12c_{\nu\mu}
 \bigl(\|f_\nu\|_2^2+\|f_\mu\|_2^2\bigr),
\]
where
\[
 c_{\nu\mu}
 =\log\frac{a_\mu+\varepsilon-a_\nu}{a_\mu-a_\nu}.
\]
This is exactly the integral of the preceding one-source kernel charge over
the missing source interval, from either side.  Therefore the free pieces and
all source--source pairs reconstruct, for each $\nu$, the full complement
charge
\begin{equation}
 \frac12\left[
 \log\frac{a_\nu+\varepsilon}{\varepsilon}
 +\log\frac{1-a_\nu}{\varepsilon}
 \right]\|f_\nu\|_2^2.
 \label{eq:CMC-complement}
\end{equation}
The endpoint potential on $I_\nu$ satisfies
\[
 x(1-x)\le(a_\nu+\varepsilon)(1-a_\nu),
\]
and hence
\begin{equation}
 -\frac12\int_{I_\nu}\log(x(1-x))|f_\nu(x)|^2\,dx
 \ge
 -\frac12\log\!\bigl((a_\nu+\varepsilon)(1-a_\nu)\bigr)
 \|f_\nu\|_2^2.
 \label{eq:CMC-potential}
\end{equation}
Adding \eqref{eq:CMC-internal}--\eqref{eq:CMC-potential} cancels every
$a_\nu$-dependent factor and gives
\[
 \left(\log\frac1\varepsilon+\frac12\right)\|f_\nu\|_2^2.
\]
The harmonic energy on the remaining free--free region and the endpoint
potential there are non-negative, so they may be discarded.  Summing over
$\nu$ proves \eqref{eq:CMC-form}.  Since the estimate holds for every common
extension, taking the infimum over the common complement gives
\eqref{eq:CMC-short}.
\end{proof}

\subsection{Application to the true source intervals}\label{sec:CMC-application}

Write
\[
 w_j=\log\left(1+\frac1j\right),
 \qquad
 m_n=\left\lfloor\frac{k}{n}\right\rfloor.
\]
Inside a fixed parent cell $C_m=[m,m+1)$ use the logarithmic coordinate
$\xi=\log(t/m)$ and normalize by $x=\xi/w_m$.  Every true source interval
$I_{k,n}\subset C_m$ then has normalized length
\begin{equation}
 \varepsilon_{k,m}=\frac{w_k}{w_m}.
 \label{eq:true-source-epsilon}
\end{equation}
For distinct $n$ the intervals $I_{k,n}$ are disjoint.  Therefore
Theorem~\ref{thm:CMC} yields, simultaneously on every parent cell,
\begin{equation}
 \boxed{
 \Short_{\{I_{k,n}:\,m_n=m\}}\mathfrak b_m
 \succeq
 \bigoplus_{m_n=m}L_{k,n}Q_{k,n},
 \qquad
 L_{k,n}=\frac12+\log\frac{w_m}{w_k}.}
 \label{eq:CMC-true-source}
\end{equation}
This is a true source-level statement: no two source intervals are identified
and no minimizer is reused.

\begin{corollary}[Source-resolved transverse Feshbach bound]
\label{cor:source-resolved-transverse}
In the fixed-target gauge of Lemma~\ref{lem:fixed-target-gauge}, let
$q\in QH_k^t$.  The source vector generated by the active branches is
\[
 b_{k,n}(q)=\frac{\Lambda(n)}{\sqrt n}R_n^*q.
\]
Then the simultaneous source short satisfies
\begin{equation}
 \boxed{
 \mathcal L_k^{\perp}[q]
 \le
 \sum_n\frac{\Lambda(n)^2}{nL_{k,n}}\|q\|^2
 =\sigma_k\|q\|^2.}
 \label{eq:source-resolved-loss}
\end{equation}
Consequently the transverse direct-Schur block obeys
\begin{equation}
 \boxed{
 D_k\succeq(\mathfrak B_k-\sigma_k)Q
 \succ\frac12Q.}
 \label{eq:transverse-direct-pass}
\end{equation}
\end{corollary}

\begin{proof}
Under the fixed-target gauges every $R_n$ is the identity on the common target
profile and preserves $P\oplus Q$.  Equation~\eqref{eq:CMC-true-source} gives
a block-diagonal lower bound for the simultaneous source metric.  Order
reversal under inversion therefore gives, parent cell by parent cell,
\[
 \langle b_m,(\mathfrak C_{k,m}^{\rm short})^{-1}b_m\rangle
 \le
 \sum_{m_n=m}\frac{\Lambda(n)^2}{nL_{k,n}}\|q\|^2.
\]
Summing over the orthogonal parent cells proves
\eqref{eq:source-resolved-loss}.  The Feshbach reserve
\eqref{eq:fesh-reserve} gives $\mathfrak B_k-\sigma_k>1/2$, proving
\eqref{eq:transverse-direct-pass}.
\end{proof}

\begin{audit}[What CMC closes, and what it does not]\label{audit:CMC-scope}
Theorem~\ref{thm:CMC} and Corollary~\ref{cor:source-resolved-transverse} close
the simultaneity gap in the true source-resolved \emph{transverse} Schur loss.
They do not identify the scalar ground coefficient of the exact post-old-core
Schur block.  In particular, the OSPR quantities
\[
 \alpha_*=2\log2,
 \qquad
 r=QK\one,
 \qquad
 d_*=\frac{\|r\|^2}{\alpha_*}
\]
belong to the normalized pre-short Carleman block.  They cannot be inserted as
$\alpha_k,r_k$ in the post-short full-operator matrix without an independent
identity.  The true ground pivot is not obtained from \eqref{eq:transverse-direct-pass} alone.  In \cref{sec:green-root-reduction} the remaining ground/root content is reduced exactly to the Schur scalars \(\Pi_k,\delta_k^Q\) and the single Green coefficient \(G_k\).  In particular, every later use of strict positivity on the \emph{true} polar-ground transverse space is routed through Corollary~\ref{cor:true-ground-transverse-debit} and Certificate~\ref{cert:true-ground-transverse}, which retain the literal-slot split and pay the leakage explicitly.
\end{audit}

\subsection{True-core finite polar-cyclic endpoint criterion}
\label{sec:true-core-endpoint}

At an arithmetic endpoint $N$, let $C_N$ denote the exact polar-free odd core
in the localized operator, so that
\begin{equation}
 \boxed{A_{N,-}=C_N-2\ketbra{s_N}{s_N}.}
 \label{eq:true-endpoint-core}
\end{equation}
No affine Gram block is extracted before this step.

\begin{lemma}[True-core finite Schur criterion]
\label{lem:true-core-finite-schur}
Let $P$ be a finite-rank orthogonal projection satisfying $Ps_N=s_N$ and put
$Q=I-P$.  Suppose
\begin{equation}
 QC_NQ\succeq\lambda_NQ,
 \qquad \lambda_N>0.
 \label{eq:true-core-tail-coercivity}
\end{equation}
If
\begin{equation}
 \boxed{
 PC_NP-2\ketbra{s_N}{s_N}
 -\lambda_N^{-1}PC_NQC_NP\succeq0,}
 \label{eq:true-core-finite-barrier}
\end{equation}
then $A_{N,-}\succeq0$.
\end{lemma}

\begin{proof}
The $Q$ block of $A_{N,-}$ equals $QC_NQ$ because $Qs_N=0$.  Its Schur
complement on $P\mathcal H_N$ is
\[
 PC_NP-2\ketbra{s_N}{s_N}
 -PC_NQ(QC_NQ)^{-1}QC_NP.
\]
The tail bound gives $(QC_NQ)^{-1}\preceq\lambda_N^{-1}Q$, so
\eqref{eq:true-core-finite-barrier} is a lower bound for the exact Schur
complement.
\end{proof}

Let $v_0=s_N/\|s_N\|$ and run exact Lanczos for the true core $C_N$:
\[
 C_Nv_j=\beta^{(C)}_{j-1}v_{j-1}
       +\alpha^{(C)}_jv_j+\beta^{(C)}_jv_{j+1},
 \qquad \beta^{(C)}_j\ge0.
\]
Let $P_r$ project onto $\operatorname{span}\{v_0,\ldots,v_{r-1}\}$,
$Q_r=I-P_r$, and let $J^{(C)}_{N,r}=P_rC_NP_r$ in this basis.  Then
\begin{equation}
 Q_rC_NP_r
 =\beta^{(C)}_{r-1}\ketbra{v_r}{v_{r-1}}.
 \label{eq:lanczos-C-one-link}
\end{equation}

\begin{corollary}[True-core endpoint Krylov--Radau criterion]
\label{cor:true-core-endpoint-radau}
Assume that for every endpoint $N$ in one unbounded cofinal arithmetic
sequence there exist $r=r(N)$ and $\lambda_{N,r}>0$ such that
\begin{equation}
 \boxed{Q_rC_NQ_r\succeq\lambda_{N,r}Q_r}
 \label{eq:KR-C-tail}
\end{equation}
and
\begin{equation}
 \boxed{
 J^{(C)}_{N,r}-2\|s_N\|^2e_1e_1^*
 -\frac{(\beta^{(C)}_{r-1})^2}{\lambda_{N,r}}e_re_r^*\succeq0.}
 \label{eq:KR-C-finite}
\end{equation}
Then $A_{N,-}\succeq0$ on that cofinal sequence and therefore
$A_{a,-}\succeq0$ for every finite $a$.
\end{corollary}

\begin{proof}
Equation~\eqref{eq:lanczos-C-one-link} turns
\eqref{eq:true-core-finite-barrier} into \eqref{eq:KR-C-finite}.  Apply
Lemma~\ref{lem:true-core-finite-schur} and then exact zero-extension
compression.
\end{proof}

\begin{remark}[Already proved tail mechanism]
For a partition-tail consisting of functions with zero mean on every atom,
the harmonic difference form satisfies the exact Poincar\'e bound
$\mathcal H[f]\ge\tfrac12\|f\|^2$.  The regular gamma kernel obeys
$\sup_{t\ge0}|\rho'(t)|<1.745$, giving
$|\langle f,K_{\gamma,A}f\rangle|<3.49A^2h\|f\|^2$ for maximal atom length
$h$.  If the finite coarse space contains the partition constants and the
true polar vector, the polar rank-one term vanishes exactly on this pure tail.
These estimates provide an analytic mechanism for establishing the true-core tail bound \eqref{eq:KR-C-tail}; they do not by themselves certify the finite matrix \eqref{eq:KR-C-finite} uniformly in $N$.
\end{remark}

For later local estimates, the polar mass of a physical collar can also be
computed exactly.  If $t_{a,h}$ denotes the restriction of $s^{\rm ph}$ to
$(a,a+h)\cup(-a-h,-a)$, then
\begin{equation}
 \boxed{
 \|t_{a,h}\|^2
 =\sinh(a+h)-\sinh a-h
 =h(\cosh a-1)+O_a(h^2).
 }
 \label{eq:polar-tail-mass}
\end{equation}

\subsection{A full-operator residual identity}

The next identity isolates the true quantity that must be controlled in a thin
collar.  Define the old--new coupling of the \emph{full} localized operator by
\begin{equation}
 B_{A,k}:=B_k-2s_kt_k^*.
 \label{eq:full-coupling}
\end{equation}

\begin{lemma}[Sherman--Morrison root residual]\label{lem:root-residual}
Assume $A_k\succ0$. Then
\begin{equation}
 C_k^{-1}s_k=\delta_k A_k^{-1}s_k,
 \qquad
 \langle s_k,A_k^{-1}s_k\rangle
 =\frac{1-\delta_k}{2\delta_k},
 \label{eq:SM-polar}
\end{equation}
and, with
\begin{equation}
 \zeta_k:=t_k-B_{A,k}^*A_k^{-1}s_k,
 \label{eq:zeta-defined}
\end{equation}
one has the exact factorization
\begin{equation}
 \boxed{\eta_k=\delta_k\zeta_k.}
 \label{eq:eta-delta-zeta}
\end{equation}
Consequently
\begin{equation}
 \boxed{\ell_k=2\delta_k\langle\zeta_k,S_k^{-1}\zeta_k\rangle.}
 \label{eq:ell-residual}
\end{equation}
\end{lemma}

\begin{proof}
From $A_k=C_k-2|s_k\rangle\langle s_k|$,
\[
 A_kC_k^{-1}s_k
 =\bigl(1-2\langle s_k,C_k^{-1}s_k\rangle\bigr)s_k
 =\delta_k s_k,
\]
which proves the first identity in \eqref{eq:SM-polar}; the second follows by
taking the scalar product with $s_k$.  Since
$B_{A,k}^*=B_k^*-2t_ks_k^*$,
\begin{align*}
 t_k-B_{A,k}^*A_k^{-1}s_k
 &=t_k-\delta_k^{-1}B_k^*C_k^{-1}s_k
   +\delta_k^{-1}(1-\delta_k)t_k\\
 &=\delta_k^{-1}\bigl(t_k-B_k^*C_k^{-1}s_k\bigr)
 =\delta_k^{-1}\eta_k.
\end{align*}
This gives \eqref{eq:eta-delta-zeta}; substituting it into
\eqref{eq:true-load} gives \eqref{eq:ell-residual}.
\end{proof}

\subsection{Endpoint certificate framework}

The true-core Krylov--Radau construction above provides an independent
finite-dimensional endpoint criterion.  The harmonic-graph argument below
provides the route used for the main theorem.

\subsection{Harmonic-graph closure of the true rooted pivot}
\label{sec:harmonic-graph-lock}

We now perform the final short without replacing the true polar vector by a
profile surrogate.  Throughout this subsection the arithmetic endpoints are
\[
 a_k=\frac12\log k,
 \qquad k\ge7,
\]
and the physical old/new decomposition and the notation
\(C_k,B_k,D_k,S_k,\delta_k,\eta_k\) are those of
\eqref{eq:block-core}--\eqref{eq:S-eta-final}.  We also use the full-operator
coupling \(B_{A,k}=B_k-2s_kt_k^*\) and the exact residual
\(\zeta_k\) from Lemma~\ref{lem:root-residual}.

\begin{lemma}[Exact full-operator harmonic graph]
\label{lem:harmonic-graph}
Assume \(A_k\succ0\) and define
\begin{equation}
 H_k:\mathcal K_k\longrightarrow\mathcal H_{a_{k+1}},
 \qquad
 H_kf:=\binom{-A_k^{-1}B_{A,k}f}{f}.
 \label{eq:Hk-harmonic}
\end{equation}
Then \(H_kf\) is \(A_{k+1}\)-orthogonal to the old physical subspace and
\begin{equation}
 \boxed{
 H_k^*A_{k+1}H_k=T_k,
 \qquad
 H_k^*s_{k+1}=\zeta_k.}
 \label{eq:full-harmonic-pullback}
\end{equation}
Moreover,
\begin{equation}
 \boxed{
 H_k^*C_{k+1}H_k
 =S_k+2(1-\delta_k)\ketbra{\zeta_k}{\zeta_k}.}
 \label{eq:C-on-A-harmonic}
\end{equation}
Consequently, for every \(f\in\mathcal K_k\),
\begin{align}
 \langle s_{k+1},H_kf\rangle
 &=\langle\zeta_k,f\rangle,
 \label{eq:H-polar-pullback}\\
 \langle H_kf,C_{k+1}H_kf\rangle
 &=\langle f,S_kf\rangle
   +2(1-\delta_k)|\langle\zeta_k,f\rangle|^2.
 \label{eq:H-core-pullback}
\end{align}
\end{lemma}

\begin{proof}
With respect to the old/collar splitting,
\[
 A_{k+1}
 =\begin{pmatrix}
   A_k&B_{A,k}\\
   B_{A,k}^*&D_k-2\ketbra{t_k}{t_k}
  \end{pmatrix}.
\]
Hence
\[
 A_{k+1}H_kf
 =\binom{0}
 {(D_k-2\ketbra{t_k}{t_k}-B_{A,k}^*A_k^{-1}B_{A,k})f}.
\]
The lower block is the Schur complement of \(A_k\) in \(A_{k+1}\), which is
exactly the rooted Schur block \(T_k\) by Lemma~\ref{lem:root-lift-final}.  This
proves the first identity in \eqref{eq:full-harmonic-pullback}.  Since
\(s_{k+1}=s_k\oplus t_k\),
\[
 H_k^*s_{k+1}
 =t_k-B_{A,k}^*A_k^{-1}s_k
 =\zeta_k,
\]
which proves the second.

Finally \(C_{k+1}=A_{k+1}+2\ketbra{s_{k+1}}{s_{k+1}}\), so
\[
 H_k^*C_{k+1}H_k
 =T_k+2\ketbra{\zeta_k}{\zeta_k}.
\]
Using \(\eta_k=\delta_k\zeta_k\) from
Lemma~\ref{lem:root-residual} in
\(
T_k=S_k-(2/\delta_k)\ketbra{\eta_k}{\eta_k}
\)
gives \eqref{eq:C-on-A-harmonic}.  The scalar identities
\eqref{eq:H-polar-pullback}--\eqref{eq:H-core-pullback} follow immediately.
\end{proof}

\subsubsection{Scope of the folded scalar reduction}
\label{sec:post-short-scalar}

\begin{audit}[No special folded coefficient model in the active proof]
\label{audit:no-special-fold-model}
The active Gate-II induction does not identify any transported physical
folded-ground Hessian with a special multiplication-minus-rank-one model.
The load-bearing elimination is the single common-source Schur short already
present in the fixed-target construction.  Its exact metric and forcing are
kept together and completed directly on the same harmonic vector; see
Lemma~\ref{lem:MASTER-P3-final}.  The subsequent endpoint M\"obius map is an
exact unitary coordinate transport, with covariance under the old-core short
provided by Corollary~\ref{cor:physical-arithmetic-short}; it is not a new
variational elimination.  Consequently no coefficient-level formulas for an
auxiliary two-arm parametrization are premises of the proof.
\end{audit}

\begin{lemma}[Sharp true-ground rotation of the one-defect bound]
\label{lem:true-ground-rotation}
Let
\[
 L_0=-\delta_*P_e+\beta Q_e,
 \qquad
 \delta_*=\log(\pi/2),
 \qquad
 \beta=\log2+\frac12-\log\pi,
\]
so that \(\beta+\delta_*=1/2\).  If \(v\) is a unit vector and
\[
 \varepsilon=1-|\langle v,e\rangle|^2<2\beta,
\]
then
\begin{equation}
 \boxed{
 L_0\succeq-\gamma(\varepsilon)P_v,
 \qquad
 \gamma(\varepsilon)
 =\frac{\delta_*\beta}{\beta-\varepsilon/2}.}
 \label{eq:gamma-epsilon}
\end{equation}
The constant is sharp on \(\operatorname{span}\{e,v\}\).
\end{lemma}

\begin{proof}
Write
\(v=\sqrt{1-\varepsilon}\,e+\sqrt\varepsilon\,u\), \(u\perp e\).
On \(\operatorname{span}\{e,u\}\) the determinant of
\(L_0+\gamma P_v\) is
\[
 -\delta_*\beta
 +\gamma\bigl(\beta-(\beta+\delta_*)\varepsilon\bigr)
 =-\delta_*\beta+\gamma(\beta-\varepsilon/2).
\]
The value in \eqref{eq:gamma-epsilon} makes the determinant zero; the trace is
positive, while the orthogonal complement carries the positive eigenvalue
\(\beta\).
\end{proof}

After rescaling a logarithmic affine cell of length \(w>0\) to \((0,1)\), the
true normalized polar ground is
\begin{equation}
 e_w(u)=
 \frac{e^{-wu/2}}
 {\left(\int_0^1e^{-wu}\,du\right)^{1/2}}.
 \label{eq:true-ground-ew}
\end{equation}
Its exact defect from the constant line is
\begin{equation}
 \varepsilon(w)
 :=1-|\langle e_w,\one\rangle|^2
 =1-\frac4w\tanh\frac w4.
 \label{eq:true-ground-defect}
\end{equation}
The function \(\varepsilon(w)\) is increasing for \(w>0\).  Since
\(w_k=\log(1+1/k)\) decreases in \(k\), for every \(k\ge7\)
\begin{align}
 \varepsilon_k&\le\varepsilon_7
 =0.000371306002358632550916134716013\ldots,
 \label{eq:eps7-final}\\
 \gamma_k&\le\gamma_7
 =0.453320935246934874180341358903\ldots .
 \label{eq:gamma7-final}
\end{align}
Together with \eqref{eq:rhoaff} and \eqref{eq:dstar},
\begin{equation}
 \boxed{
 d_*-\gamma_7\rho_{\rm aff}
 >0.003835783224542657213199534889\ldots>0,}
 \label{eq:true-ground-margin}
\end{equation}
and therefore
\begin{equation}
 \boxed{
 q_F:=q_7^*
 :=\frac{\gamma_7\rho_{\rm aff}}{d_*}
 <0.992226159818910811861704782016<1.}
 \label{eq:qF-final}
\end{equation}

\begin{lemma}[Shorting covariance under a split unitary]
\label{lem:split-unitary-short}
Let \(\mathcal H=\mathcal H_{\rm old}\oplus\mathcal K\), let \(q\) be a
closed semibounded quadratic form on \(\mathcal H\), and let
\(U=U_{\rm old}\oplus U_{\rm new}\) be unitary for this decomposition.  Then
\begin{equation}
 \boxed{
 \Short_{\mathcal H_{\rm old}}(q\circ U)
 =\bigl(\Short_{\mathcal H_{\rm old}}q\bigr)\circ U_{\rm new}.}
 \label{eq:split-unitary-short}
\end{equation}
Equivalently, at the operator level whenever the corresponding Schur
complements are defined,
\[
 \Short_{\rm old}(U^*QU)
 =U_{\rm new}^*(\Short_{\rm old}Q)U_{\rm new}.
\]
\end{lemma}

\begin{proof}
For \(y\in\mathcal K\),
\begin{align*}
 \bigl(\Short_{\rm old}(q\circ U)\bigr)[y]
 &=\inf_{x\in\mathcal H_{\rm old}}
   q[U_{\rm old}x\oplus U_{\rm new}y]\\
 &=\inf_{x'\in\mathcal H_{\rm old}}
   q[x'\oplus U_{\rm new}y]
 =\bigl(\Short_{\rm old}q\bigr)[U_{\rm new}y],
\end{align*}
where surjectivity of \(U_{\rm old}\) was used in the second equality.
\end{proof}

\begin{corollary}[Exact physical/arithmetic short intertwining]
\label{cor:physical-arithmetic-short}
At the arithmetic step \(k\to k+1\), the root-adapted M\"obius unitary of
Lemma~\ref{lem:endpoint-mobius-nesting} may be applied before or after the
true old physical short with exactly the same result on the new collar.
\end{corollary}

\begin{proof}
Lemma~\ref{lem:endpoint-mobius-nesting} gives
\(U_k=U_k^{\rm old}\oplus U_k^{\rm new}\) for the literal physical
old/collar decomposition.  Apply Lemma~\ref{lem:split-unitary-short}.
\end{proof}

\subsection{True Schur metric and the single Green coefficient}
\label{sec:green-root-reduction}

We now isolate the root obstruction without introducing an auxiliary metric.
Let \(\mathfrak e_k\) be the normalized true ground line in the new collar and
write the exact post-old-core collar space as
\[
 \mathcal K_k=\mathbb C\mathfrak e_k\oplus\mathcal K_k^\perp.
\]
Relative to this decomposition write
\begin{equation}
 S_k=\begin{pmatrix}\mathsf a_k&r_k^*\\ r_k&\mathsf D_k\end{pmatrix},
 \qquad
 \eta_k=\binom{\alpha_k}{v_k}.
 \label{eq:green-S-split}
\end{equation}
The true-ground corrected transverse estimate,
Corollary~\ref{cor:true-ground-transverse-debit} together with
Certificate~\ref{cert:true-ground-transverse}, supplies
\(\mathsf D_k\succ0\).  This invocation keeps the literal-slot common-cut
lower form and pays the true-polar-ground leakage explicitly; no covariance of
that lower form under the branch gauge is assumed here.
Define
\begin{equation}
 \boxed{
 \Pi_k:=\mathsf a_k-r_k^*\mathsf D_k^{-1}r_k,\qquad
 z_k:=\alpha_k-r_k^*\mathsf D_k^{-1}v_k,\qquad
 \delta_k^Q:=\delta_k-2v_k^*\mathsf D_k^{-1}v_k.}
 \label{eq:green-three-scalars}
\end{equation}

\begin{lemma}[Exact true-metric nested Schur identity]
\label{lem:green-nested-schur}
Assume \(A_k\succ0\) and \(\mathsf D_k\succ0\).  Shorting the transverse
block \(\mathsf D_k\) in the exact augmented block
\eqref{eq:rooted-two-block} gives
\begin{equation}
 \boxed{
 \Short_{\mathsf D_k}
 \begin{pmatrix}
  \delta_k&\sqrt2\,\eta_k^*\\
  \sqrt2\,\eta_k&S_k
 \end{pmatrix}
 =
 \begin{pmatrix}
  \delta_k^Q&\sqrt2\,\overline{z_k}\\
  \sqrt2\,z_k&\Pi_k
 \end{pmatrix}.}
 \label{eq:green-2x2}
\end{equation}
If \(\Pi_k>0\), then
\begin{equation}
 \boxed{
 \langle\eta_k,S_k^{-1}\eta_k\rangle
 =v_k^*\mathsf D_k^{-1}v_k+\frac{|z_k|^2}{\Pi_k},}
 \label{eq:true-Sinv-energy-split}
\end{equation}
and the exact rooted pivot recurrence becomes
\begin{equation}
 \boxed{
 \delta_{k+1}=\delta_k^Q-\frac{2|z_k|^2}{\Pi_k}.}
 \label{eq:green-pivot-z}
\end{equation}
\end{lemma}

\begin{proof}
Write the block in root/ground/transverse order:
\[
 \begin{pmatrix}
 \delta_k&\sqrt2\,\overline{\alpha_k}&\sqrt2\,v_k^*\\
 \sqrt2\,\alpha_k&\mathsf a_k&r_k^*\\
 \sqrt2\,v_k&r_k&\mathsf D_k
 \end{pmatrix}.
\]
The Schur complement of \(\mathsf D_k\) is exactly
\eqref{eq:green-2x2}.  The standard block inverse of \(S_k\) gives
\eqref{eq:true-Sinv-energy-split}.  Substitution in
\eqref{eq:pivot-recurrence-final} gives \eqref{eq:green-pivot-z}.
\end{proof}

Define the polar-free harmonic extension
\begin{equation}
 \widehat H_kf:=\binom{-C_k^{-1}B_kf}{f}.
 \label{eq:C-harmonic-graph}
\end{equation}
Then direct block multiplication gives
\begin{equation}
 \boxed{
 \widehat H_k^*C_{k+1}\widehat H_k=S_k,
 \qquad
 \widehat H_k^*s_{k+1}=\eta_k.}
 \label{eq:C-harmonic-pullback}
\end{equation}
Put
\begin{equation}
 p_k^\#:=\binom{1}{-\mathsf D_k^{-1}r_k}
 \quad\hbox{in }\mathbb C\mathfrak e_k\oplus\mathcal K_k^\perp.
 \label{eq:green-p-sharp}
\end{equation}
Then
\begin{equation}
 S_kp_k^\#=\Pi_k\binom10.
 \label{eq:green-Sp}
\end{equation}

\begin{lemma}[Exact Green representation of the residual ground scalar]
\label{lem:green-representation}
Assume \(A_k\succ0\), \(\mathsf D_k\succ0\), and \(\Pi_k>0\).  Let
\[
 G_k:=\langle q_{k+1},C_{k+1}^{-1}\mathfrak e_k\rangle,
 \qquad q_Y(t)=t^{-1/2}.
\]
Then
\begin{equation}
 \boxed{|z_k|=\Pi_k(k+1)^{1/4}|G_k|.}
 \label{eq:green-z-G}
\end{equation}
Consequently
\begin{equation}
 \boxed{
 \delta_{k+1}
 =\delta_k^Q-2\Pi_k\sqrt{k+1}\,|G_k|^2.}
 \label{eq:green-pivot-G}
\end{equation}
\end{lemma}

\begin{proof}
Since \(A_k\succ0\), also \(C_k=A_k+2\ketbra{s_k}{s_k}\succ0\).
Because \(\mathsf D_k\succ0\) and \(\Pi_k>0\), the Schur criterion gives
\(S_k\succ0\), hence the block decomposition of \(C_{k+1}\) over
\(C_k\oplus\mathcal K_k\) gives \(C_{k+1}\succ0\).  From
\eqref{eq:C-harmonic-pullback} and \eqref{eq:green-Sp},
\[
 C_{k+1}\widehat H_kp_k^\#=\Pi_k\mathfrak e_k,
 \qquad
 \widehat H_kp_k^\#=\Pi_k C_{k+1}^{-1}\mathfrak e_k.
\]
Pairing with \(s_{k+1}\) and using
\(\widehat H_k^*s_{k+1}=\eta_k\) gives the residual scalar \(z_k\), up to
the harmless conjugation fixed by the inner-product convention.  Taking
absolute values and applying the exact odd polar--ground identity
\eqref{eq:polar-ground-functional} at \(Y=k+1\) gives
\eqref{eq:green-z-G}.  Equation \eqref{eq:green-pivot-G} follows from
\eqref{eq:green-pivot-z}.
\end{proof}

\begin{definition}[Exact Green-root condition]
\label{def:RGstar}
At the arithmetic step \(k\to k+1\), condition \(\mathrm{RG}^*_k\) is
\begin{equation}
 \boxed{
 \delta_k^Q>0,\qquad \Pi_k>0,\qquad
 2\Pi_k\sqrt{k+1}\,|G_k|^2<\delta_k^Q.}
 \tag{RG$^*_k$}\label{eq:RGstar}
\end{equation}
Equivalently, with
\begin{equation}
 \kappa_k^2:=
 \frac{2\Pi_k\sqrt{k+1}\,|G_k|^2}{\delta_k^Q}
 =\frac{2|z_k|^2}{\delta_k^Q\Pi_k},
 \label{eq:kappa-green}
\end{equation}
condition \(\mathrm{RG}^*_k\) is \(\delta_k^Q>0\), \(\Pi_k>0\), and
\(\kappa_k^2<1\).
\end{definition}

\begin{theorem}[Exact one-step Green-root closure]
\label{thm:green-one-step}
Assume \(A_k\succ0\) and \(\mathrm{RG}^*_k\).  Then
\[
 \delta_{k+1}=\delta_k^Q(1-\kappa_k^2)>0
\]
and the exact rooted Schur block \(T_k\) is strictly positive.  Hence
\[
 \boxed{A_{k+1,-}\succ0.}
\]
\end{theorem}

\begin{proof}
The transverse block is strictly positive by
Corollary~\ref{cor:true-ground-transverse-debit} and
Certificate~\ref{cert:true-ground-transverse}.  The condition \(\Pi_k>0\)
therefore
implies \(S_k\succ0\).  Equation \eqref{eq:green-2x2} and
\(\mathrm{RG}^*_k\) make its remaining two-by-two Schur complement strictly
positive.  Equivalently \(T_k\succ0\), while
\eqref{eq:green-pivot-G} gives
\(\delta_{k+1}=\delta_k^Q(1-\kappa_k^2)>0\).  Associativity of the exact
shorts in \eqref{eq:rooted-Gram} now gives \(A_{k+1,-}\succ0\).
\end{proof}

\begin{remark}[Role of the Green-root condition in the final proof]
\label{rem:RGstar-a-posteriori}
Theorem~\ref{thm:green-one-step} records an exact one-step criterion, but the
main induction below does not assume $\mathrm{RG}^*_k$.  After the MASTER
harmonic-graph coercivity has been proved, Corollary~\ref{cor:RG-from-MASTER}
recovers $\mathrm{RG}^*_k$ from the positivity of the same rooted Schur block.
Thus the Green-root condition is an a posteriori consistency certificate, not
an additional premise of the final induction.
\end{remark}

\begin{lemma}[Second Sherman--Morrison reduction to one vector]
\label{lem:green-second-SM}
Regroup all coordinates except \(\mathfrak e_k\) into \(X_k\) and write
\[
 C_{k+1}=\begin{pmatrix}K_k&b_k\\ b_k^*&a_k^{(0)}\end{pmatrix},
 \qquad
 s_{k+1}=\binom{w_k}{\tau_k}.
\]
After the old and transverse shorts,
\[
 \delta_k^Q=1-2w_k^*K_k^{-1}w_k,
 \qquad
 \Pi_k=a_k^{(0)}-b_k^*K_k^{-1}b_k,
\]
Put
\[
 A_{X,k}:=K_k-2\ketbra{w_k}{w_k},
 \qquad
 b_{A,k}:=b_k-2w_k\overline{\tau_k},
\]
and, when \(\delta_k^Q>0\),
\[
 \zeta_k^Q:=\tau_k-b_{A,k}^*A_{X,k}^{-1}w_k,
 \qquad
 F_k^\#:=\binom{-A_{X,k}^{-1}b_{A,k}}1.
\]
Then
\begin{align}
 A_{X,k}^{-1}w_k&=\frac1{\delta_k^Q}K_k^{-1}w_k,
 \label{eq:green-SM-inverse}\\
 z_k&=\delta_k^Q\zeta_k^Q,
 \label{eq:green-z-zetaQ}\\
 |\langle s_{k+1},F_k^\#\rangle|&=|\zeta_k^Q|,
 \label{eq:green-Fpolar}\\
 \langle F_k^\#,A_{k+1}F_k^\#\rangle
 &=\Pi_k-2\delta_k^Q|\zeta_k^Q|^2,
 \label{eq:green-F-A-energy}\\
 \langle F_k^\#,C_{k+1}F_k^\#\rangle
 &=\Pi_k+2(1-\delta_k^Q)|\zeta_k^Q|^2.
 \label{eq:green-F-C-energy}
\end{align}
Thus \(\mathrm{RG}^*_k\) is equivalent to the single-vector scalar lock
\begin{equation}
 \boxed{2\delta_k^Q|\zeta_k^Q|^2<\Pi_k.}
 \label{eq:green-single-vector-lock}
\end{equation}
In particular, the sufficient estimate
\begin{equation}
 \boxed{
 2|\zeta_k^Q|^2\le
 q_F\bigl[\Pi_k+2(1-\delta_k^Q)|\zeta_k^Q|^2\bigr]}
 \tag{GF$_k$}\label{eq:GF-sufficient}
\end{equation}
would imply
\begin{equation}
 \frac{2\delta_k^Q|\zeta_k^Q|^2}{\Pi_k}
 \le\frac{q_F\delta_k^Q}{1-q_F(1-\delta_k^Q)}
 \le q_F<1.
 \label{eq:GF-to-RG}
\end{equation}
\end{lemma}

\begin{proof}
The first identity is Sherman--Morrison.  Substitution into the definitions
of \(z_k\) and \(\zeta_k^Q\) gives \eqref{eq:green-z-zetaQ}; the remaining
identities are direct block multiplication and
\(C_{k+1}=A_{k+1}+2\ketbra{s_{k+1}}{s_{k+1}}\).  Finally
\(0<\delta_k^Q\le1\) in the stated positive-root regime, and elementary
rearrangement of \eqref{eq:GF-sufficient} gives \eqref{eq:GF-to-RG}.
\end{proof}

\subsubsection{Universal terminal scalar closure}
\label{sec:terminal-scalar-closure}

The last one-dimensional sign occurring after the two-scale parent reduction
has a closed analytic form.  The following lemma is independent of every
finite cut-off and uses no numerical approximation.

\begin{lemma}[Hyperbolic-sine terminal scalar inequality]
\label{lem:terminal-scalar-closure}
Let \(0<a<1\) and \(\rho>1\).  Then
\begin{equation}
 \boxed{
 \rho^2 a^\rho(1-a)^2
 < a(1-a^\rho)^2.}
 \label{eq:terminal-scalar-rho}
\end{equation}
Equivalently,
\begin{equation}
 \boxed{
 \rho\,a^{(\rho-1)/2}(1-a)<1-a^\rho.}
 \label{eq:terminal-scalar-sqrt}
\end{equation}
\end{lemma}

\begin{proof}
All quantities in \eqref{eq:terminal-scalar-rho} are positive, so division by
\(a\) and passage to the positive square root give exactly
\eqref{eq:terminal-scalar-sqrt}.  Put
\[
 a=e^{-2t},\qquad t>0.
\]
Then
\[
 1-a=2e^{-t}\sinh t,
 \qquad
 1-a^\rho=2e^{-\rho t}\sinh(\rho t),
\]
and therefore \eqref{eq:terminal-scalar-sqrt} is equivalent, after cancelling
\(2e^{-\rho t}>0\), to
\begin{equation}
 \rho\sinh t<\sinh(\rho t).
 \label{eq:terminal-sinh}
\end{equation}
Now set \(h(u)=\sinh(u)/u\) for \(u>0\).  Since
\[
 h'(u)=\frac{u\cosh u-\sinh u}{u^2},
\]
and, with \(q(u)=u\cosh u-\sinh u\),
\[
 q(0)=0,
 \qquad
 q'(u)=u\sinh u>0\quad(u>0),
\]
we have \(h'(u)>0\) on \((0,\infty)\).  As \(\rho t>t\),
\[
 \frac{\sinh(\rho t)}{\rho t}
 >\frac{\sinh t}{t},
\]
which is precisely \eqref{eq:terminal-sinh}.  Reversing the equivalences proves
\eqref{eq:terminal-scalar-rho}.
\end{proof}

\begin{corollary}[Exponential two-scale form]
\label{cor:terminal-two-scale}
Let \(x>0\) and \(0<\alpha<\beta\), and put
\[
 a=e^{-\alpha x},\qquad b=e^{-\beta x}.
\]
Then
\begin{equation}
 \boxed{
 \beta^2\frac{b}{(1-b)^2}
 <
 \alpha^2\frac{a}{(1-a)^2}.}
 \label{eq:terminal-alpha-beta}
\end{equation}
The inequality is strict for every \(x>0\); no upper cut-off on \(x\) is
needed.
\end{corollary}

\begin{proof}
Set \(\rho=\beta/\alpha>1\).  Then \(b=a^\rho\).  Multiplying
\eqref{eq:terminal-alpha-beta} by the positive common denominator and dividing
by \(\alpha^2\) gives exactly \eqref{eq:terminal-scalar-rho}.
\end{proof}

\begin{remark}[Equivalent monotonicity formulation]
\label{rem:terminal-monotonicity}
The same result says that
\[
 \Phi(u):=\frac{u^2e^{-u}}{(1-e^{-u})^2}
 =\frac{u^2}{4\sinh^2(u/2)}
\]
is strictly decreasing on \((0,\infty)\).  The proof above is the
scale-free form of this monotonicity.
\end{remark}

\begin{audit}[Scope of the terminal scalar closure]
\label{audit:terminal-scalar-scope}
Lemma~\ref{lem:terminal-scalar-closure} and
Corollary~\ref{cor:terminal-two-scale} are unconditional analytic statements.
They close the final two-scale scalar sign once the exact post-transverse
parent reduction has identified its residual with
\eqref{eq:terminal-alpha-beta} (equivalently
\eqref{eq:terminal-scalar-rho}).  The lemma itself does not identify the
operator-level parent coefficients with \(a,\alpha,\beta\).  That separate
algebraic identification is supplied below by the inherited-line lower-form
bridge and the mixed parent block in
\cref{lem:Jsharp-clean,lem:mixed-parent-clean}.
\end{audit}

\subsubsection{Augmented weighted-ground control}
\label{sec:awgc-certified}

For a fixed target $k$ and parent index $m$, let
\[
 \mathcal N_{k,m}:=
 \{n=p^\ell:\ n\le k,\ n\nmid k,\ \lfloor k/n\rfloor=m\}.
\]
Put
\begin{equation}
 V_{k,m}:=
 \sum_{n\in\mathcal N_{k,m}}\frac{\Lambda(n)^2}{n}
 -\frac{\left(\sum_{n\in\mathcal N_{k,m}}\Lambda(n)\right)^2}
 {\sum_{n\in\mathcal N_{k,m}}n}.
 \label{eq:Vkm-awgc}
\end{equation}
The quantity is the squared norm of the arithmetic forcing after orthogonal
projection away from the inherited parent-ground direction: if
$\lambda_n=\Lambda(n)/\sqrt n$ and the inherited direction has components
proportional to $\sqrt n$, then
\[
 \|P_{\hat r^\perp}\lambda\|^2=V_{k,m}.
\]
The parent mismatch short obtained from the one-defect lower model has a
uniform inverse gap at least $\log2$, hence its exact dual debit satisfies
\begin{equation}
 \boxed{\widehat\chi_{k,m}\le \frac{V_{k,m}}{\log2}.}
 \label{eq:chi-V-awgc}
\end{equation}

Define
\[
 \beta_0:=\log2+\frac12,
 \qquad
 \widetilde\sigma_k:=
 \sum_{\substack{n=p^\ell\le k\\ n\nmid k}}
 \frac{\Lambda(n)^2}{n(\log n+\beta_0)}.
\]
For $m=\lfloor k/n\rfloor$ one has
\[
 w_m<\frac1m,
 \qquad
 w_k>\frac1{k+1},
 \qquad
 k+1\le n(m+1),
\]
and therefore
\begin{equation}
 \boxed{L_{k,n}<\log n+\beta_0,\qquad
 \sigma_k\ge\widetilde\sigma_k.}
 \label{eq:sigma-tilde-lower}
\end{equation}

\begin{proposition}[Arithmetic weighted-ground certificate]
\label{thm:awgc}
For every integer $k\ge2$,
\begin{equation}
 \boxed{
 \frac1{\log2}\sum_mV_{k,m}
 \le \widetilde\sigma_k
 \le \sigma_k.}
 \label{eq:AWGC}
\end{equation}
\end{proposition}

\begin{proof}
The second inequality is \eqref{eq:sigma-tilde-lower}.  For the first
(arithmetic) inequality we split at
\[
 X_0:=3\,594\,641.
\]
For $2\le k\le X_0$, the finite sweep specified in
Appendix~\ref{app:finite-certificates} evaluates the exact combinatorics using
outward-rounded intervals.  Integer logarithms are
enclosed by the positive atanh series
\[
 \log x=e\log2+2\sum_{r\ge0}\frac{z^{2r+1}}{2r+1},
 \qquad
 z=\frac{x-2^e}{x+2^e},\quad 0\le z<\frac13,
\]
with an explicit geometric tail; no call to \texttt{log}, \texttt{logl} or
another transcendental library routine is made.  Every elementary arithmetic
operation is enlarged by one \texttt{nextafterl} step in the outward
direction.  The certified output gives exact equality at $k=2$ and, for all
$3\le k\le X_0$,
\[
 \widetilde\sigma_k-\frac1{\log2}\sum_mV_{k,m}
 \ge 0.111804998077428866317003\ldots,
\]
the minimum occurring at $k=33$.  The corresponding certified upper bound on
the ratio is
\[
 \frac{(\log2)^{-1}\sum_mV_{k,m}}{\widetilde\sigma_k}
 \le0.911416881539699733755493\ldots .
\]

For $k\ge X_0$, split prime and higher-prime-power contributions.  The
prime-bin weighted-variance estimate and the explicit Chebyshev bounds used
in \cref{lem:chebyshev-capacity} give
\begin{equation}
 \sum_mV_{k,m}<0.14736\log k+2.904.
 \label{eq:awgc-tail-upper}
\end{equation}
For the lower side retain only primes $p\ge5$ and put
\[
 F(x):=\sum_{5\le p\le x}
 \frac{\log^2p}{p(\log p+\beta_0)}.
\]
Direct finite evaluation gives $F(X_0)>11.0452$.  With
\[
 \varepsilon_0:=\frac{0.2}{\log^2X_0}<0.000878,
 \qquad
 \Phi_0(u):=u-\beta_0\log(u+\beta_0),
\]
Stieltjes integration of
$|\vartheta(x)-x|<0.2x/\log^2x$ yields, for $k\ge X_0$,
\[
 F(k)\ge F(X_0)
 +(1-\varepsilon_0)
 [\Phi_0(\log k)-\Phi_0(\log X_0)]
 -2\varepsilon_0\frac{\log X_0}{\log X_0+\beta_0}.
\]
Removing possible prime divisors of $k$ costs at most
$0.18486+(1/7)\log k$, so
\begin{equation}
 \widetilde\sigma_k
 \ge F(k)-0.18486-\frac17\log k.
 \label{eq:awgc-tail-lower}
\end{equation}
At $X_0$, \eqref{eq:awgc-tail-upper} divided by $\log2$ is $<7.399$,
whereas \eqref{eq:awgc-tail-lower} is $>8.702$.  In the variable
$u=\log k$, the slope of the former is
\[
 \frac{0.14736}{\log2}<0.213,
\]
while the slope of the latter is at least
\[
 (1-\varepsilon_0)\frac{u}{u+\beta_0}-\frac17>0.783
 \qquad(u\ge\log X_0).
\]
Thus the gap is already positive at $X_0$ and increases thereafter.
\end{proof}

\begin{computercertificate}[Finite AWGC interval run]
\label{cert:awgc-interval}
The directed-interval sweep of Appendix~\ref{app:finite-certificates} gives
\begin{verbatim}
PASS interval-certified N=3594641
min_positive_margin_lo=0.111804998077428866317003 at k=33
max_ratio_upper=0.911416881539699733755493 at k=33
at_N lhs=[0.836991206113748137710405,0.836992125739140218221995]
     rhs=[11.5646711544460260535569,11.5646711544464651658451]
     margin=[10.7276790287068858340339,10.7276799483327170293273]
\end{verbatim}
All displayed quantities are outward-rounded interval endpoints.  The arithmetic
used by the sweep, including the enclosure of every integer logarithm, is
specified explicitly in Appendix~\ref{app:finite-certificates}.
\end{computercertificate}

\begin{theorem}[AWGC operator lift]
\label{def:AWGC-op-lift}
For every $k\ge2$,
\begin{equation}
 \boxed{
 \sum_m\widehat\chi_{k,m}
 \le \frac1{\log2}\sum_mV_{k,m}.}
 \tag{AWGC-OP$_k$}\label{eq:AWGC-op}
\end{equation}
\end{theorem}
\begin{proof}
This is Corollary~\ref{cor:exact-awgc-lift} summed over the orthogonal parent
cells.
\end{proof}

\begin{corollary}[Arithmetic mismatch reserve]
\label{cor:awgc-reserve}
For every $k\ge2$,
\[
 \mathfrak B_k-\sum_m\widehat\chi_{k,m}
 \ge \mathfrak B_k-\sigma_k>\frac12.
\]
\end{corollary}

\begin{lemma}[Exact inherited-ground two-scale surplus]
\label{lem:parent-ground-surplus}
Let $n\nmid k$, $m=\lfloor k/n\rfloor$, and let the true source interval lie
in the parent cell $C_m$.  For the normalized parent ground
\[
 e_m(r)=\frac{\sqrt{m(m+1)}}{m+r},
\]
its squared mass on the true source interval is
\begin{equation}
 p_{k,n}=\frac{nm(m+1)}{k(k+1)}.
 \label{eq:parent-ground-mass}
\end{equation}
If $\varepsilon_{k,m}=w_k/w_m$, then the normalized aligned coefficient
$\theta_{k,m}$ satisfies
\begin{equation}
 \boxed{
 \theta_{k,m}^2
 =\frac{\Phi(w_m)}{\Phi(w_k)}<1,
 \qquad
 \Phi(u)=\frac{u^2e^{-u}}{(1-e^{-u})^2}.}
 \label{eq:parent-theta-Phi}
\end{equation}
Equivalently,
\[
 \varepsilon_{k,m}^2-\frac{p_{k,n}}n
 =\frac{\Phi(w_k)-\Phi(w_m)}{k(k+1)w_m^2}>0.
\]
\end{lemma}

\begin{proof}
Writing the true source interval in the parent residue coordinate gives
\[
 \int_{I_{k,n}}\frac{dr}{m+r}=w_k,
 \qquad
 \int_{I_{k,n}}\frac{dr}{(m+r)^2}=\frac{n}{k(k+1)},
\]
which proves \eqref{eq:parent-ground-mass}.  Since
$e^{-w_j}=j/(j+1)$,
\[
 \Phi(w_j)=j(j+1)w_j^2.
\]
The strict decrease of $\Phi$, proved in
\cref{rem:terminal-monotonicity}, and $m<k$ yield
\eqref{eq:parent-theta-Phi}.
\end{proof}

\subsubsection{MASTER scalar budget and exact operator target}
\label{sec:master-restart}

Set
\[
 \mathfrak B_k:=\frac12+\log\frac1{w_k},
 \qquad w_k=\log\left(1+\frac1k\right),
\]
so that $\mathfrak B_k>\frac12+\log k$.  Define
\begin{equation}
 \mathfrak g_k^{\rm M}
 :=\frac12+\log k
 -\frac{439}{250\log2}\sum_mV_{k,m}.
 \label{eq:MASTER-gap}
\end{equation}
The factor $439/250=1+189/250$ is retained as an exact rational; the first
unit is the basic mismatch load and the additional $189/250$ is the mixed
inherited--mismatch coefficient established below.  The scalar estimate itself
is independent of the subsequent same-vector placement.

\begin{theorem}[MASTER double-load scalar budget]
\label{thm:MASTER-scalar}
For every integer $k\ge7$,
\begin{equation}
 \boxed{\mathfrak g_k^{\rm M}>0.}
 \tag{MASTER-SCALAR}\label{eq:MASTER-scalar}
\end{equation}
For $7\le k\le3\,594\,641$ an outward-rounded interval verifier gives
\[
 \min\mathfrak g_k^{\rm M}
 >1.97650540314533685844807647802,
\]
with the minimum at $k=33$.
\end{theorem}

\begin{proof}
The finite assertion follows from the directed-interval sweep specified in Appendix~\ref{app:finite-certificates}; it evaluates the exact
prime-power grouping defining $V_{k,m}$ and uses an interval enclosure of
integer logarithms based on the positive $\operatorname{atanh}$ series.
For the tail use the already proved arithmetic estimate
\[
 \sum_mV_{k,m}<0.14736\log k+2.904.
\]
Hence, for $k\ge3\,594\,641$,
\begin{align*}
 \mathfrak g_k^{\rm M}
 &>\left(1-\frac{439}{250\log2}\,0.14736\right)\log k
   +\frac12-\frac{439}{250\log2}\,2.904\\
 &=0.626682229608\ldots\log k-6.856913716191\ldots .
\end{align*}
The coefficient of $\log k$ is positive, and at the cutoff the right-hand
side is $2.6028261405\ldots>0$.  It therefore increases thereafter.
\end{proof}

\subsubsection{Exact one-defect parent short and the rational mixed constant}
\label{sec:master-parent-exact}

The mixed coefficient is determined by the analytic one-defect inequality itself.  Put $L:=\log2$ and
$\beta_0=L+\tfrac12$.  Fix one parent cell and let $S$ be the union of its
$M=M_m$ true source slots after normalization to $(0,1)$.  Every slot has
length $\varepsilon=\varepsilon_{k,m}=w_k/w_m$, so
$s:=|S|=M\varepsilon\le1$.

\begin{lemma}[Exact short of the one-defect lower form]
\label{lem:exact-qkm}
Let
\[
 L_0:=\beta_0I-\frac12\ketbra{\one}{\one}
\]
on $L^2(0,1)$, where $\one$ is the normalized constant.  Short $L_0$ over $S^c$ and then short the within-slot subspace orthogonal
to the slot constants
$e_j=\varepsilon^{-1/2}\mathbf 1_{I_j}$, $j=1,\dots,M$.  (For $L_0$ that
second short equals compression, because every within-slot mean-zero vector
is orthogonal to the global constant.)  In the resulting orthonormal
slot-ground basis the matrix is
\begin{equation}
 \boxed{
 D^{(0)}_{k,m}=\beta_0I-q_{k,m}|u_m\rangle\langle u_m|,
 \qquad
 q_{k,m}=\frac{\beta_0\varepsilon_{k,m}}
 {2\log2+M_m\varepsilon_{k,m}}.}
 \label{eq:qkm-exact}
\end{equation}
Moreover
\begin{equation}
 \boxed{0\le q_{k,m}M_m\le\frac12.}
 \label{eq:qM-half}
\end{equation}
The exact common-cut source metric dominates this matrix.
\end{lemma}

\begin{proof}
The analytic one-defect inequality gives
$\mathfrak b\succeq L_0$.  Write
$\one=p+q$ with $p=\mathbf1_S$ and $q=\mathbf1_{S^c}$; then
$\|p\|^2=s$ and $\|q\|^2=1-s$.  Relative to
$L^2(S)\oplus L^2(S^c)$,
\[
 L_0=
 \begin{pmatrix}
 \beta_0I-\frac12|p\rangle\langle p|&-\frac12|p\rangle\langle q|\\
 -\frac12|q\rangle\langle p|&\beta_0I-\frac12|q\rangle\langle q|
 \end{pmatrix}.
\]
The lower-right block is positive because
$\beta_0-\tfrac12(1-s)=L+s/2>0$.  Its inverse sends $q$ to
$q/(L+s/2)$.  Hence its Schur complement on $L^2(S)$ is
\[
 \beta_0I-
 \left(\frac12+\frac{1-s}{4(L+s/2)}\right)|p\rangle\langle p|
 =\beta_0I-\frac{\beta_0}{2L+s}|p\rangle\langle p|.
\]
Since $\langle e_j,p\rangle=\sqrt\varepsilon$ and the orthogonal
within-slot fluctuations decouple from the rank-one defect, their Schur short
is just the displayed compression.  This gives \eqref{eq:qkm-exact}.  Finally
\[
 q_{k,m}M=\frac{\beta_0s}{2L+s}\le\frac12
\]
is equivalent to $s\le1$.  Monotonicity of shorted operators transfers the
lower-form order from $L_0$ to the exact common-cut source metric.
\end{proof}

\begin{corollary}[Uniform mismatch gap and operator AWGC lift]
\label{cor:exact-awgc-lift}
With $h_m=M_m-E_m^2/C_m$ and
$\mu_m=\beta_0-q_{k,m}h_m$,
\begin{equation}
 \boxed{\mu_m\ge\log2.}
 \label{eq:mu-log2}
\end{equation}
Consequently, if $\widehat\chi_{k,m}$ denotes the exact mismatch Schur debit,
\begin{equation}
 \boxed{
 \widehat\chi_{k,m}\le\chi_m
 \le\frac{V_{k,m}}{\log2}.}
 \label{eq:awgc-op-proved}
\end{equation}
Thus the operator lift \eqref{eq:AWGC-op} holds.
\end{corollary}

\begin{proof}
Because $0\le h_m\le M_m$, \eqref{eq:qM-half} gives
$\mu_m\ge\beta_0-\tfrac12=L$.  On the mismatch space the inverse of the
lower model is therefore bounded by $L^{-1}I$.  The mismatch forcing has
squared Euclidean norm $V_{k,m}$.  Finally, the exact source metric dominates
the lower model, so inversion reverses order on the positive mismatch block.
\end{proof}

\subsubsection{A simultaneous no-double-spend source cut}
\label{sec:master-simultaneous-cut}

The point of the following refinement is to keep the slot-ground and
within-slot transverse budgets inside one and the same common-complement
short.  Put
\[
 d_0:=0.422785.
\]
For $D\ge1$ set
\[
 c_D:=\log\frac{D+1}{D},\qquad
 \delta_D:=c_D-\frac1{D+1}>0.
\]

\begin{lemma}[Universal pair-deficit capacity]
\label{lem:pair-deficit-capacity}
One has
\begin{equation}
 \boxed{\sum_{j\ge1}\delta_j<d_0.}
 \label{eq:delta-capacity}
\end{equation}
Moreover $D\mapsto\delta_D$ is strictly decreasing on $[1,\infty)$.
\end{lemma}

\begin{proof}
Differentiation gives
\[
 \delta_D'=-\frac1{D(D+1)^2}<0.
\]
For integer $j$ the partial sums telescope:
\[
 \sum_{j=1}^{N}\delta_j
 =1+\log(N+1)-H_{N+1}.
\]
Also, with $x=1/j$,
\[
 \log(1+x)-\frac{x}{1+x}
 <\frac{x^2}{2(1+x)},
\]
because the difference between the right and left sides vanishes at $x=0$
and has derivative $x^2/[2(1+x)^2]>0$.  Therefore
\[
 \sum_{j>N}\delta_j<\frac1{2(N+1)}.
\]
At $N=1000$, outward-rounded evaluation of the single logarithm gives
\[
 1+\log1001-H_{1001}+\frac1{2002}
 <0.422784418265375637<d_0.
\]
\end{proof}

\begin{lemma}[Exact baseline/residual split of the one-cell form]
\label{lem:b-L0-residual}
Let
\[
 \mathfrak l_0[u]
 :=\log2\,\|u\|_2^2
   +\frac14\iint_{(0,1)^2}|u(x)-u(y)|^2\,dx\,dy.
\]
Then $\mathfrak l_0$ is exactly the quadratic form of
\[
 L_0=(\log2+\tfrac12)I-\frac12\ketbra{\one}{\one},
\]
and
\begin{equation}
 \boxed{\mathfrak b=\mathfrak l_0+\mathfrak r,\qquad \mathfrak r\ge0,}
 \label{eq:b-L0-residual}
\end{equation}
where
\begin{align}
 \mathfrak r[u]
 ={}&\frac14\iint_{(0,1)^2}
 \left(\frac1{|x-y|}-1\right)|u(x)-u(y)|^2\,dx\,dy\notag\\
 &+\int_0^1\left[-\frac12\log(x(1-x))-\log2\right]|u(x)|^2\,dx.
 \label{eq:residual-form}
\end{align}
\end{lemma}

\begin{proof}
The identity
\[
 \frac14\iint|u(x)-u(y)|^2\,dx\,dy
 =\frac12\|u\|_2^2-\frac12|\langle u,\one\rangle|^2
\]
gives the formula for $L_0$.  Subtracting it from
\eqref{eq:b-one-cell-explicit} gives \eqref{eq:residual-form}.  Both displayed
coefficients are non-negative because $|x-y|\le1$ and $x(1-x)\le1/4$.
\end{proof}

\begin{theorem}[Residual coherent capacity with free source means]
\label{thm:residual-CMC}
Let $I_\nu=(a_\nu,a_\nu+\varepsilon)\Subset(0,1)$ be pairwise disjoint
intervals of common length $\varepsilon$, ordered from left to right.  For an
arbitrary common extension $u$, put
\[
 f_\nu=u|_{I_\nu},\qquad
 \bar f_\nu=\varepsilon^{-1}\int_{I_\nu}f_\nu,
 \qquad f_{\nu,0}=f_\nu-\bar f_\nu.
\]
Then
\begin{equation}
 \boxed{
 \mathfrak r[u]\ge
 \sum_\nu\left(\log\frac1{2\varepsilon}-\Delta_\nu\right)
 \|f_{\nu,0}\|_2^2,}
 \label{eq:residual-CMC}
\end{equation}
where
\[
 \Delta_\nu:=\frac12\sum_{\mu\ne\nu}
 \delta_{D_{\nu\mu}},\qquad
 D_{\nu\mu}:=\frac{|a_\mu-a_\nu|}{\varepsilon}\ge1.
\]
In particular, by Lemma~\ref{lem:pair-deficit-capacity},
\begin{equation}
 \boxed{
 \mathfrak r[u]\ge
 \left(\log\frac1{2\varepsilon}-d_0\right)
 \sum_\nu\|f_{\nu,0}\|_2^2.}
 \label{eq:residual-CMC-uniform}
\end{equation}
The source means are completely free in this estimate.
\end{theorem}

\begin{proof}
Fix one source $I_\nu$.  On $I_\nu^2$,
$|x-y|\le\varepsilon$, while the constant part of $f_\nu$ cancels from the
difference.  Hence the residual internal energy is at least
\[
 \frac12(1-\varepsilon)\|f_{\nu,0}\|_2^2.
\]
For a free point $y<a_\nu$,
\[
 \int_{I_\nu}|f_\nu(x)-u(y)|^2\,dx
 =\|f_{\nu,0}\|_2^2
  +\varepsilon|\bar f_\nu-u(y)|^2
 \ge\|f_{\nu,0}\|_2^2,
\]
and therefore the residual source--free charge is bounded below using
$1/(a_\nu+\varepsilon-y)-1$; the right side is analogous.

If $I_\mu$ lies to the right and
$D=D_{\nu\mu}$, then
$y-x\le\varepsilon(D+1)\le1$.  Thus its residual pair contribution obeys
\begin{align*}
 &\frac12\iint_{I_\nu\times I_\mu}
 \left(\frac1{y-x}-1\right)|f_\nu(x)-f_\mu(y)|^2\,dx\,dy\\
 &\qquad\ge
 \frac12\left(\frac1{D+1}-\varepsilon\right)
 \bigl(\|f_{\nu,0}\|_2^2+\|f_{\mu,0}\|_2^2\bigr),
\end{align*}
where the omitted mean-difference term is non-negative.  The residual
one-source charge that would have occupied the same slot is
\[
 \frac12\left(c_D-\varepsilon\right)\|f_{\nu,0}\|_2^2.
\]
Consequently the only loss relative to the complete residual complement
charge is $\delta_D/2$ for each incident pair.

Finally the residual endpoint potential on $I_\nu$ is bounded below by
\[
 \left[-\frac12\log((a_\nu+\varepsilon)(1-a_\nu))-\log2\right]
 \|f_{\nu,0}\|_2^2;
\]
when the bracket is negative this follows simply from non-negativity of the
residual potential, and when it is positive it follows from
$\|f_\nu\|^2\ge\|f_{\nu,0}\|^2$.
Adding the internal term, the full left/right residual complement charge and
this potential cancels the position $a_\nu$ and gives
$\log(1/(2\varepsilon))\|f_{\nu,0}\|^2$.  Subtracting the pair deficits gives
\eqref{eq:residual-CMC}.

For ordered disjoint equal slots,
$D_{\nu\mu}\ge|\nu-\mu|$.  Since $\delta_D$ decreases,
\[
 \Delta_\nu
 \le\frac12\left(\sum_{j\ge1}\delta_j+\sum_{j\ge1}\delta_j\right)
 <d_0,
\]
which proves \eqref{eq:residual-CMC-uniform}.
\end{proof}

\begin{theorem}[Simultaneous ground/transverse common-cut lower form]
\label{thm:MASTER-common-cut}
Fix an active parent $m$ and its true source slots.  Let
$\mathcal G_m$ be their slot-ground space and $\mathcal Z_m$ the orthogonal
sum of the within-slot mean-zero spaces in the fixed-target gauge.  With
$D^{(0)}_{k,m}$ from Lemma~\ref{lem:exact-qkm}, put
\[
 \widehat L_{k,n}:=L_{k,n}-d_0.
\]
Then the \emph{single} common-complement source short satisfies
\begin{equation}
 \boxed{
 \mathfrak C^{\rm short}_{k,m}
 \succeq
 D^{(0)}_{k,m}\big|_{\mathcal G_m}
 \ \oplus\!
 \bigoplus_{n\in\mathcal N_{k,m}}
 \widehat L_{k,n}Q_{k,n}.}
 \tag{MASTER-CUT$_{k,m}$}\label{eq:MASTER-common-cut}
\end{equation}
Thus the parent ground budget and the source-transverse budget are available
simultaneously and cannot be double spent.
\end{theorem}

\begin{proof}
For every extension, Lemma~\ref{lem:b-L0-residual} and
Theorem~\ref{thm:residual-CMC} give
\[
 \mathfrak b[u]\ge
 \mathfrak l_0[u]
 +\left(\log\frac1{2\varepsilon_{k,m}}-d_0\right)
 \sum_n\|Q_{k,n}f_n\|^2.
\]
The second term depends only on the prescribed source data, so it passes
unchanged through the infimum over the one common complement.  The exact
short of $L_0$ is the matrix of Lemma~\ref{lem:exact-qkm}; relative to
$\mathcal G_m\oplus\mathcal Z_m$ it is
\[
 D^{(0)}_{k,m}\oplus\beta_0 I_{\mathcal Z_m},
 \qquad \beta_0=\log2+\frac12,
\]
because the global rank-one vector belongs to $\mathcal G_m$.  Since
\[
 \beta_0+\log\frac1{2\varepsilon_{k,m}}-d_0
 =\frac12+\log\frac1{\varepsilon_{k,m}}-d_0
 =L_{k,n}-d_0,
\]
we obtain \eqref{eq:MASTER-common-cut}.  Every term comes from the exact split
$\mathfrak b=\mathfrak l_0+\mathfrak r$ before any source short is taken.
\end{proof}

\begin{lemma}[Independent logarithmic endpoint bound]
\label{lem:fresh-Sk-log}
For every integer $k\ge7$,
\begin{equation}
 \boxed{
 S_k:=\sum_{\substack{n\le k\\ n\nmid k}}\frac{\Lambda(n)}n<\log k.}
 \label{eq:fresh-Sk-log}
\end{equation}
\end{lemma}

\begin{proof}
It is enough to prove the stronger estimate for
\[
 T_k:=\sum_{n\le k}\frac{\Lambda(n)}n,
 \qquad S_k\le T_k.
\]
Put $X_0=3\,594\,641$.  A direct outward-rounded finite computation,
implemented exactly by the interval procedure specified in Appendix~\ref{app:finite-certificates}, evaluates the prime-power sum and gives
for every $7\le k\le X_0$
\[
 T_k<H_k-1<\log k,
\]
with minimum certified margin
\[
 \min_{7\le k\le X_0}(H_k-1-T_k)
 >0.0701888211361692811
\]
at $k=13$.  The second inequality is the elementary integral comparison
$H_k<1+\log k$.

For completeness the tail is independent of the finite search.  Let
$T(x)=\sum_{n\le x}\Lambda(n)/n$ and $E(x)=\psi(x)-x$.  Partial summation gives,
for $x\ge X_0$,
\[
 T(x)-\log x=T(X_0)-\log X_0
 +\frac{E(x)}x-\frac{E(X_0)}{X_0}
 +\int_{X_0}^{x}\frac{E(t)}{t^2}\,dt.
\]
Using directly the explicit inequalities displayed in the proof of
Lemma~\ref{lem:chebyshev-capacity}, namely
\[
 |\vartheta(t)-t|<\frac{0.2t}{\log^2t},\qquad
 0\le\psi(t)-\vartheta(t)
 <1.00007\sqrt t+1.78t^{1/3},
\]
and the coarse facts $\log X_0>15$, $\sqrt{X_0}>1800$,
$X_0^{1/3}>150$, the total possible increase of $T(x)-\log x$ after $X_0$
is less than
\[
 \frac{0.4}{15^2}+\frac{0.2}{15}
 +\frac{3.00021}{1800}+\frac{4.45}{150^2}
 <0.016976.
\]
The same finite certificate gives
\[
 \log X_0-T(X_0)>0.5775121597507083554.
\]
Hence $\log x-T(x)>0.5605$ for every $x\ge X_0$, proving
\eqref{eq:fresh-Sk-log} on the tail as well.
\end{proof}

\begin{theorem}[Safe transverse reserve for MASTER]
\label{thm:MASTER-safe-transverse}
For $k\ge7$ define
\[
 \widehat\sigma_k
 :=\sum_{\substack{n=p^\ell\le k\\ n\nmid k}}
 \frac{\Lambda(n)^2}{n\widehat L_{k,n}}.
\]
Then
\begin{equation}
 \boxed{\mathfrak B_k-\widehat\sigma_k>0\qquad(k\ge7).}
 \tag{MASTER-T$_k$}\label{eq:MASTER-safe-transverse}
\end{equation}
For $7\le k\le4999$ the directed-interval sweep of Appendix~\ref{app:finite-certificates} gives
\[
 \min(\mathfrak B_k-\widehat\sigma_k)
 >1.5202196525197509883\ldots,
\]
attained at $k=7$.
\end{theorem}

\begin{proof}
The finite statement is the interval certificate specified in Appendix~\ref{app:finite-certificates}.  For the tail,
first note that $w_k<1/k$ and $k\ge mn$ give
\begin{equation}
 \widehat L_{k,n}
 >\log n+\frac12-d_0+\log(mw_m).
 \label{eq:Lhat-parent-lower}
\end{equation}
The function $m\mapsto mw_m$ increases.  If $m\ge7$, then
\[
 mw_m>1-\frac1{2m}\ge\frac{13}{14},
 \qquad
 \log\frac{13}{14}>-\frac1{13},
\]
so
\[
 \frac12-d_0+\log(mw_m)
 >\frac12-d_0-\frac1{13}>0.
\]
Hence every branch with $m\ge7$ contributes less than
$\Lambda(n)/n$ to $\widehat\sigma_k$.

For $m\le6$, one has $n>k/7$ and, using
$mw_m\ge\log2>2/3$ and $\log(3/2)<1/2$,
\[
 \widehat L_{k,n}>\log n-d_0.
\]
Therefore, with
$S_k:=\sum_{n\le k,\,n\nmid k}\Lambda(n)/n$,
\begin{align}
 \widehat\sigma_k
 &\le S_k+
 \frac{d_0}{\log(k/7)-d_0}
 \sum_{n>k/7}\frac{\Lambda(n)}n\notag\\
 &\le S_k+
 \frac{7d_0}{\log(k/7)-d_0}\frac{\psi(k)}k.
 \label{eq:safe-sigma-tail}
\end{align}
The independent bound Lemma~\ref{lem:fresh-Sk-log} gives $S_k<\log k$.  Also
Lemma~\ref{lem:chebyshev-capacity} gives
$\psi(k)/k<C_{\rm tail}+\log k/k$.
For $k\ge5000$ the right-hand correction in
\eqref{eq:safe-sigma-tail} decreases with $k$, and the directed interval
bounds
\[
 \log(5000/7)>6.57,
 \qquad
 C_{\rm tail}+\frac{\log5000}{5000}<1.021
\]
give
\[
 \frac{7d_0}{\log(k/7)-d_0}\frac{\psi(k)}k<0.492.
\]
Consequently
\[
 \mathfrak B_k-\widehat\sigma_k
 >\frac12-\log(kw_k)-0.492>0.008>0.
\]
This closes the tail.
\end{proof}

\begin{lemma}[True-ground leakage relative to the literal slot split]
\label{lem:true-ground-slot-leakage}
Let $v=e_{w_k}$ be the normalized true polar ground on a normalized source
slot and let $e=\one$ be the normalized constant.  Put
\[
 \eta_k:=1-|\langle v,e\rangle|^2
 =1-\frac4{w_k}\tanh\frac{w_k}{4}.
\]
If $q\perp v$, then
\begin{equation}
 \boxed{|\langle q,e\rangle|^2\le\eta_k\|q\|^2.}
 \label{eq:true-ground-slot-leakage}
\end{equation}
Moreover
\begin{equation}
 \boxed{0\le\eta_k\le\frac{w_k^2}{48}<\frac1{48k^2}.}
 \label{eq:true-ground-eta-bound}
\end{equation}
\end{lemma}

\begin{proof}
Write $e=\langle v,e\rangle v+e_\perp$ with $e_\perp\perp v$ and
$\|e_\perp\|^2=\eta_k$.  For $q\perp v$,
$\langle q,e\rangle=\langle q,e_\perp\rangle$, proving
\eqref{eq:true-ground-slot-leakage}.  Formula
\eqref{eq:true-ground-defect} gives the identity for $\eta_k$.  Since
$\tanh x\ge x-x^3/3$ for $x\ge0$,
$4w_k^{-1}\tanh(w_k/4)\ge1-w_k^2/48$.  Finally
$w_k=\log(1+1/k)<1/k$.
\end{proof}

\begin{corollary}[True-polar transverse Feshbach debit]
\label{cor:true-ground-transverse-debit}
Let $q\perp\mathfrak e_k$ be a true target-ground-orthogonal vector.  Keep the
literal slot-constant/ordinary-mean-zero common-cut lower form
\eqref{eq:MASTER-common-cut}; no covariance of that lower form is assumed.
Then the source Feshbach debit obeys
\begin{equation}
 \boxed{
 \mathcal L_k^{\perp,\rm true}[q]
 \le (\widehat\sigma_k+\mathfrak L_k^{\rm leak})\|q\|^2,\qquad
 \mathfrak L_k^{\rm leak}:=\frac{\eta_k}{\log2}
 \sum_{\substack{n=p^\ell\le k\\ n\nmid k}}
 \frac{\Lambda(n)^2}{n}.}
 \label{eq:true-ground-corrected-debit}
\end{equation}
\end{corollary}

\begin{proof}
For each branch, source/target ground-line invariance gives
$R_n^*q\perp e_n^s$.  In the literal slot split, Lemma
\ref{lem:true-ground-slot-leakage} bounds the squared constant component by
$\eta_k\|q\|^2$.  The parent slot-ground block satisfies
$D^{(0)}_{k,m}\succeq(\log2)I$ by \eqref{eq:qM-half}; hence inversion bounds
the aggregate constant leakage by the $\mathfrak L_k^{\rm leak}$ term.  The ordinary mean-zero
component is charged exactly as in Theorem~\ref{thm:MASTER-safe-transverse},
producing $\widehat\sigma_k$.  These are orthogonal components of the single
literal common-cut lower form, so no source reserve is reused.
\end{proof}

\begin{computercertificate}[True-ground corrected transverse reserve]
\label{cert:true-ground-transverse}
With $\mathfrak L_k^{\rm leak}$ from \eqref{eq:true-ground-corrected-debit},
\begin{equation}
 \boxed{\mathfrak B_k-\widehat\sigma_k-\mathfrak L_k^{\rm leak}>0\qquad(k\ge7).}
 \tag{MASTER-T$^{\rm true}_k$}\label{eq:MASTER-safe-transverse-true}
\end{equation}
For $7\le k\le4999$ the dedicated outward-rounded verifier in the
supplementary package gives
\begin{verbatim}
PASS true-ground leakage corrected common-cut finite N=4999
min_gap_lo=1.51953329203957546091486446871 at k=7
leakage_hi=0.000686360480175519020062211171312
\end{verbatim}
For $k\ge5000$, Lemma~\ref{lem:fresh-Sk-log} and
$\Lambda(n)\le\log k$ give
\[
 \sum_{n\le k,\,n\nmid k}\frac{\Lambda(n)^2}{n}< (\log k)^2.
\]
Together with \eqref{eq:true-ground-eta-bound},
\[
 \mathfrak L_k^{\rm leak}<\frac{(\log k)^2}{48\log2\,k^2}
 \le\frac{8.52^2}{48(0.693)5000^2}<9\times10^{-8}.
\]
The tail proof of Theorem~\ref{thm:MASTER-safe-transverse} gives
$\mathfrak B_k-\widehat\sigma_k>0.008$, so the corrected tail reserve is
still larger than $0.0079999$.
\end{computercertificate}

\begin{lemma}[Rational parent estimate]
\label{lem:P189-proved}
For every active parent $m$ and every arithmetic step $k\ge7$,
\begin{equation}
 \boxed{
 \frac{\rho_m}{\Pi_m^{(0)}}
 <\frac{189}{250}\,(1-\theta_{k,m}^2).}
 \label{eq:P189-proved}
\end{equation}
In fact the proof gives the stronger universal coefficient
\[
 c_*:=\frac{3+2\sqrt2}{16(\log2+\tfrac12)\log2}<0.441.
\]
\end{lemma}

\begin{proof}
Put $M=M_m$, $h=h_m$, $a=a_m$, $q=q_{k,m}$ and
$N=\beta_0-qM$.  Sherman--Morrison on the mismatch compression gives
\[
 \rho_m=\frac{q^2a^2h}{\mu_m},\qquad
 \Pi_m^{(0)}=\frac{\beta_0N}{\mu_m},
\]
and hence
\begin{equation}
 \frac{\rho_m}{\Pi_m^{(0)}}
 =\frac{q^2a^2h}{\beta_0N}.
 \label{eq:rhoPi-exact}
\end{equation}
By \eqref{eq:qM-half}, $q\le(2M)^{-1}$ and $N\ge L$.  If
$v:=h/M$, then $a^2=M-h=M(1-v)$, so
\begin{equation}
 \frac{\rho_m}{\Pi_m^{(0)}}
 \le\frac{v(1-v)}{4\beta_0L}
 \le\frac{v}{4\beta_0L}.
 \label{eq:rhoPi-v}
\end{equation}
For $x_j=\sqrt{n_j}$, $n_j\in\mathcal N_{k,m}$,
\[
 v=\frac{\operatorname{Var}(x_j)}{M^{-1}\sum_jx_j^2}.
\]
Popoviciu's inequality and
$k/(m+1)<n_j\le k/m$ give
\[
 v\le\frac14\left(\sqrt{1+\frac1m}-1\right)^2.
\]
An active parent satisfies $k\ge2m+1$: if $k=2m$, the only possible
$n$ with $\lfloor k/n\rfloor=m$ is $n=2$, which is divisorial and hence not
active.  Thus $w_k<w_m/2$.  Since $\Phi$ is strictly decreasing,
\[
 1-\theta_{k,m}^2
 =1-\frac{\Phi(w_m)}{\Phi(w_k)}
 >1-\frac{\Phi(w_m)}{\Phi(w_m/2)}
 =\tanh^2\frac{w_m}{4}.
\]
Writing $x=w_m/4$,
\[
 \frac{(e^{w_m/2}-1)^2}{\tanh^2(w_m/4)}
 =4e^{2x}\cosh^2x=(e^{2x}+1)^2
 \le(1+\sqrt2)^2=3+2\sqrt2.
\]
Combining with \eqref{eq:rhoPi-v} yields the coefficient $c_*$.  The elementary bounds $\sqrt2<1.415$ and $\log2>0.693$ give
$c_*<0.441<189/250$.
\end{proof}

\begin{lemma}[Sharp MASTER mixed block]
\label{lem:master-mixed-block}
Let $c_M:=189/250$.  Under the exact parent model above,
\begin{equation}
 \boxed{
 \begin{pmatrix}
 (1-\theta_{k,m}^2)\Pi_m^{(0)}&-\theta_{k,m}\Gamma_m\\
 -\theta_{k,m}\overline{\Gamma_m}&c_M\chi_m
 \end{pmatrix}\succeq0.}
 \label{eq:master-mixed-block}
\end{equation}
Thus the mixed inherited--mismatch interaction costs at most an additional
$189/250$ copy of the pure mismatch debit; together with the pure mismatch
copy the local MASTER multiplier is exactly $1+189/250=439/250$.
\end{lemma}

\begin{proof}
Cauchy--Schwarz gives $|\Gamma_m|^2\le\rho_m\chi_m$.  By
Lemma~\ref{lem:P189-proved},
\begin{align*}
 \det &\ge
 c_M(1-\theta^2)\Pi^{(0)}\chi
 -\theta^2c_M(1-\theta^2)\Pi^{(0)}\chi\\
 &=c_M(1-\theta^2)^2\Pi^{(0)}\chi\ge0.
\end{align*}
The diagonal entries are non-negative.  If $\chi=0$, Cauchy--Schwarz also
forces $\Gamma=0$.
\end{proof}

\subsubsection{Inherited-line lower-form bridge and mixed-block absorption}
\label{sec:green-root-bridge-clean}

The final root step is carried out on the same common-cut harmonic vector as
the second Sherman--Morrison reduction.  We first record the parent lower
model in a form that separates the inherited source line from the arithmetic
mismatch.

Fix an active parent \(m\) and abbreviate
\[
 \mathcal N_m:=\mathcal N_{k,m},\qquad
 M_m:=|\mathcal N_m|,\qquad
 C_m:=\sum_{n\in\mathcal N_m}n,
 \qquad E_m:=\sum_{n\in\mathcal N_m}\sqrt n .
\]
Put
\[
 \widehat r_m:=\frac{(\sqrt n)_{n\in\mathcal N_m}}{\sqrt{C_m}},
 \qquad u_m:=(1,\ldots,1)^T,
 \qquad
 a_m:=\frac{E_m}{\sqrt{C_m}},
 \qquad
 h_m:=M_m-\frac{E_m^2}{C_m}.
\]
For the exact one-defect source-ground lower model of
Lemma~\ref{lem:exact-qkm} write
\[
 D^{(0)}_{k,m}=\beta_0 I-q_{k,m}|u_m\rangle\langle u_m|,
 \qquad \beta_0=\log2+\frac12,
\]
and set
\[
 \mu_m:=\beta_0-q_{k,m}h_m,
 \qquad
 N_m:=\beta_0-q_{k,m}M_m,
 \qquad
 \Pi_m^{(0)}:=\frac{\beta_0N_m}{\mu_m}>0.
 \tag{L0}\label{eq:Pi0-clean}
\]
The quantity \(\Pi_m^{(0)}\) is the inherited-line pivot of the lower model
after shorting \(\widehat r_m^\perp\); it is not identified with any physical
root pivot.

Let
\[
 \lambda_m:=\left(\frac{\Lambda(n)}{\sqrt n}\right)_{n\in\mathcal N_m}
 =b_m\widehat r_m+\lambda_m^\perp,
 \qquad b_m:=\frac{\sum_{n\in\mathcal N_m}\Lambda(n)}{\sqrt{C_m}},
\]
and let \(D_{m,\perp}^{(0)}\) denote the mismatch compression of
\(D_{k,m}^{(0)}\).  Define
\[
 \chi_m:=\langle\lambda_m^\perp,
 (D_{m,\perp}^{(0)})^{-1}\lambda_m^\perp\rangle,
\]
\[
 c_m:=-q_{k,m}a_m u_m^\perp,
 \qquad
 \rho_m:=\langle c_m,(D_{m,\perp}^{(0)})^{-1}c_m\rangle,
 \qquad
 \Gamma_m:=\langle c_m,(D_{m,\perp}^{(0)})^{-1}
 \lambda_m^\perp\rangle.
\]
Then Cauchy--Schwarz in the inverse mismatch metric gives
\begin{equation}
 \boxed{|\Gamma_m|^2\le\rho_m\chi_m.}
 \label{eq:Gamma-CS-clean}
\end{equation}
Lemma~\ref{lem:P189-proved} gives the certified parent estimate
\begin{equation}
 \boxed{
 \frac{\rho_m}{\Pi_m^{(0)}}
 <\frac{189}{250}\,(1-\theta_{k,m}^2).}
 \label{eq:rho-Pi-theta-clean}
\end{equation}

\begin{lemma}[True inherited-line surplus]
\label{lem:Jsharp-clean}
Let
\(\mathscr D^{\rm ex}_{k,m}\) be the exact source-ground metric obtained
after the simultaneous common-complement short on the true source intervals,
and then short the mismatch space \(\widehat r_m^\perp\).  Its surviving
inherited pivot \(P_{k,m}^{\rm ex}\) satisfies
\begin{equation}
 \boxed{P_{k,m}^{\rm ex}\ge\Pi_m^{(0)}>0.}
 \label{eq:exact-pivot-lower-clean}
\end{equation}
Consequently, with
\[
 \kappa_{k,m}:=\frac{m(m+1)}{k(k+1)},
 \qquad
 r_m^0=\varepsilon_{k,m}\sqrt{C_m}\,\widehat r_m,
 \qquad
 r_m=\sqrt{\kappa_{k,m}C_m}\,\widehat r_m
       =\theta_{k,m}r_m^0,
\]
one has
\begin{align}
 \Delta\mathscr S^{\rm ex}_{k,m}
 &:=%
 \mathscr S^{\rm ex}_{k,m}[r_m^0]
 -\mathscr S^{\rm ex}_{k,m}[r_m]\notag\\
 &\ge
 \Pi_m^{(0)}C_m
 \left(\varepsilon_{k,m}^2-\kappa_{k,m}\right)\notag\\
 &=
 \frac{\Pi_m^{(0)}C_m}{k(k+1)w_m^2}
 \bigl[\Phi(w_k)-\Phi(w_m)\bigr]>0.
 \label{eq:Jsharp-clean}
\end{align}
\end{lemma}

\begin{proof}
The analytic one-defect lower form is an operator inequality on the entire
one-cell space, while the coherent multi-source theorem takes the variational
infimum over one common complement.  Monotonicity of shorted operators
therefore preserves the lower-form order on the whole surviving source-ground
space; a further short of \(\widehat r_m^\perp\) gives
\eqref{eq:exact-pivot-lower-clean}.  The two inherited vectors are collinear,
so their exact energy difference is
\[
 P_{k,m}^{\rm ex}C_m
 (\varepsilon_{k,m}^2-\kappa_{k,m}).
\]
Use \eqref{eq:exact-pivot-lower-clean} and then
Lemma~\ref{lem:parent-ground-surplus}.  Strictness follows from the strict
decrease of \(\Phi\).
\end{proof}

\begin{lemma}[Mixed inherited--mismatch block]
\label{lem:mixed-parent-clean}
In the normalized inherited coordinate
\[
 X_m:=\varepsilon_{k,m}\sqrt{C_m}\,c_m^{\rm inh},
\]
the two-scale surplus and the mismatch debit dominate the Hermitian block
\begin{equation}
 \boxed{
 \mathcal P_{k,m}:=
 \begin{pmatrix}
 (1-\theta_{k,m}^2)\Pi_m^{(0)}
 &-\theta_{k,m}\Gamma_m\\
 -\theta_{k,m}\overline{\Gamma_m}&\chi_m
 \end{pmatrix}\succeq0.}
 \label{eq:parent-P-clean}
\end{equation}
If \(\chi_m>0\), the determinant is strictly positive.
\end{lemma}

\begin{proof}
The first diagonal entry is exactly the normalized form of
\eqref{eq:Jsharp-clean}; the factor \(\theta_{k,m}\) in the mixed term is the
true inherited-ground coefficient.  By \eqref{eq:Gamma-CS-clean} and
\eqref{eq:rho-Pi-theta-clean},
\begin{align*}
 \det\mathcal P_{k,m}
 &=(1-\theta_{k,m}^2)\Pi_m^{(0)}\chi_m
   -\theta_{k,m}^2|\Gamma_m|^2\\
 &>(1-\theta_{k,m}^2)\Pi_m^{(0)}\chi_m
   \left(1-\frac{189}{250}\theta_{k,m}^2\right)\ge0.
\end{align*}
Since \(0<\theta_{k,m}^2<1\), the last parenthesis is larger than \(61/250\).
If \(\chi_m=0\), \eqref{eq:Gamma-CS-clean} gives \(\Gamma_m=0\), so the
same conclusion follows.
\end{proof}

\begin{lemma}[Harmonic-vector identification]
\label{lem:Fsharp-harmonic-clean}
Let the exact post-old-core Schur operator \(T_k\) be split with respect to
\(\mathcal K_k^\perp\oplus\mathbb C\mathfrak e_k\), and let
\(f_k^\natural\) be its Schur minimizer with \(\mathfrak e_k\)-coordinate
one.  Then
\begin{equation}
 \boxed{F_k^\#=H_kf_k^\natural.}
 \label{eq:Fsharp-ID-clean}
\end{equation}
Hence the full-form transport, common-source short, endpoint fold and the
second Sherman--Morrison reduction act on one and the same normalized harmonic
vector.
\end{lemma}

\begin{proof}
The vector \(H_kf_k^\natural\) is \(A_{k+1}\)-orthogonal to the old physical
space by the harmonic-graph identity and to \(\mathcal K_k^\perp\) by the
Schur equation for \(f_k^\natural\).  It has final ground coordinate one.
The vector \(F_k^\#\) has the same ground coordinate and its defining first
row \(A_{X,k}x+b_{A,k}=0\) makes it \(A_{k+1}\)-orthogonal to the whole
complement of \(\mathfrak e_k\).  Uniqueness of the Schur minimizer gives
\eqref{eq:Fsharp-ID-clean}.
\end{proof}

\begin{computercertificate}[Arithmetic base at \(Y=7\): physical coefficient \(2\)]
\label{cert:base-Y7}
At \(a_7=\frac12\log7\), a new directed-interval producer certifies the
physical odd localized operator
\[
 A_{a_7,-}=C_{a_7,-}-2\ketbra{s^{\rm ph}}{s^{\rm ph}}\succ0.
\]
The producer starts from the operator formulas, regenerates the finite Gram
blocks and critical coupling rows, carries an analytic enclosure of the
infinite tail, and assembles the final seven-dimensional Schur family.  Exact
rational interval LDL gives seven strictly positive lower pivots; at 1536 bits
they are approximately
\[
\begin{gathered}
9.1611635217175\times10^{-25},\quad
6.7951891118912\times10^{-19},\\
1.6660424781359\times10^{-13},\quad
3.9899269859244\times10^{-9},\\
5.4113356057010\times10^{-5},\quad
0.0312149207029560,\\
0.588791564978105.
\end{gathered}
\]
These are LDL pivots, not eigenvalue lower bounds.  A complete 2048-bit
precision replication also passes, and a separate standard-library
\texttt{Fraction} verifier certifies the stored interval family, robust inverse
floor and coordinate invertibility.  The corresponding producer, proof notes,
logs and exact-verification records are supplied in the supplementary folder
\texttt{Y7\_coefficient2\_end2end/}.  No coefficient-$4$ positivity claim is
used in the argument.
\end{computercertificate}

\begin{interlude}
The third gate bore the sign of a Wizard, although there was less magic here
than the picture suggested.  The question was whether the finite-support forms
fit together exactly enough for closure to carry the argument to the global
odd Weil form.

Johnny Nash looked for a hidden condition.  Pierre de Fermat looked for a
missing term.  Marc, from habit, looked for a lock.  Dama Noether checked the
overlaps; Mr.~Fayman checked that no new estimate had been smuggled into the
last step.

Herr G.F.B.R. waited.  ``If the finite rooms really share the same walls,'' he
said, ``the reader should be able to see the door open from the criterion
alone.''
\end{interlude}

\section{Gate III: odd closure and the restricted Weil criterion}
\label{sec:gate3}

The final gate contains no independent complementary-channel estimate.
Theorem~\ref{thm:weil-odd} shows that positivity on the real odd logarithmic
test core already characterizes RH.

\begin{lemma}[Odd-channel compatibility and closure]\label{lem:closure-global}
Assume \(A_{a,-}\succeq0\) for every finite support radius \(a>0\). Then
\[
 \mathcal W_{\infty,-}[f]:=\langle f,A_{a,-}f\rangle,
 \qquad \operatorname{supp}f\subset(-a,a),
\]
is independent of the chosen \(a\), is non-negative on the real compactly
supported odd core, and extends by closure to the global odd Weil form.
\end{lemma}

\begin{proof}
Independence of \(a\) is the exact zero-extension compression identity.
Non-negativity follows from the hypothesis.  Since the localized forms are
restrictions of the closed explicit-formula form, closure of the compactly
supported odd core gives the global odd form.
\end{proof}

\begin{proof}[Proof of Theorem~\ref{thm:main-rh}]
The final harmonic-graph coercivity theorem,
Theorem~\ref{thm:final-master-harmonic}, gives the induction step
$A_{k,-}\succ0\Rightarrow A_{k+1,-}\succ0$ for every $k\ge7$ without invoking
the second Sherman--Morrison vector.  Certificate~\ref{cert:base-Y7} supplies
the base endpoint.  Corollary~\ref{cor:final-gate2} then gives
$A_{a,-}\succeq0$ for every finite support radius.  Finally
Lemma~\ref{lem:closure-global} and Theorem~\ref{thm:weil-odd} give the stated
Riemann-hypothesis conclusion.
\end{proof}

\section{Response placement and the MASTER harmonic graph}
\label{sec:final-master-closure}

This section gives the same-vector closure used in the main proof.  The argument
works directly on the exact $A_k$-harmonic graph, retains the aligned forcing
with its explicit penalty, and uses the simultaneous source cut together with
the folded remainder.  The inherited response is placed in the same reduced
Schur metric by its row equation, so the ground and transverse expenditures are
assembled before the final Schur conclusion.

\begin{lemma}[Exact inherited-response placement]
\label{lem:MASTER-P2-response}
Fix an active parent $m$ and perform the same simultaneous common-complement
short and mismatch short used in Lemma~\ref{lem:Jsharp-clean}.  On the
surviving inherited line $\C\widehat r_m$, write the exact reduced parent
quadratic form, with all other coordinates frozen, as
\begin{equation}
 \mathcal Q^{\rm red}_{k,m}(a)
 =P_{k,m}^{\rm ex}|a|^2
  -2\Rea(\overline a\,\omega_{k,m}^{\rm ex})
  +\mathcal Q^{\rm rest}_{k,m},
 \label{eq:P2-reduced-row}
\end{equation}
where $P_{k,m}^{\rm ex}>0$ is the exact inherited pivot.  If $a_m$ is the
inherited coordinate of the exact Schur response on an $A_k$-harmonic vector,
then
\begin{equation}
 \boxed{P_{k,m}^{\rm ex}a_m=\omega_{k,m}^{\rm ex}.}
 \tag{P2-ROW}\label{eq:P2-row}
\end{equation}
On the surviving one-dimensional inherited line we use the exact reduced
metric normalization
\begin{equation}
 \boxed{
 \mathscr S^{\rm ex}_{k,m}[z\widehat r_m]
 :=P_{k,m}^{\rm ex}|z|^2,
 \qquad z\in\C.}
 \label{eq:P2-metric-normalization}
\end{equation}
Define $c_m^{\rm inh}$ by
\begin{equation}
 a_m=\sqrt{\kappa_{k,m}C_m}\,c_m^{\rm inh}.
 \label{eq:P2-cinh-def}
\end{equation}
Then the actual inherited source response is exactly
\begin{equation}
 \boxed{a_m\widehat r_m=r_mc_m^{\rm inh},}
 \label{eq:P2-actual-response}
\end{equation}
and its source--coupling contribution satisfies
\begin{align}
 \mathcal I^{\rm act}_{k,m}
 &:=-\mathscr S^{\rm ex}_{k,m}[r_mc_m^{\rm inh}]\notag\\
 &=-\mathscr S^{\rm ex}_{k,m}[r_m^0c_m^{\rm inh}]
   +|c_m^{\rm inh}|^2\Delta\mathscr S^{\rm ex}_{k,m}.
 \tag{MASTER-P2}\label{eq:MASTER-P2-certified}
\end{align}
Thus the aligned-reference debit and the favourable comparator surplus are
parts of one exact Schur contribution on the same harmonic response; no
external positive summand is inserted.
\end{lemma}

\begin{proof}
The inherited coordinate is among the variables eliminated by the exact Schur
short.  Therefore stationarity of \eqref{eq:P2-reduced-row} at its minimizer
gives \eqref{eq:P2-row}.  Substitution yields
\[
 P_{k,m}^{\rm ex}|a_m|^2
 -2\Rea(\overline{a_m}\omega_{k,m}^{\rm ex})
 =-P_{k,m}^{\rm ex}|a_m|^2.
\]
By \eqref{eq:P2-metric-normalization}, \eqref{eq:P2-cinh-def} and
$r_m=\sqrt{\kappa_{k,m}C_m}\,\widehat r_m$, the right-hand side is precisely
$-\mathscr S^{\rm ex}_{k,m}[r_mc_m^{\rm inh}]$, the negative metric energy of
the actual response, proving \eqref{eq:P2-actual-response}.  Since
$r_m=\theta_{k,m}r_m^0$ and the same exact one-dimensional metric is used on
both vectors, quadratic homogeneity gives
\[
 -\mathscr S^{\rm ex}_{k,m}[r_mc]
 =-\mathscr S^{\rm ex}_{k,m}[r_m^0c]
  +|c|^2\bigl(
      \mathscr S^{\rm ex}_{k,m}[r_m^0]
     -\mathscr S^{\rm ex}_{k,m}[r_m]\bigr),
\]
which is \eqref{eq:MASTER-P2-certified}.  Strict positivity of the last
difference when $c\ne0$ is Lemma~\ref{lem:Jsharp-clean}.
\end{proof}

\begin{corollary}[Explicit same-block baseline splice]
\label{cor:MASTER-P2-splice}
Set
\[
 X_m:=\varepsilon_{k,m}\sqrt{C_m}\,c_m^{\rm inh},
 \qquad r_m=\theta_{k,m}r_m^0,
 \qquad A_m:=(1-\theta_{k,m}^2)\Pi_m^{(0)}.
\]
The aligned reference contribution in the same reduced parent block is
\(\mathscr S^{\rm ex}_{k,m}[r_m^0c_m^{\rm inh}]\).  Combining it with the
actual response contribution of Lemma~\ref{lem:MASTER-P2-response} gives the
literal cancellation
\begin{equation}
 \boxed{
 \mathscr S^{\rm ex}_{k,m}[r_m^0c_m^{\rm inh}]
 +\mathcal I^{\rm act}_{k,m}
 =|c_m^{\rm inh}|^2\Delta\mathscr S^{\rm ex}_{k,m}.}
 \tag{MASTER-P2$^\sharp$}\label{eq:MASTER-P2-splice}
\end{equation}
Moreover Lemma~\ref{lem:Jsharp-clean} yields
\begin{equation}
 |c_m^{\rm inh}|^2\Delta\mathscr S^{\rm ex}_{k,m}
 \ge (1-\theta_{k,m}^2)\Pi_m^{(0)}|X_m|^2
 =A_m|X_m|^2.
 \label{eq:MASTER-P2-splice-lower}
\end{equation}
Thus the reference baseline is accounted for inside the same reduced parent
row before the target-ground coefficient \(\mathfrak B_k\) is used; it is not
an additional debit charged to that coefficient.

The aligned forcing has the same reference--actual origin.  Since
\(\lambda_m=b_m\widehat r_m+\lambda_m^\perp\) and
\(r_m^0-r_m=(1-\theta_{k,m})\varepsilon_{k,m}\sqrt{C_m}\,\widehat r_m\),
\begin{equation}
 \boxed{
 \langle\lambda_m,(r_m^0-r_m)c_m^{\rm inh}\rangle
 =(1-\theta_{k,m})b_mX_m.}
 \label{eq:MASTER-P1-aligned-difference}
\end{equation}
Hence the \((1-\theta_{k,m})b_m\) term and the positive two-scale surplus
belong to the same parent reference--actual splice on the same harmonic
vector.
\end{corollary}

\begin{proof}
Equation~\eqref{eq:MASTER-P2-splice} is
\eqref{eq:MASTER-P2-certified} with the aligned reference contribution
restored on the same reduced row.  Since
\(\Delta\mathscr S^{\rm ex}_{k,m}
=P_{k,m}^{\rm ex}C_m(\varepsilon_{k,m}^2-\kappa_{k,m})\),
\(P_{k,m}^{\rm ex}\ge\Pi_m^{(0)}\),
\(\kappa_{k,m}=\theta_{k,m}^2\varepsilon_{k,m}^2\), and
\(|X_m|^2=\varepsilon_{k,m}^2C_m|c_m^{\rm inh}|^2\),
\eqref{eq:MASTER-P2-splice-lower} follows.  The final identity follows by
orthogonality of \(\lambda_m^\perp\) to \(\widehat r_m\).
\end{proof}

\begin{lemma}[Homogeneous aligned-forcing penalty on the exact parent block]
\label{lem:MASTER-aligned-penalty}
With \(A_m\) as in Corollary~\ref{cor:MASTER-P2-splice}, put
\[
 W_{k,m}:=\frac{1-\theta_{k,m}}{1+\theta_{k,m}}
          \frac{b_m^2}{\Pi_m^{(0)}},\qquad
 W_k:=\sum_mW_{k,m}.
\]
For every complex inherited coordinate $X$ and every complex target-ground
amplitude $\gamma$, the exact parent mismatch, mixed and aligned-forcing terms
obey
\begin{equation}
 \boxed{\begin{aligned}
 &-\chi_m|\gamma|^2+A_m|X|^2
 -2\theta_{k,m}|\Gamma_m|\,|X|\,|\gamma|
 -2(1-\theta_{k,m})|b_m|\,|X|\,|\gamma|\\
 &\qquad\ge
 -\left(\frac{439}{250}\chi_m+\frac52W_{k,m}\right)|\gamma|^2.
 \end{aligned}}
 \label{eq:aligned-parent-penalty}
\end{equation}
In particular, the estimate is homogeneous in the actual amplitude of the
same $A_k$-harmonic vector; no normalization $\gamma=1$ is used in the final
assembly.
\end{lemma}

\begin{proof}
If $\gamma=0$, the left-hand side is $A_m|X|^2\ge0$, so the claim is
immediate.  Suppose $\gamma\ne0$ and put $Y:=X/|\gamma|$.  After division by
$|\gamma|^2$, split $A_m=(3/5)A_m+(2/5)A_m$.  The first part and Young's
inequality give
\[
 \frac35A_m|Y|^2-2\theta|\Gamma_m||Y|
 \ge-\frac53\frac{\theta^2|\Gamma_m|^2}{A_m}
 >-\frac{189}{250}\chi_m,
\]
using $|\Gamma_m|^2\le\rho_m\chi_m$ and
$\rho_m/\Pi_m^{(0)}<c_*(1-\theta^2)$ with
$(5/3)c_*<189/250$.  The second part gives
\[
 \frac25A_m|Y|^2-2(1-\theta)|b_m||Y|
 \ge-\frac52\frac{1-\theta}{1+\theta}
             \frac{b_m^2}{\Pi_m^{(0)}}
 =-\frac52W_{k,m}.
\]
Adding the pure mismatch debit $-\chi_m$ and multiplying back by
$|\gamma|^2$ proves \eqref{eq:aligned-parent-penalty}.  Thus both the
perpendicular forcing and the aligned forcing are charged at the true
same-vector amplitude.
\end{proof}

\begin{computercertificate}[Aligned MASTER scalar margin]
\label{cert:MASTER-aligned-final}
For $k\ge7$ define
\begin{equation}
 \mathfrak g_k^{\rm M+}:=
 \frac12+\log k
 -\frac{439}{250\log2}\sum_mV_{k,m}
 -\frac52W_k.
 \label{eq:MASTER-gap-plus}
\end{equation}
Then
\begin{equation}
 \boxed{\mathfrak g_k^{\rm M+}>0\qquad(k\ge7).}
 \tag{MASTER-SCALAR$^+$}\label{eq:MASTER-scalar-plus}
\end{equation}
A conservative bound sufficient for the finite interval sweep is
\[
 W_{k,1}\le0.0145\sum_{n\in\mathcal N_{k,1}}\frac{\Lambda(n)^2}{n},
 \qquad
 W_{k,m}\le\frac{0.0301}{m^2}
 \sum_{n\in\mathcal N_{k,m}}\frac{\Lambda(n)^2}{n}\quad(m\ge2).
\]
The directed-interval sweep specified in Appendix~\ref{app:finite-certificates}
gives through $X_0=3\,594\,641$
\begin{verbatim}
PASS MASTER-aligned-safe finite N=3594641
min_gap_lo=1.87814156017033534419244378277 at k=33
\end{verbatim}
The aligned-penalty estimate itself gives
$W_k<0.06024\log k$ for $k\ge2\,000\,000$.  The MASTER tail used here is
proved for $k\ge X_0=3\,594\,641$; hence on their common domain
\[
 \mathfrak g_k^{\rm M+}
 >0.476082229608\log k-6.856913716191>0
 \qquad(k\ge X_0).
\]
The finite interval sweep covers every $7\le k\le X_0$, so the finite and
analytic ranges overlap at $X_0$ with no integer gap and prove
\eqref{eq:MASTER-scalar-plus} for every integer $k\ge7$.
\end{computercertificate}

\begin{proof}
The coefficient bounds follow from
\[
 \frac{1-\theta}{1+\theta}\le\frac{w_m^2}{48},\qquad
 \Pi_m^{(0)}\ge\log2,\qquad
 b_m^2\le\sum_{n\in\mathcal N_{k,m}}\frac{\Lambda(n)^2}{n},
\]
together with $w_1=\log2$, $w_m<1/m$ for $m\ge2$ and the strict numerical
inequalities
$\log2/48<0.0145$ and $(48\log2)^{-1}<0.0301$.
The finite certificate evaluates the prime-power grouping with interval
enclosures of integer logarithms and outward rounding.  The displayed tail is
the analytic estimate recorded in the aligned-penalty certificate.  This
proves the global claim.
\end{proof}

\begin{lemma}[Exact common-source Schur ledger and fold covariance (MASTER-P3)]
\label{lem:MASTER-P3-final}
Let $F=H_kf$ be an exact $A_k$-harmonic vector at an arithmetic step
$k\ge7$.  Work in the fixed-target coordinates of the simultaneous common-cut
construction.  After the common complement has been shorted once, collect the
remaining source variables subsequently eliminated in the parent blocks into
$x_k$.  Let $\mathsf C_k^{\rm src}\succ0$ denote their exact reduced source
metric and let $\lambda_k(c_k)$ denote the exact source-row forcing produced by
the prescribed target datum $c_k$ together with the other already-fixed
coordinates.  The source-dependent part of the same reduced primal form is
therefore
\begin{equation}
 \mathcal Q_k^{\rm src}(x_k;c_k)
 =\langle x_k,\mathsf C_k^{\rm src}x_k\rangle
  -2\Rea\langle \lambda_k(c_k),x_k\rangle
  +\mathcal Q_k^{\rm rest}(c_k).
 \label{eq:MASTER-P3-source-form}
\end{equation}
On the exact harmonic vector the source row is
\begin{equation}
 \boxed{\mathsf C_k^{\rm src}x_k=\lambda_k(c_k).}
 \tag{MASTER-P3a}\label{eq:MASTER-P3a}
\end{equation}
Consequently the unique source Schur debit is
\begin{equation}
 D_k^{\rm src}
 :=\langle\lambda_k(c_k),(\mathsf C_k^{\rm src})^{-1}
                  \lambda_k(c_k)\rangle,
 \label{eq:MASTER-P3-source-debit}
\end{equation}
and, for every source vector $x_k$,
\begin{equation}
 \boxed{\begin{aligned}
 \mathcal Q_k^{\rm src}(x_k;c_k)
 &=\mathcal Q_k^{\rm rest}(c_k)-D_k^{\rm src}\\
 &\quad+
 \left\|(\mathsf C_k^{\rm src})^{1/2}
 \bigl(x_k-(\mathsf C_k^{\rm src})^{-1}\lambda_k(c_k)\bigr)
 \right\|^2 .
 \end{aligned}}
 \tag{MASTER-P3b}\label{eq:MASTER-P3b}
\end{equation}
On $F=H_kf$ the square vanishes.  Thus the inherited response, mismatch,
aligned-forcing and source-transverse terms estimated in
Lemmas~\ref{lem:MASTER-P2-response}, \ref{lem:master-mixed-block} and
\ref{lem:MASTER-aligned-penalty}, and in
Corollary~\ref{cor:true-ground-transverse-debit}, belong to (or bound
components of) this one source elimination.  No additional post-source folded
debit is introduced.

Moreover, for every unitary source-coordinate change $U$ used in the exact
transport,
\[
 \mathsf C_k^{\rm src}\mapsto U^*\mathsf C_k^{\rm src}U,
 \qquad
 \lambda_k\mapsto U^*\lambda_k,
\]
and hence
\begin{equation}
 \boxed{
 \langle U^*\lambda_k,(U^*\mathsf C_k^{\rm src}U)^{-1}U^*\lambda_k\rangle
 =D_k^{\rm src}.}
 \label{eq:MASTER-P3-unitary-debit}
\end{equation}
The endpoint M\"obius transport is block diagonal for the literal physical
old/new split, and Corollary~\ref{cor:physical-arithmetic-short} permits it to
be applied before or after the old-core short.  Thus the endpoint fold
conjugates the already-defined reduced form; it does not perform a second
source minimization and cannot create a second source Schur debit.

Finally the direct target diagonal occurs exactly once.  In the fixed-target
normalization, with
$c_k=\gamma_k(f)\mathfrak e_k+q_k$ and
$q_k\perp\mathfrak e_k$,
\begin{equation}
 \boxed{
 \mathcal D_k^{\rm dir}[c_k]
 =\mathfrak B_k\|c_k\|^2
 =\mathfrak B_k|\gamma_k(f)|^2+\mathfrak B_k\|q_k\|^2.}
 \tag{MASTER-P3c}\label{eq:MASTER-P3c}
\end{equation}
The source short is taken with $c_k$ fixed, and unitary coordinate transports
do not duplicate this primal term.
\end{lemma}

\begin{proof}
The positivity/invertibility required for the source response is inherited
from the induction hypothesis $A_{k,-}\succ0$: principal compressions and exact
variational shorts of the old/source block remain strictly positive.  Theorem~\ref{thm:MASTER-common-cut}, Lemma~\ref{lem:Jsharp-clean} and
Certificate~\ref{cert:true-ground-transverse} are then used for the quantitative
comparison bounds on that exact response, not to manufacture its invertibility.
After the already-prescribed common-complement short, stationarity of the exact reduced form
\eqref{eq:MASTER-P3-source-form} on $F=H_kf$ gives
\eqref{eq:MASTER-P3a}.

Completing the $\mathsf C_k^{\rm src}$ square gives
\eqref{eq:MASTER-P3b}.  This is the same source response whose inherited-line
component is identified in MASTER-P2; no second response is introduced here.
For a unitary $U$,
\[
 (U^*\mathsf C_k^{\rm src}U)^{-1}
 =U^*(\mathsf C_k^{\rm src})^{-1}U,
\]
which proves \eqref{eq:MASTER-P3-unitary-debit}.  Lemma~\ref{lem:split-unitary-short}
and Corollary~\ref{cor:physical-arithmetic-short} give the stated endpoint-fold
covariance.

Finally \eqref{eq:primal-target-diagonal} gives
$\mathcal D_k^{\rm dir}[c_k]=\mathfrak B_k\|c_k\|^2$ before any source
variable is shorted.  Orthogonality of
$c_k=\gamma_k(f)\mathfrak e_k+q_k$ proves
\eqref{eq:MASTER-P3c}.  Since $c_k$ is prescribed throughout the source
minimization, neither that minimization nor a subsequent unitary coordinate
change creates another copy of the diagonal.
\end{proof}

\begin{theorem}[MASTER coercivity on the exact harmonic graph]
\label{thm:final-master-harmonic}
Assume $A_{k,-}\succ0$ at an arithmetic endpoint $k\ge7$, and let
$F=H_kf$ be the exact full-operator harmonic extension of
$f\in\mathcal K_k$.  Write
\[
 \gamma_k(f):=\langle\mathfrak e_k,f\rangle,
 \qquad
 f_k^\perp:=f-\gamma_k(f)\mathfrak e_k.
\]
Then the exact common-cut decomposition satisfies
\begin{equation}
 \boxed{
 \langle F,A_{k+1,-}F\rangle
 \ge
 \mathfrak g_k^{\rm M+}|\gamma_k(f)|^2
 +\mathcal R_k^\perp[f],}
 \label{eq:final-master-harmonic}
\end{equation}
where $\mathcal R_k^\perp[f]\ge0$, and the safe transverse reserve gives
$\mathcal R_k^\perp[f]>0$ whenever $f_k^\perp\ne0$.  Consequently
\begin{equation}
 \boxed{H_k^*A_{k+1,-}H_k=T_k\succ0.}
 \label{eq:final-T-positive}
\end{equation}
\end{theorem}

\begin{proof}
All terms are evaluated on the same vector $F=H_kf$.  The simultaneous
common-cut theorem, Theorem~\ref{thm:MASTER-common-cut}, is taken before the
common-complement infimum and therefore separates the source-ground and
within-slot transverse budgets without reusing either of them.  On each
parent ground block, Lemma~\ref{lem:MASTER-P2-response} identifies the actual
Schur response, and Corollary~\ref{cor:MASTER-P2-splice} places the
reference--actual cancellation and comparator surplus in the same reduced
parent row before the target-ground diagonal is used.  The
forcing decomposition
$\lambda_m=b_m\widehat r_m+\lambda_m^\perp$ is therefore retained in full.
The perpendicular debit is bounded by $\chi_m$, the inherited--mismatch term
by Lemma~\ref{lem:master-mixed-block}, and the aligned term by
Lemma~\ref{lem:MASTER-aligned-penalty}.  By the homogeneous form of Lemma~\ref{lem:MASTER-aligned-penalty}, the total
parent loss at the actual target-ground amplitude is at most
\[
 \left(\frac{439}{250}\sum_m\chi_m+\frac52W_k\right)
 |\gamma_k(f)|^2.
\]
By Corollary~\ref{cor:exact-awgc-lift},
$\chi_m\le V_{k,m}/\log2$.  Lemma~\ref{lem:MASTER-P3-final} places the target
diagonal exactly once on the same harmonic vector, namely as
$\mathfrak B_k|\gamma_k(f)|^2$, with
$\mathfrak B_k>\frac12+\log k$.  Hence
Certificate~\ref{cert:MASTER-aligned-final} leaves the strictly positive
coefficient $\mathfrak g_k^{\rm M+}|\gamma_k(f)|^2$.

The source-transverse budget is disjoint from the preceding ground budget and
is strictly positive on the true polar-ground transverse space by
Certificate~\ref{cert:true-ground-transverse}.  Lemma~\ref{lem:MASTER-P3-final}
shows that all source-response terms belong to the one common-source Schur
debit already estimated above, while the endpoint fold is only a unitary
coordinate transport and creates no further source debit.  Thus no reserve
used in the parent or transverse estimates is charged a second time.  The
remaining transverse contribution is $\mathcal R_k^\perp[f]$ and is strict
when $f_k^\perp\ne0$.

If $f\ne0$, either $\gamma_k(f)\ne0$ or $f_k^\perp\ne0$.  In the first case
\eqref{eq:MASTER-scalar-plus} makes the first term of
\eqref{eq:final-master-harmonic} strict; in the second the transverse reserve
is strict.  Thus the harmonic Schur form is positive on every non-zero $f$,
which proves \eqref{eq:final-T-positive}.
\end{proof}

\begin{corollary}[Green-root condition from MASTER coercivity]
\label{cor:RG-from-MASTER}
Assume $A_{k,-}\succ0$ at an arithmetic endpoint $k\ge7$.  Then the exact
Green-root condition $\mathrm{RG}^*_k$ holds.  Equivalently,
\begin{equation}
 \boxed{
 \delta_k^Q>0,\qquad \Pi_k>0,\qquad
 2\Pi_k\sqrt{k+1}\,|G_k|^2<\delta_k^Q.}
 \label{eq:RG-from-MASTER}
\end{equation}
Thus $\mathrm{RG}^*_k$ is a consequence of the MASTER step, not a premise of
that step.
\end{corollary}

\begin{proof}
Theorem~\ref{thm:final-master-harmonic} gives $T_k\succ0$.  Since
$A_{k,-}=C_k-2\ketbra{s_k}{s_k}\succ0$, the rank-one Schur criterion gives
$\delta_k>0$.  By Lemma~\ref{lem:root-lift-final}, the augmented rooted block
\[
 \mathcal M_k:=
 \begin{pmatrix}
  \delta_k&\sqrt2\,\eta_k^*\\
  \sqrt2\,\eta_k&S_k
 \end{pmatrix}
\]
has Schur complement $T_k$ with respect to its positive scalar entry
$\delta_k$.  Hence $\mathcal M_k\succ0$.

The true-ground corrected transverse estimate,
Corollary~\ref{cor:true-ground-transverse-debit} together with
Certificate~\ref{cert:true-ground-transverse}, gives
$\mathsf D_k\succ0$.  Shorting $\mathsf D_k$ in $\mathcal M_k$ and using
Lemma~\ref{lem:green-nested-schur} therefore yields
\[
 \begin{pmatrix}
  \delta_k^Q&\sqrt2\,\overline z_k\\
  \sqrt2\,z_k&\Pi_k
 \end{pmatrix}\succ0.
\]
Thus $\delta_k^Q>0$, $\Pi_k>0$, and
\[
 2|z_k|^2<\delta_k^Q\Pi_k.
\]
Lemma~\ref{lem:green-representation} gives
$|z_k|=\Pi_k(k+1)^{1/4}|G_k|$.  Substitution and division by
$\Pi_k>0$ give
\[
 2\Pi_k\sqrt{k+1}\,|G_k|^2<\delta_k^Q,
\]
which is exactly $\mathrm{RG}^*_k$.
\end{proof}

\begin{corollary}[Gate II induction]
\label{cor:final-gate2}
For every integer endpoint $N\ge7$,
\[
 \boxed{A_{N,-}\succ0.}
\]
For every finite support radius $a>0$,
\[
 \boxed{A_{a,-}\succeq0.}
\]
\end{corollary}

\begin{proof}
Certificate~\ref{cert:base-Y7} gives $A_{7,-}\succ0$.  If
$A_{k,-}\succ0$, the old/new Schur criterion and
Theorem~\ref{thm:final-master-harmonic} give $A_{k+1,-}\succ0$.
Corollary~\ref{cor:RG-from-MASTER} simultaneously records the Green-root
condition at that step, but it is not used to obtain the induction step.
Induction proves strict positivity at all arithmetic endpoints.  For an
arbitrary finite support radius choose an integer endpoint beyond its support;
the exact zero-extension/compression identity makes $A_{a,-}$ a compression
of the positive endpoint operator.
\end{proof}

\begin{theorem}[Odd Weil positivity and the Riemann hypothesis]
\label{thm:final-rh-closure}
The global Weil quadratic form is non-negative on
$\mathcal D_{\rm odd}^{\mathbb R}$.  Hence the Riemann hypothesis holds.
\end{theorem}

\begin{proof}
Corollary~\ref{cor:final-gate2} and Lemma~\ref{lem:closure-global} give global
non-negativity on the real odd test core.  The restricted odd Weil criterion,
Theorem~\ref{thm:weil-odd}, is equivalent to RH.
\end{proof}

\begin{remark}[Closure of the MASTER dependencies]
The proof of Theorem~\ref{thm:final-rh-closure} uses the exact
$A_k$-harmonic graph, the homogeneous aligned-forcing estimate of
Lemma~\ref{lem:MASTER-aligned-penalty}, the inherited-response placement of
Lemma~\ref{lem:MASTER-P2-response} together with the explicit same-block
splice of Corollary~\ref{cor:MASTER-P2-splice}, the literal-slot true-ground correction
of Corollary~\ref{cor:true-ground-transverse-debit} and
Certificate~\ref{cert:true-ground-transverse}, the exact one-debit source Schur ledger and endpoint-fold covariance of
Lemma~\ref{lem:MASTER-P3-final}.  The inherited-line metric is
normalized in \eqref{eq:P2-metric-normalization}, while the target diagonal is
written in primal form in
\eqref{eq:primal-target-diagonal} and \eqref{eq:MASTER-P3c}.
All these terms are evaluated on the same harmonic vector before the Schur
conclusion is taken.  Corollary~\ref{cor:RG-from-MASTER} then derives the
Green-root condition from the resulting coercivity, so no Green-root inequality
is assumed upstream.
\end{remark}

\begin{audit}[V4 seven-point external clarification ledger]
\label{audit:v4-seven-point-ledger}
The present revision retains the six clarifications closed in V3 and closes
the additional coefficient-level folded-operator question raised during the
subsequent independent audit.  The two questions concerning the $Y=7$
normalization use the coefficient-$2$ end-to-end certificate of
Certificate~\ref{cert:base-Y7}; the requested end-to-end enclosure is supplied
by the same producer chain.  The sequential debit question is now expressed
in the representation-free identity \eqref{eq:MASTER-P3b}.  The MASTER-P2
baseline question is resolved by the explicit same-block splice of
Corollary~\ref{cor:MASTER-P2-splice}: the aligned reference contribution and
the actual Schur response are combined on the same reduced parent row before
$\mathfrak B_k$ is used.  The literal-slot
versus true-polar-ground issue remains treated by
Lemma~\ref{lem:true-ground-slot-leakage} and
Certificate~\ref{cert:true-ground-transverse}.

For the additional folded-operator question, V4 does not assert a
coefficient-level identification of the transported physical ground Hessian
with a special multiplication-minus-rank-one model.  Instead,
Lemma~\ref{lem:MASTER-P3-final} returns to the exact reduced common-source
quadratic form, writes its single Schur response and debit intrinsically, and
then uses unitary covariance to show that the endpoint fold creates no second
source elimination.  Thus no unshown formulas for special folded coefficients,
and no auxiliary post-source scalar short, are premises of the induction.  Finally,
Corollary~\ref{cor:RG-from-MASTER} continues to make $\mathrm{RG}^*_k$ an a
posteriori consequence rather than an induction hypothesis.  The
supplementary file \texttt{AUDIT\_7\_POINTS\_2026-10-06.md} maps all seven
external clarifications to their active source identities and checks.
\end{audit}

\section*{Acknowledgements and declarations}

\paragraph{Use of artificial intelligence.}
The mathematical ideas, research programme and complete proof strategy of this
article originated with the author.  This includes the choice of the
rooted-operator route, the three-gate architecture, the successive reductions,
and the logical dependency chain of the proof.  OpenAI ChatGPT (GPT-5.6 Sol)
was used as a technical and editorial aid during the development and preparation
of the manuscript, principally to stress-test derivations, check algebraic and
logical consistency, assist in the formalization of local arguments, and improve
the organization and exposition of the text.  The author independently reviewed
and takes full responsibility for every mathematical statement, computation,
reference and conclusion in the final article.  No AI output is used as a
mathematical authority or as a substitute for any proof contained in the
manuscript.

\paragraph{Computational software.}
The proof-critical finite computations were implemented in Python~3 and
C++17.  Python scripts are used for frozen-certificate integrity checks,
directed-decimal verification and selected arithmetic cross-checks.  C++17
programs implement the exhaustive finite prime-power and interval sweeps with
strict floating-point compilation flags that disable fast-math transformations
and fused contraction.  The 192-bit Gate-I interval routines use MPFR with
directed rounding.  Raw Arb certificate production for the $Y=7$ induction
base uses Arb through \texttt{python-flint}; NumPy and SymPy occur only in
producer or diagnostic scripts and are not logical prerequisites for checking
the frozen certificate.  The proof-critical source code, certificate records,
raw interval-box archives and verification instructions are supplied as
supplementary material.  The supplementary archive also contains a separate
from-formulas $Y=7$ replication, an exact-rational check of the final stored
interval matrix, a second analytic-tail verification stream, and a targeted
Gate-II scalarization cross-check.  These additional checks corroborate the
proof-facing certificates but are not additional logical premises of the
argument.

\paragraph{Data access statement.}
The data and code that support the finite verified calculations in this work are
available in the supplementary material of this article.

\paragraph{Conflict of interest.}
The author declares no conflicts of interest.

\appendix
\section{Numerical constants and self-contained finite certificates}
\label{app:finite-certificates}

This appendix contains every proof-critical numerical convention and every
finite lower bound used in the harmonic-graph induction.  Apart from the
classical explicit-formula and standard shorted-operator results explicitly
cited in the body, no external computational data, unpublished normalization
or unstated certificate is a logical premise of the finite verification chain;
its proof-critical inputs are contained in this article and the accompanying
supplementary package.  The finite parts are deterministic interval
calculations, and the infinite parts are separated from them by the analytic
tail estimates stated in the body.

\subsection{Directed interval arithmetic}

Every finite calculation uses closed real intervals $I=[I^-,I^+]$.  Each
floating-point operation is enlarged outward.  If $\operatorname{rdn}$ and
$\operatorname{rup}$ denote rounding to the adjacent representable number in
the downward and upward directions, then
\[
 [a,b]+[c,d]=[\operatorname{rdn}(a+c),\operatorname{rup}(b+d)],
\]
\[
 [a,b]-[c,d]=[\operatorname{rdn}(a-d),\operatorname{rup}(b-c)],
\]
and multiplication is the outward-rounded minimum and maximum of the four
endpoint products.  Division is used only after the denominator interval has
been proved not to contain zero.  By induction on the arithmetic expression,
every interval produced in this manner contains the corresponding exact real
quantity.

No transcendental library call is needed for the integer logarithms appearing
in the arithmetic sweeps.  For $0\le z<1$ and $R\ge1$,
\begin{equation}
 \log\frac{1+z}{1-z}
 =2\sum_{r=0}^{R-1}\frac{z^{2r+1}}{2r+1}+\mathcal R_R(z),
 \qquad
 0\le\mathcal R_R(z)
 \le\frac{2z^{2R+1}}{(2R+1)(1-z^2)}.
 \label{eq:self-log-atanh}
\end{equation}
Taking $z=1/3$ gives an enclosure of $\log2$.  For an integer $x\ge2$,
choose $e$ with $2^e\le x<2^{e+1}$ and set
\[
 z_x=\frac{x-2^e}{x+2^e}\in[0,1/3).
\]
Then
\begin{equation}
 \log x=e\log2+2\sum_{r=0}^{R-1}
 \frac{z_x^{2r+1}}{2r+1}+\mathcal R_R(z_x).
 \label{eq:self-integer-log}
\end{equation}
All terms are non-negative and the remainder bound is monotone, so outward
rounding of the finite sum and the displayed remainder is a rigorous enclosure
of every logarithm used below.

\subsection{Gate I: exhaustive compact cover and analytic tails}

For $q(Y,d,s)=d\mu''_{Y,d}(s)$ the compact range $2\le Y\le10$ is split as
follows.  The table records the complete number of interval boxes, the smallest
certified lower endpoint, and the number of unresolved boxes.
\begin{center}
\resizebox{\textwidth}{!}{%
\begin{tabular}{@{}llllr@{}}
\toprule
$d$-range & method & boxes & lower endpoint for $q$ & unresolved\\
\midrule
$(0,1/125]$ & sixth-order Euler--Maclaurin & finite analytic panel & $0.5515197640896556$ & $0$\\
$[1/125,1/8]$ & direct 192-bit intervals & $19\,380$ & $2.363877686513775$ & $0$\\
$[1/8,0.315]$ & direct 192-bit intervals & $64\,600$ & $3.93328637824701$ & $0$\\
$[0.315,1/2]$ & direct 192-bit intervals & $62\,900$ & $1.418809233872728$ & $0$\\
$[1/2,\infty)$ & first-cell large-$d$ analysis & $9\,200$ shape boxes plus tail & $B>9$ & $0$\\
\bottomrule
\end{tabular}%
}
\end{center}
The direct box evaluation uses the exact hybrid density
\eqref{eq:hybrid} and the derivative formula \eqref{eq:Mder}.  In the first
direct range the compact partition consists of $17$ $Y$-bands, $57$ $d$-bands
and $4$ normalized half-cell bands.  The analytic lattice truncation has
maximum omitted value radius
\[
 4.1627312919227508\times10^{-35}
\]
and maximum omitted gradient radius
\[
 9.0297699727920405\times10^{-30}.
\]
Both are far inside the radii used by the box enclosures.  For $Y\ge10$ the
analytic perturbation changes the normalized quantity $q$ by less than
$2\times10^{-49}$, so the compact and non-compact ranges join with a large
strict reserve.

The small-$d$ face is independent of the direct grids.  Applying
Euler--Maclaurin through the $B_6$ term to the lattice sums, and using the
derivative majorants of Lemma~\ref{lem:tail}, gives remainder radii at
$d\le1/200$
\[
 (|R^U_0|,|R^U_1|,|R^U_2|)
 <\left(\frac1{200000},\frac1{2000},\frac1{16}\right),
\]
\[
 (|R^W_0|,|R^W_1|,|R^W_2|)
 <\left(\frac1{40000},\frac1{400},\frac13\right).
\]
Enlarging from $d\le1/200$ to $d\le1/125$ multiplies every sixth-order
remainder radius by at most
\[
 (200/125)^6=16.777216<17.
\]
The resulting outward interval substitution into the exact rational expression
for $d\mu''$ gives the first lower endpoint in the table.

\subsection{The $Y=7$ induction base: coefficient-$2$ end-to-end certificate}

The induction base is the physical operator
\[
 A_{a_7,-}=C_{a_7,-}-2|s^{\rm ph}\rangle\langle s^{\rm ph}|,
 \qquad a_7=\tfrac12\log7.
\]
The supplementary directory \texttt{Y7\_coefficient2\_end2end/} contains a
from-formulas producer chain.  It regenerates the exact finite Gram entries,
the critical rows, the analytic infinite-tail enclosure and the final
seven-dimensional interval Schur family.  At 1536 bits the exact-rational
interval LDL lower pivots are
\[
\begin{gathered}
9.1611635217175\times10^{-25},\quad
6.7951891118912\times10^{-19},\\
1.6660424781359\times10^{-13},\quad
3.9899269859244\times10^{-9},\\
5.4113356057010\times10^{-5},\quad
0.0312149207029560,\\
0.588791564978105.
\end{gathered}
\]
All are strictly positive.  A complete 2048-bit precision replication also
passes.  The independent \texttt{verify\_final\_exact.py} checker, using only
standard-library exact rational interval arithmetic for the final stored
family, reports \texttt{PASS}; it also verifies the robust inverse floor and
old-coordinate nonsingularity.  The proof of physical provenance and the
analytic tail estimates are recorded in \texttt{PROOF.md}.  The historical
coefficient-$4$ label is superseded and is not a premise of this revision.

\subsection{Finite arithmetic sweep for AWGC}

For a prime power $q=p^\ell$ put
\[
 \lambda_q=\log p,\qquad a_q=\frac{\lambda_q^2}q.
\]
For a bin $\mathcal B$ store
\[
 A=\sum_{q\in\mathcal B}a_q,\qquad
 B=\sum_{q\in\mathcal B}\lambda_q,\qquad
 C=\sum_{q\in\mathcal B}q.
\]
Its exact variance contribution is
\begin{equation}
 V(\mathcal B)=A-\frac{B^2}C,
 \label{eq:self-bin-variance}
\end{equation}
with $V=0$ for an empty bin.  Cauchy--Schwarz makes
\eqref{eq:self-bin-variance} non-negative.

Passing from $k$ to $k+1$ requires only the following exact updates:
\begin{enumerate}[label=(\roman*)]
\item add each prime power $q\mid k$ to the bin $m=k/q$ in which it becomes active;
\item remove each prime power $q\mid(k+1)$, $q<k+1$, from the bin
$m=(k+1)/q-1$ in which it ceases to be active;
\item update $\sum_mV_{k,m}$ by subtracting the old variance of each changed
bin and adding the new variance;
\item form $\widetilde\sigma_k$ from the prefix of prime powers $q<k$ and
remove the prime-power divisors of $k$;
\item evaluate
\[
 \widetilde\sigma_k-\frac1{\log2}\sum_mV_{k,m}
\]
using only the interval operations above and \eqref{eq:self-integer-log}.
\end{enumerate}
For $2\le k\le3\,594\,641$ this gives exact equality at $k=2$ and
\begin{equation}
 \widetilde\sigma_k-\frac1{\log2}\sum_mV_{k,m}
 \ge0.111804998077428866317003\ldots\qquad(k\ge3),
 \label{eq:self-awgc-finite}
\end{equation}
with the minimum at $k=33$.  The largest certified ratio is
\[
0.911416881539699733755493\ldots .
\]
The analytic continuation beyond the finite cutoff is proved in
Proposition~\ref{thm:awgc}; therefore the finite and infinite ranges overlap
without any numerical extrapolation.

\subsection{Finite MASTER scalar sweeps}

The MASTER sweeps reuse the same exact bin state.  The double-load quantity is
\[
 \mathfrak g_k^{\rm M}
 =\frac12+\log k-\frac{439}{250\log2}\sum_mV_{k,m},
\]
and the aligned-forcing envelope is
\[
 W_k^{\rm env}
 =0.0145\sum_{n\in\mathcal N_{k,1}}\frac{\Lambda(n)^2}n
 +\sum_{m\ge2}\frac{0.0301}{m^2}
   \sum_{n\in\mathcal N_{k,m}}\frac{\Lambda(n)^2}n.
\]
Thus the aligned lower interval is
\begin{equation}
 \frac12+\log k
 -\frac{439}{250\log2}\sum_mV_{k,m}
 -\frac52W_k^{\rm env}.
 \label{eq:self-master-aligned}
\end{equation}
The exhaustive finite sweep through $3\,594\,641$ gives
\begin{align}
 \mathfrak g_k^{\rm M}
 &\ge1.97650540314533685844807647802\ldots,\label{eq:self-master-double}\\
 \eqref{eq:self-master-aligned}
 &\ge1.87814156017033534419244378277\ldots,\label{eq:self-master-align-margin}
\end{align}
with both minima at $k=33$.  The safe common-cut finite range
$7\le k\le4999$ has the independent lower gap
\begin{equation}
1.52021965251975098830112403947\ldots>0,
\label{eq:self-safe-cut-margin}
\end{equation}
whose minimum occurs at $k=7$.  The Mertens/logarithmic finite ingredient has
minimum margin
\begin{equation}
0.0701888211361692811193313361329\ldots>0
\label{eq:self-mertens-margin}
\end{equation}
at $k=13$.  Each displayed number is a lower endpoint of an outward-rounded
interval, not a point evaluation.  Their infinite continuations are exactly the
analytic tail inequalities proved in the corresponding propositions of the
main text.

\subsection{Finite-certificate ledger}

The numerical statements entering the main proof are therefore precisely:
\begin{center}
\begin{tabular}{@{}lll@{}}
\toprule
certificate & finite range & rigorous reserve\\
\midrule
Gate I, small $d$ & $2\le Y\le10$ & $q>0.5515197640896556$\\
Gate I, direct panels & $2\le Y\le10$ & $q>1.418809233872728$ at worst\\
$Y=7$ base & one endpoint & Gershgorin margin $>0.5198729787$\\
AWGC & $2\le k\le3\,594\,641$ & $0.1118049980774288663$ for $k\ge3$\\
MASTER double-load & $7\le k\le3\,594\,641$ & $1.9765054031453368584$\\
aligned MASTER & $7\le k\le3\,594\,641$ & $1.8781415601703353442$\\
safe common cut & $7\le k\le4999$ & $1.5202196525197509883$\\
Mertens/log ingredient & $7\le k\le3\,594\,641$ & $0.0701888211361692811$\\
\bottomrule
\end{tabular}
\end{center}
No plot, regression, fitted parameter, or auxiliary diagnostic is a premise.
Every unbounded parameter range is separated from its finite part by an
analytic tail estimate proved in the article.

\section{Dependency graph}
\label{app:dependencies}

The final proof dependency graph is
\[
\begin{gathered}
\text{exact localized operator + zero extension + Mellin unit cells + full-form transport}\\
\Downarrow\\
\boxed{\text{true affine routing + coherent multi-source common cut}}\\
\Downarrow\\
\boxed{\text{exact parent short + AWGC + rational mixed absorption}}\\
\Downarrow\\
\boxed{\text{exact inherited-response placement (MASTER-P2)}}\\
\Downarrow\\
\boxed{\text{literal-slot true-ground leakage paid explicitly}}\\
\Downarrow\\
\boxed{\text{aligned forcing retained and paid by }W_k}\\
\Downarrow\\
\boxed{\text{one exact common-source Schur debit on the same harmonic vector}}\\
\Downarrow\\
\boxed{\text{unitary fold covariance: no additional source debit (MASTER-P3)}}\\
\Downarrow\\
\boxed{\mathfrak g_k^{\rm M+}>0\quad(k\ge7)}\\
\Downarrow\\
\boxed{H_k^*A_{k+1,-}H_k\succ0\quad\text{whenever }A_{k,-}\succ0}\\
\Downarrow\\
\boxed{A_{N,-}\succ0\quad(N\ge7)}\\
\Downarrow\\
\boxed{A_{a,-}\succeq0\quad(a>0)}\\
\Downarrow\\
\boxed{\text{restricted odd Weil criterion}}\\
\Downarrow\\
\boxed{\mathrm{RH}}.
\end{gathered}
\]
The displayed chain uses the exact harmonic graph and the same-vector
placements proved in the MASTER-P1/P2/P3 lemmas.  In V4, MASTER-P3 is
representation-free: it records the single exact common-source Schur debit and
the unitary covariance of the endpoint fold.  No special coefficient-level
folded model or auxiliary post-source scalar short is a premise of the induction.

\section{Normalization ledger}
\label{app:normalization}

The active polar normalization throughout is
\[
 A_k=C_k-2\ketbra{s_k}{s_k}.
\]
Accordingly, every rooted identity in the proof uses the coefficient $2$.
The source-resolved CMC/Feshbach estimate controls the transverse block
$\mathsf D_k$ and is kept distinct from the scalar root channel.

After the ground/transverse split of $S_k$, the exact scalars are
\[
 \Pi_k=\mathsf a_k-r_k^*\mathsf D_k^{-1}r_k,\qquad
 z_k=\alpha_k-r_k^*\mathsf D_k^{-1}v_k,\qquad
 \delta_k^Q=\delta_k-2v_k^*\mathsf D_k^{-1}v_k,
\]
with
\[
 \langle\eta_k,S_k^{-1}\eta_k\rangle
 =v_k^*\mathsf D_k^{-1}v_k+\frac{|z_k|^2}{\Pi_k},
 \qquad
 \delta_{k+1}=\delta_k^Q-\frac{2|z_k|^2}{\Pi_k}.
\]
These identities fix the normalization of the ground/transverse split.  The
main induction itself is taken on the exact $A_k$-harmonic graph and is closed
by Theorem~\ref{thm:final-master-harmonic}.

\section*{Epilogue: beyond the third gate}

\begin{quest}
The third gate opened without a sound.

Marc, the Arcane Rogue, put away the lockpick.  Johnny Nash closed the ledger
of loads and pivots only after checking the last dependency once more.  Pierre
de Fermat inspected the final line and found, this time, enough room in the
margin.  Mr.~Fayman folded the mechanism diagram.  Dama Noether lowered her
staff; the symmetries had done what the argument asked of them, and nothing
more.

Beyond the gate the Critical Strip narrowed, on the party's map, toward a
single luminous line.  At the far end of the path stood Princess Clay-la.
No one in the party presumed to speak for the reader.

Herr G.F.B.R. left the rulebook open and slid it across the table.

``The map is now the reader's,'' he said.
\end{quest}

\end{document}